\documentclass[11pt, letterpaper]{amsart}
\usepackage{amsfonts,amsmath, amsthm, amssymb, latexsym, epsfig, mathtools, xcolor, array}
\usepackage[english]{babel}
\usepackage[utf8]{inputenc}
\usepackage[all]{xy}
\usepackage{xspace}
\usepackage{comment}
\usepackage{setspace}
\usepackage{enumerate, enumitem}
\usepackage{stmaryrd}
\usepackage{tikz}
\usepackage{tikz-cd}
\usepackage[mathscr]{eucal}
\usepackage{hyperref}
	\newcommand{\Ext}{\ensuremath{\operatorname{Ext}}}
	
	\newcommand{\multialg}[1]{\mathcal{M}(#1)\xspace}

	\newcommand{\Ad}{\ensuremath{\operatorname{Ad}}\xspace}
	\renewcommand{\min}{\mathrm{min}}

	\newcommand{\Aut}{\mathrm{Aut}}
	
    \newcommand{\Inn}{\mathrm{Inn}}
	
	\newcommand{\Aff}{\mathrm{Aff}}
    \DeclareMathOperator{\K}{K}
    \DeclareMathOperator{\T}{T}
    \DeclareMathOperator{\KT}{Ell_s}
    \DeclareMathOperator{\KTu}{KT_u}
    \DeclareMathOperator{\TKT}{\underline{K}T_s}
    \DeclareMathOperator{\TKTu}{\underline{K}T_u}
	\newcommand{\Ka}{\K^{\mathrm{alg}}_1}
    \newcommand{\HKa}{\overline{\K}^{\mathrm{alg}}_1}
	
	\newcommand{\Kd}{\K^{\Delta}_1}
    \newcommand{\HKd}{\overline{\K}^{\Delta}_1}
	\newcommand{\Thd}{\mathrm{Th}^{\Delta}}
    \newcommand{\oimrho}[1]{\overline{\rho_{#1}(\K_0(#1))}}
    \newcommand{\imrho}[1]{\rho_{#1}(\K_0(#1))}

	\newcommand{\id}{\mathrm{id}}
        \newcommand{\cs}{\mathrm{C}^\ast}
        \newcommand{\Det}{\mathrm{Det}}

        \newcommand{\hash}{\mathrel{\#}}

        \newcommand{\dd}[1]{\tfrac{\mathrm d}{\mathrm d#1}}
        \newcommand*{\suchthat}{\;\ifnum\currentgrouptype=16 \middle\fi|\;}
	
\newcommand{\Ped}{\mathrm{Ped}}

\newcommand{\Tr}{\mathrm{Tr}}

\newcommand{\Th}{\mathrm{Th}}
\newcommand{\Thod}{\overline{\mathrm{Th}}^{\Delta}}

    \newcommand{\isom}{\xrightarrow{\cong}}

    \theoremstyle{plain}
	\newtheorem{thma}{Theorem}
	
	\newtheorem{introcorollary}[thma]{Corollary}
	
	\theoremstyle{definition}

	\theoremstyle{plain}
	\newtheorem{thm}{Theorem}[section]
	\newtheorem{lemma}[thm]{Lemma}
	\newtheorem{theorem}[thm]{Theorem}
	\newtheorem{proposition}[thm]{Proposition}
	\newtheorem{corollary}[thm]{Corollary}

	\theoremstyle{definition}
	\newtheorem{definition}[thm]{Definition}
	
	\newtheorem{remark}[thm]{Remark}
	
	\newtheorem{example}[thm]{Example}

	\numberwithin{equation}{section}
	\numberwithin{figure}{section}

\begin{document}
	\title[$\mathrm{K}_1^\Delta$ for $\mathrm{C}^\ast$-algebras]{$\boldsymbol{\mathrm{K}_1^\Delta}$ for $\boldsymbol{\mathrm{C}^\ast}$-algebras}
    \author{James Gabe} 
    \author{Alistair Miller}

\thanks{J.~G.~was supported by the Carlsberg Foundation Grant CF24-1144, and A.~M.~was supported by the Research Foundation Flanders (FWO) through Project 1212126N. Both authors were supported by the Independent Research Fund Denmark Grant 1054-00094B}

\begin{abstract}
We introduce a topological abelian group $\mathrm{K}_1^\Delta(A)$ for every $\mathrm{C}^\ast$-algebra $A$, which is functorial in $A$ and invariant under stabilisations. $\mathrm{K}_1^\Delta$ can distinguish group actions indistinguishable by equivariant (K)K-theory and traces --- we illustrate this on the irrational rotation algebra by exhibiting strongly outer actions of $\mathbb Z$ which, via $\mathrm{K}_1^\Delta$, are not (stably) cocycle conjugate. These actions are moreover trivial in $\mathrm{KK}^{\mathbb Z}$ and on the Hausdorffised unitary algebraic $\mathrm{K}_1$ group $\overline{\mathrm{K}}_1^{\mathrm{alg}}(A)$.

The Hausdorffisation $\overline{\mathrm{K}}_1^{\Delta}(A)$ of $\mathrm{K}_1^\Delta(A)$ forms the final part of the total stable invariant $\mathrm{\underline{K}T}_{\mathrm{s}}(A)$, which we also introduce. We prove that $\mathrm{\underline{K}T}_{\mathrm{s}}$ classifies $*$-homomorphisms up to approximate unitary equivalence between certain stable $\mathrm{C}^\ast$-algebras which contain a non-zero projection.

$\mathrm{K}_1^\Delta(A)$ is new even in the unital setting, but under further assumptions on $A$ recovers the unitary algebraic $\mathrm{K}_1$ group $\mathrm{K}_1^{\mathrm{alg}}(A)$, whereas the Hausdorffisations $\overline{\mathrm{K}}_1^{\Delta}(A)$ and $\overline{\mathrm{K}}_1^{\mathrm{alg}}(A)$ agree for every unital $A$. 
\end{abstract}

\maketitle

\tableofcontents

\section{Introduction}
The unital classification theorem for $\cs$-algebras (\cite{GLN2}, \cite{EGLN}, \cite{CGSTW}) sta\-tes that unital simple separable nuclear $\cs$-algebras $A$ which are $\mathcal Z$-stable and satisfy the UCT are classified up to isomorphism by the invariant $\KTu$ containing K-theory, traces and the pairing map:
\begin{equation}\label{Elliott invariant}
\KTu(A) := (\K_0(A),\K_1(A),[1_A]_0,\T_1(A), \K_0(A) \xrightarrow{\rho_A} \Aff \T_1(A)).
\end{equation}
However, the classification of $*$-homomorphisms (up to various weak notions of unitary equivalence) between such $\cs$-algebras -- and in particular automorphisms -- requires further invariants. To classify $*$-homomorphisms up to approximate unitary equivalence \cite{CGSTW}, these extra invariants are the K-theory with coefficients groups $\K_*(A;\mathbb Z / n \mathbb Z)$, the Hausdorffised unitary algebraic $\K_1$ group $\HKa(A)$, and a system of natural homomorphisms on top of the pairing map $\rho_A$; altogether these form the \emph{total unital invariant} $\TKTu$. 

For non-unital $\cs$-algebras, the invariants above require modifications. All of the K-theory groups themselves, including those with coefficients, are well-suited to non-unital $\cs$-algebras. However, there is no unit class $[1_A]_0$ in $\K_0(A)$, and the tracial state space $\T_1(A)$ does not necessarily capture all the relevant tracial information; indeed it vanishes for stable $\cs$-algebras. The stable Elliott invariant
\begin{equation}\label{eq:KTs}
\KT(A) := (\K_0(A),\K_1(A),\T_{\mathbb R}(A),\K_0(A) \xrightarrow{\rho_A} \T_{\mathbb R}(A)^\ast)
\end{equation}
incorporates the space $\T_{\mathbb R}(A)$ \cite{GMtraces} of self-adjoint traces defined on the Pedersen ideal $\Ped(A)$ as the tracial invariant (by duality this is equivalent to the data of its topological dual space $\T_{\mathbb R}(A)^\ast$). Under mild assumptions (e.g. if $A$ is {simple} or commutative \cite[Corollary 6.5]{GMtraces}), the dual space $\T_{\mathbb R}(A)^\ast$ is isomorphic to the space $\Aff_0 \T_+(A)$ of real affine functions vanishing at $0$ on the space of lower semicontinuous densely defined positive traces, and for unital $A$, there is a natural isomorphism $\T_{\mathbb R}(A)^\ast \cong \Aff \T_1(A)$. 

The Hausdorffised unitary algebraic $\K_1$ group $\HKa(A)$ has two drawbacks: for a stable $\cs$-algebra $A$ it is simply the $\K_1$ group $\K_1(A)$, and it has no hope of distinguishing the asymptotic unitary equivalence classes of $*$-homomorphisms within the same approximate unitary equivalence class. 

The (unHausdorffised) unitary algebraic $\K_1$ group $\Ka(A)$ is simply the abelianisation of the infinite unitary group $\mathcal U_\infty(A) = \varinjlim_n \mathcal U_n(A)$. This is a semitopological abelian group and can better distinguish asymptotic unitary equivalence classes than $\HKa$, but retains the first drawback for stable $\cs$-algebras, and is poorly-behaved for general $\cs$-algebras.

In this article we introduce the invariant $\Kd$, which assigns a topological abelian group $\Kd(A)$ to each $\cs$-algebra $A$, and is functorial with respect to $*$-homomorphisms. The group $\Kd(A)$ is generated by unitaries in $\mathcal U_\infty(\Ped(A))$ modulo the relation that two unitaries are equivalent if they can be connected by a unitary path with trivial de la Harpe--Skandalis determinant. 
We use a modified version of the de la Harpe--Skandalis determinant \cite{dlHS-det} better suited to handling unbounded traces, see Section \ref{s:Det}.

\begin{thma}[Corollary \ref{c:K1Dfullher} and Propositions \ref{p:asu}, \ref{p:exactseq} and \ref{p:pure} respectively]
The functor $\Kd$ satisfies the following properties:
\begin{enumerate}
\item Stable invariance: for any $\cs$-algebra $A$ the top-corner inclusion $A \hookrightarrow A \otimes \mathcal K$ induces an isomorphism $\Kd(A) \isom \Kd(A \otimes \mathcal K)$.
\item Asymptotic unitary equivalence invariance: if $\phi, \psi \colon A \to B$ are $*$-homomorphisms with $A$ separable which are asymptotically unitarily equivalent, then $\Kd(\phi) = \Kd(\psi)$.
\item For every $\cs$-algebra $A$, the group $\Kd(A)$ fits into a natural exact sequence
\begin{equation}\label{eq:exactseqintro}
\K_0(A) \xrightarrow{\rho_A} \T_{\mathbb R}(A)^\ast \xrightarrow{\Thd_A} \Kd(A) \to \K_1(A) \to 0.
\end{equation}
\item For each unital separable simple exact finite pure $\cs$-algebra $A$, there is a natural isomorphism $\Kd(A) \cong \Ka(A)$.
\end{enumerate}
\end{thma}

The topological abelian group $\Kd(A)$ is generally not Hausdorff. We also study its Hausdorffisation $\HKd(A)$, which is a Hausdorff topological abelian group. This group turns out to agree with $\HKa(A)$ for unital $\cs$-algebras but retains many of the same benefits as $\Kd(A)$ such as stability. It is also an approximate unitary equivalence invariant, and is therefore suitable for Elliott intertwining arguments in classification theory.

\begin{thma}[Proposition \ref{p:HK1Dfullher}, Corollaries \ref{c:approx} and \ref{c:HK1Dexactseq} and Proposition \ref{p:HK1Dalg} respectively]
The functor $\HKd$ satisfies the following properties:
\begin{enumerate}
\item Stable invariance: for any $\cs$-algebra $A$ the top-corner inclusion $A \hookrightarrow A \otimes \mathcal K$ induces an isomorphism $\HKd(A) \isom \HKd(A \otimes \mathcal K)$.
\item Approximate unitary equivalence invariance: if $\phi, \psi \colon A \to B$ are $*$-homomorphisms which are approximately unitarily equivalent, then $\HKd(\phi) = \HKd(\psi)$.
\item For every $\cs$-algebra $A$, the group $\HKd(A)$ fits into a natural exact sequence
\begin{equation}\label{eq:exactseqintro2}
0 \to \frac{\T_{\mathbb R}(A)^\ast}{\, \, \overline{ \imrho{A} }\, \,} \xrightarrow{\Thod_A} \HKd(A) \to \K_1(A) \to 0.
\end{equation}
\item For each unital $\cs$-algebra $A$, there is a natural isomorphism $\HKd(A)$ $ \cong $ $\HKa(A)$.
\end{enumerate}
\end{thma}
In fact, we establish the isomorphism $\HKd(A) \cong \HKa(A)$ for a larger class of $\cs$-algebras which we call \textit{tracially unital} $\cs$-algebras; these algebras behave like unital $\cs$-algebras tracially and include all separable simple $\cs$-algebras with continuous scale such as the Razak--Jacelon algebra $\mathcal W$ \cite{Razak, Jacelon}. 

The functor $\HKd$ forms the final piece of the \emph{total stable invariant} $\TKT$, along with K-theory, K-theory with coefficients, the space $\T_{\mathbb R}(A)$ of self-adjoint traces, and a system of natural homomorphisms as in the total unital invariant (see Definition \ref{d:total} for a complete description of $\TKT$). The aim of the total stable invariant is to classify $*$-homomorphisms between stable $\cs$-algebras up to approximate unitary equivalence. In the setting where the domain has a full projection, we obtain the following classification theorem through the application of \cite{CGSTW}. $\TKT$ will similarly be used in \cite{CGSTW2} to get more general classification results of $\ast$-homomorphisms between more general $\cs$-algebras without assuming the existence of projections.

\begin{thma}\label{t:embeddingsintro}
Let $A$ and $B$ be stable separable nuclear $\cs$-algebras such that $A$ satisfies the UCT and contains a full projection, and $B$ is simple and $\mathcal Z$-stable. Then embeddings $A \hookrightarrow B$ which map traces on $B$ to faithful traces on $A$ correspond up to approximate unitary equivalence bijectively to morphisms $\TKT(A) \to \TKT(B)$ of the total stable invariant.
\end{thma}

As a corollary we recover the following well-known result.

\begin{introcorollary}\label{c:classintro}
    Let $A$ and $B$ be stable simple separable nuclear $\mathcal Z$-stable $\cs$-algebras which satisfy the UCT and contain a non-zero projection. Then $A \cong B$ if and only if $\KT(A) \cong \KT(B)$. 
\end{introcorollary}

While Corollary \ref{c:classintro} also follows straightforwardly from the unital classification theorem \cite{GLN2}, \cite{EGLN}, \cite{CGSTW} and Brown's stable isomorphism theorem \cite{Brown77}, the strategy of using Theorem \ref{t:embeddingsintro} further yields the following corollary.

\begin{introcorollary}
    Let $A$ be a stable simple separable nuclear $\mathcal Z$-stable $\cs$-algebra which satisfies the UCT and contains a non-zero projection. Then the canonical map $\Aut(A) \to \Aut(\TKT(A))$ induces a group isomorphism 
    \[\Aut(A)/\overline{\Inn(A)} \isom \Aut(\TKT(A)). \]
\end{introcorollary}
We illustrate this result by computing $\Aut(A)/\overline{\Inn(A)}$ for certain stable $\cs$-algebras $A$, see Examples \ref{e:Aut1} and \ref{e:Aut2}.

We predict that in the future, $\Kd(A)$ will play a prominent role in classifying group actions on $\mathrm{C}^\ast$-algebras. In the purely infinite setting, certain outer actions of a group $G$ may be classified up to cocycle conjugacy using K-theoretic invariants such as $\mathrm{KK}^G$ \cite{IzumiMatuiPolyZII, DynamicalKP}. Classification of group actions up to cocycle conjugacy is even more challenging in the tracial setting, yet there is remarkable progress for various special classes \cite{MatuiSatoStronglyOuter, MatuiSatoStronglyOuterII} of actions. One should expect future classification results to involve K-theory, traces and $\mathrm{KK}^G$. The following example suggests that $\Kd$ should be added to this arsenal, as it exhibits elementary strongly outer actions of $\mathbb Z$ on irrational rotation algebras that K-theory, traces and $\mathrm{KK}^G$ cannot distinguish, but which $\Kd$ can. Here, given an action $\alpha \colon G \to \Aut(A)$, we write $\Kd(A,\alpha)$ for $\Kd(A)$ along with the $G$-module structure it inherits, and similarly for $\T_{\mathbb R}(A,\alpha)^\ast$.

\begin{thma}
There exist strongly outer actions $\alpha, \beta \colon \mathbb Z \to \Aut(A)$ on an irrational rotation algebra $A$ such that there exists a $\mathrm{KK}^{\mathbb Z}$-equivalence $x\in \mathrm{KK}^{\mathbb Z}((A, \alpha), (A, \beta))$ with $\K_0(x)([1_A]_0) = [1_A]_0$ and $\rho_A \circ \K_0(x) = \rho_A$, but also such that $(\T_{\mathbb R}(A,\alpha)^\ast \xrightarrow{\Thd_A} \Kd(A,\alpha)) \not \cong (\T_{\mathbb R}(A,\beta)^\ast \xrightarrow{\Thd_A} \Kd(A,\beta))$. Consequently, $(A,\alpha)$ and $(A,\beta)$ are not stably cocycle conjugate.
\end{thma}

In the above example, the canonical map $\HKa(A) \to \K_1(A)$ is an isomorphism of $G$-modules with respect to both $\alpha$ and $\beta$, hence $\HKa$ is also not helpful for distinguishing $(A,\alpha)$ and $(A,\beta)$.

We also predict that $\Kd$ will be useful for studying KMS states on $\cs$-algebras with a given flow, by considering the interaction between $\Kd$ of the (non-unital) crossed product and its induced trace, which is scaled by the dual action (see e.g. \cite[Theorem 3.2]{KishimotoKumjian}).

\textbf{Acknowledgements.}
A.~M.~would like to thank Sergio Gir\'on Pacheco for communicating Example \ref{e:Bunce Deddens} and G\'abor Szab\'o for a reference on strong outerness.

\textbf{Tool disclosure.} No LLMs were used in the production of this article.

\section{Preliminaries}

This article builds upon the space $\T_{\mathbb R}(A)$ of self-adjoint traces on a $\cs$-algebra $A$ introduced in \cite{GMtraces}; we recall here some of the basic facts needed to work with this invariant. Traces are defined on the Pedersen ideal $\Ped(A)$, which is the unique minimal dense two-sided ideal in $A$. We will make frequent use of \cite[Proposition 5.6.2]{Pedersen-book-automorphism} that for any finite set $X \subseteq \Ped(A)$, the hereditary $\cs$-subalgebra of $A$ generated by $X$ is contained in $\Ped(A)$, and its variation \cite[Proposition 2.2]{GMtraces} that for a finite set $\mathcal D$ of $\cs$-subalgebras of $\Ped(A)$, the hereditary $\cs$-subalgebra of $A$ generated by $D \in \mathcal D$ is contained in $\Ped(A)$. As a result, $\Ped(A)$ is the union of its $\cs$-subalgebras, and we equip it with the resulting inductive limit topological vector space structure. 

The space $\T_{\mathbb R}(A)$ consists of continuous self-adjoint tracial functionals on $\Ped(A)$, equipped with the weak$^\ast$ topology with respect to $\Ped(A)$. We write $\T_+(A)$ for the subset of positive traces. The dual space $\T_{\mathbb R}(A)^\ast$ is a quotient of $\Ped(A)_{\mathrm{sa}}$, and the quotient topology is given by the compact-open topology \cite[Theorem C]{GMtraces}. In this article, we write $\phi_\# \colon \T_{\mathbb R}(A)^\ast \to \T_{\mathbb R}(B)^\ast$ for the map induced by a $*$-homomorphism $\phi \colon A \to B$.

For $\tau \in \T_{\mathbb R}(A)$, we consider the (unnormalised) trace $\tau^{(n)} \in \T_{\mathbb R}(M_n(A))$ given at $a \in \Ped(M_n(A)) = M_n(\Ped(A))$ by 
\[ \tau^{(n)}(a) = \sum_{i=1}^n \tau(a_{i,i}). \]
For $a\in M_n(\Ped(A))_{\mathrm{sa}}$ we let $\widehat a \in \T_{\mathbb R}(A)^\ast$ denote the function given by
\[
\widehat a(\tau) = \tau^{(n)}(a) , \qquad \textrm{for all }\tau\in \T_{\mathbb R}(A).
\]
For a $\cs$-algebra $A$ we write $A^\dagger$ for the forced unitisation of $A$, and $\Ped(A)^\dagger$ for those $a \in A^\dagger$ with $s(a) - a \in \Ped(A)$ where $s\colon A^\dagger \to \mathbb C \subseteq A^\dagger$ is the canonical scalar $\ast$-homomorphism. For $n\in \mathbb N$ we let $\mathcal U_n(A)$ denote the group of unitaries $u\in M_n(A^\dagger)$ for which $1_{n}-u \in M_n(A)$. We write $\mathcal U_n(\Ped(A))$ for the unitaries $u \in \mathcal U_n(A)$ with $1_n - u \in M_n(\Ped(A))$. 
Also, we let $\mathcal M(A)$ denote the multiplier algebra.

\section{The de la Harpe--Skandalis determinant}\label{s:Det}

In this section we extend the de la Harpe--Skandalis determinant \cite{dlHS-det} to the setting of possibly unbounded self-adjoint traces on the Pedersen ideal, and document its basic properties. We provide a self-contained account, but follow the same procedure as for the classical de la Harpe--Skandalis determinant on bounded traces.

\begin{definition}[The de la Harpe--Skandalis determinant on specific unitary paths]
Let $A$ be a $\cs$-algebra and let $\mathbf u \colon [t_0,t_1] \to \mathcal U_n(\Ped(A))$ be a piecewise $C^1$ path  defined on a real interval $[t_0,t_1] \subseteq \mathbb R$ with values in $\mathcal U_n(D)$ for some $\cs$-subalgebra $D \subseteq \Ped(A)$. The \emph{de la Harpe--Skandalis determinant} $\Det_A(\mathbf u) \colon \T_{\mathbb R}(A) \to \mathbb C$ of $\mathbf u$ is given at $\tau\in \T_{\mathbb R}(A)$ by
\begin{equation}\label{eq:DetPiecewisesmooth}
    \mathrm{Det}_A(\mathbf u)(\tau) = \frac{1}{2\pi i} \int_{t_0}^{t_1} \tau^{(n)}(\mathbf u(t)^\ast \tfrac{\mathrm d}{\mathrm dt} \mathbf u(t)) \, \mathrm dt.
\end{equation}
\end{definition}

We extend the determinant to norm-continuous unitary paths $\mathbf u \colon [t_0,t_1] \to \mathcal U_n(\Ped(A))$. The following elementary lemmas explain how this is done, while simultaneously recording some basic facts about $\mathcal U_n(\Ped(A))$.

A priori, it is not clear that $\Det_A(\mathbf u) \colon \T_{\mathbb R}(A) \to \mathbb C$ is continuous and real valued, and hence an element in $\T_{\mathbb R}(A)^\ast$, but this follows easily from Lemma \ref{l:unitary} \eqref{l:unitary-4} and \eqref{l:unitary-5} and Lemma \ref{l:Det_lemma} below, since each $\widehat a_j \in \T_{\mathbb R}(A)^\ast$.

\begin{lemma}\label{l:unitary}
    Let $A$ be a $\cs$-algebra, $n\in \mathbb N$ and $t_0, t_1 \in \mathbb R$ with $t_0 < t_1$. 
    \begin{enumerate}
        \item\label{l:unitary-1} $\mathcal U_n(\Ped(A))$ is a normal subgroup of $\mathcal U_n(\mathcal M(A))$.

        \item\label{l:unitary-1.5} $\mathcal U_n(\Ped(A))$ is norm-dense in $\mathcal U_n(A)$.
        
        \item\label{l:unitary-2} For $a\in M_n(\Ped(A))_{\mathrm{sa}}$ one has $e^{2\pi i a}\in \mathcal U_n(\Ped(A))$.
        
        \item\label{l:unitary-3} If $u,v\in \mathcal U_n(\Ped(A))$ and $\| u - v \| < 2$ then $a:=\tfrac{1}{2\pi i} \log(u^\ast v) \in M_n(\Ped(A))$, where $\log$ is the branch of logarithm with domain $\mathbb C \setminus (-\infty , 0]$, satisfies $u e^{2\pi i a} = v$ and
        \[
        \| a\| = \tfrac{1}{2\pi} \arccos(1-\tfrac{1}{2} \| u - v\|^2)
        \]
        where $\arccos \colon [-1,1]\to [0,\pi]$. 
        \item\label{l:unitary-6} For a $\cs$-subalgebra $D \subseteq \Ped(A)$ and piecewise $C^1$ unitary paths $\mathbf u,\mathbf v \colon [t_0,t_1] \to \mathcal U_n(D)$, 
        \[\Det_A(\mathbf u \mathbf v) = \Det_A(\mathbf u) + \Det_A(\mathbf v).\]
        \item\label{l:unitary-7} For a $\cs$-subalgebra $D \subseteq \Ped(A)$ and a piecewise $C^1$ self-adjoint path $\mathbf a \colon [t_0,t_1] \to M_n(D)_{\mathrm{sa}}$, the unitary path $e^{2 \pi i \mathbf a}$ has determinant
        \[\Det_A(e^{2 \pi i \mathbf a}) = \widehat{\mathbf a(t_1)} - \widehat{\mathbf a(t_0)}. \]
        \item\label{l:unitary-4} For norm-continuous unitary paths $\mathbf u, \mathbf v \colon [t_0,t_1] \to \mathcal U_n(\Ped(A))$ with
        \[
        \mathbf u(t_0) = \mathbf v(t_0),\quad \mathbf u(t_1) = \mathbf v(t_1), \quad \textrm{and} \quad \sup_{t\in [t_0,t_1]}\| \mathbf u(t) - \mathbf v(t)\| < 2,
        \]
        then there is an endpoint-preserving norm-homotopy between $\mathbf u$ and $\mathbf v$ with values in $\mathcal U_n(\Ped(A))$. If moreover $\mathbf u$ and $\mathbf v$ take values in $\mathcal U_n(D)$ for some $\cs$-subalgebra $D \subseteq \Ped(A)$, then the homotopy may be chosen within $\mathcal U_n(D)$. 
        \item\label{l:unitary-5} If $\mathbf u \colon [t_0,t_1] \to \mathcal U_n(\Ped(A))$ 
        is a norm-continuous unitary path, then for any $\epsilon > 0$, there is a $\cs$-subalgebra $D \subseteq \Ped(A)$ and a piecewise smooth unitary path $\mathbf v \colon [t_0,t_1] \to \mathcal U_n(D)$ with
        \[
        \mathbf u(t_0) = \mathbf v(t_0),\quad \mathbf u(t_1) = \mathbf v(t_1), \quad \textrm{and} \quad \sup_{t\in [t_0,t_1]}\| \mathbf u(t) - \mathbf v(t)\| < \epsilon.
        \]
        Moreover, we may pick $\mathbf v$ such that there exist $m\in \mathbb N$ and self-adjoint elements $a_1,\dots,a_m \in M_n(D)_{\mathrm{sa}}$ such that $\mathbf u(t_0)^\ast \mathbf u(t_1) = e^{2\pi i a_1}\cdots e^{2\pi i a_m}$ and $\Det_A(\mathbf v) = \sum_{j=1}^m \widehat a_j$. 
    \end{enumerate}
\end{lemma}
\begin{proof}
    \eqref{l:unitary-1}: If $u,v\in \mathcal U_n(\Ped(A))$ then $uv -1_n = (u-1_n)v + (v-1_n)$ is in $M_n(\Ped(A))$ since $M_n(\Ped(A))$ is an ideal in $M_n(\mathcal M(A))$, and $u^\ast - 1_n = (u-1_n)^\ast \in M_n(\Ped(A))$, so $\mathcal U_n(\Ped(A))$ is a subgroup. If $w\in \mathcal U_n(\mathcal M(A))$ then $wuw^\ast - 1_n = w(u-1_n)w^\ast \in M_n(\Ped(A))$ (again, since $M_n(\Ped(A))$ is an ideal in $M_n(\mathcal M(A))$) and hence the subgroup is normal.

    \eqref{l:unitary-1.5}: Let $u\in \mathcal U_n(A)$ and $0<\epsilon <1/2$. Pick $x \in M_n(\Ped(A))$ such that $\| x - (u-1_n)\| < \epsilon$ and $\| x+1_n\| \leq 1$. Then 
    \[
    \delta := \| x^\ast x + x +x^\ast \| = \|(x+1_n)^\ast (x+1_n) - u^\ast u \| < 2 \epsilon < 1
    \]
    so $v := (x+1_n)|x+1_n|^{-1} \in \mathcal U_n(A^\dagger)$. Let $f\in C_0([-\delta, \delta]\setminus\{0\})$ be given by $f(t) = (t+1)^{-1/2} -1$. Then $|x+1_n|^{-1} -1_n = f(x^\ast x +x + x^\ast) \in C^\ast(x) \subseteq M_n(\Ped(A))$ and thus
    \[
        v-1_n = x|x+1_n|^{-1} + (|x+1_n|^{-1} -1_n) \in M_n(\Ped(A))
    \]
    so $v\in \mathcal U_n(\Ped(A))$. Moreover,
    \[
    \| v-u\| \leq \| (x+1_n)(|x+1_n|^{-1} -1_n)\| + \| x+1_n -u \| \leq \| f\| + \epsilon \leq \tfrac{2\epsilon}{1-2\epsilon} + \epsilon.
    \]
    
    \eqref{l:unitary-2}: Let $D \subseteq \Ped(A)$ be a $\cs$-subalgebra containing $a_{k,\ell} \in \Ped(A)$ for each $k, \ell = 1, \dots, n$. Then $a \in M_n(D)$ and thus $e^{2\pi i a} -1_n = f(a) \in M_n(D) \subseteq M_n(\Ped(A))$ by functional calculus where $f(t) = e^{2\pi i t}-1$ satisfies $f(0)= 0$.

    \eqref{l:unitary-3}: Since $u^\ast v - 1_n \in M_n(\Ped(A))$ by \eqref{l:unitary-1}, we let $D = \cs((u^\ast v-1_n)_{k,\ell} \mid k, \ell = 1, \dots, n)$, which is a $\cs$-subalgebra of $\Ped(A)$. Then $a \in M_n(D)$ (as opposed to $M_n(D^\dagger)$) since $u^\ast v -1_n \in M_n(D)$, and $\log(u^\ast v) = f(u^\ast v - 1_n)$ where $f(z) = \log(z+1)$ vanishes at $0$. Clearly $u e^{2\pi i a} =v$. The norm computation follows from functional calculus (and using that $\| u^\ast v - 1_n\|= \| u - v\|$). In fact, if $t\in (-\tfrac{1}{2},\tfrac{1}{2})$ and $z = e^{2\pi i t}$ then 
    \[
    |z-1|^2 = (\cos(2\pi t)-1)^2 + \sin^2(2\pi t) = 2(1-\cos(2\pi t))
    \]
    and hence
    \[
    |\tfrac{1}{2\pi i} \log(z) | = |t|  = \tfrac{1}{2\pi}\arccos(\cos(2\pi t)) = \tfrac{1}{2\pi}\arccos(1-\tfrac{1}{2}|z-1|^2).
    \]
    
    \eqref{l:unitary-6}: Let $\tau \in \T_{\mathbb R}(A)$ and note that $\tau^{(n)}$ is bounded on $M_n(D)$. For $t \in [t_0,t_1]$ we calculate
\begin{eqnarray*}
& \tau^{(n)}\left( (\mathbf u\mathbf v)(t)^\ast \dd t (\mathbf u \mathbf v)(t)\right) & = \tau^{(n)}\left( \mathbf v(t)^\ast \dd t \mathbf v(t) + \mathbf v(t)^\ast \mathbf u(t)^\ast  (\dd t \mathbf u(t))  \mathbf v(t) \right) \\
&& = \tau^{(n)} \left(  \mathbf v(t)^\ast \dd t \mathbf v(t) + \mathbf u(t)^\ast  \dd t \mathbf u(t)  \right),
\end{eqnarray*}
from which the conclusion is immediate.

    \eqref{l:unitary-7}: Let $\tau \in \T_{\mathbb R}(A)$ and note that $\tau^{(n)}$ is bounded on $M_n(D)$. For $t \in [t_0,t_1]$ we calculate
\begin{eqnarray*}
&  \tau^{(n)}\left(e^{-2 \pi i \mathbf a(t)} \dd t e^{2 \pi i \mathbf a(t)} \right) & = \sum_{k \geq 0,\ell \geq 0} \tau^{(n)} \left( \tfrac{(-2 \pi i \mathbf a(t))^k}{k!} \dd t \tfrac{(2 \pi i \mathbf a(t))^\ell}{\ell !}   \right)  \\
&& = \sum_{k \geq 0,\ell \geq 1} \tau^{(n)} \left( 2 \pi i \tfrac{(-2 \pi i \mathbf a(t))^k}{k!} \tfrac{(2 \pi i \mathbf a(t))^{\ell-1}}{(\ell-1) !}  \dd t \mathbf a(t) \right) \\
&& = 2 \pi i \tau^{(n)} \left( e^{-2 \pi i \mathbf a(t)} e^{2 \pi i \mathbf a(t)} \dd t \mathbf a(t) \right) \\
&& =  2 \pi i \dd t \tau^{(n)}( \mathbf a(t) ).
\end{eqnarray*}
The result then follows from the fundamental theorem of calculus.

    \eqref{l:unitary-4}: Set $\mathbf a = \tfrac{1}{2\pi i} \log(\mathbf u^* \mathbf v) \colon [t_0,t_1] \to M_n(\Ped(A))_{\mathrm{sa}}$. Then 
    \[(s,t) \mapsto \mathbf u(t) e^{2 \pi i s \mathbf a(t)} \colon [0,1] \times [t_0,t_1] \to \mathcal U_n(\Ped(A))\]
    is a endpoint preserving norm-homotopy from $\mathbf u$ to $\mathbf v$. It is contained within $\mathcal U_n(D)$ if both $\mathbf u$ and $\mathbf v$ are.
    
    \eqref{l:unitary-5}: Pick $t_0 = s_0 < s_1 < \dots < s_m = t_1$ such that $\sup_{t\in [s_{j-1}, s_j]}\| \mathbf u(s_{j-1}) - \mathbf u(t)\| < \tfrac{\epsilon}{2}$ for $j=1,\dots,m$. 
    Let 
    \[
    a_j  = \tfrac{1}{2\pi i} \log(\mathbf u(s_{j-1})^\ast \mathbf u(s_j)) \in M_n(\Ped(A))_{\mathrm{sa}}
    \]
    for $j=1,\dots, m$. 
    We obtain smooth paths $\mathbf v_j \colon [s_{j-1},s_j] \to \mathcal U_n(\Ped(A))$ defined by 
    \[
    \mathbf v_j(t) = \mathbf u(s_{j-1}) e^{\tfrac{t-s_{j-1}}{s_j-s_{j-1}} 2\pi i a_j}
    \]
    for $j=1, \dots, m$ and $t\in [s_{j-1},s_j]$. Note that $\mathbf v_j(s_j) = \mathbf v_{j+1}(s_j)$ for $j=1,\dots, m-1$. Hence we can concatenate the paths and obtain a piecewise smooth path $\mathbf v \colon [t_0,t_1] \to \mathcal U_n(\Ped(A))$ which satisfies $\mathbf v(s_j) = \mathbf u(s_j)$ for $j = 0, \dots, m$. Moreover, 
\begin{eqnarray*}
 & \sup_{t\in [t_0,t_1]} \| \mathbf v(t) - \mathbf u(t)\| & \leq \sup_{j=1,\dots, m} \sup_{t \in [s_{j-1}, s_j]} \| \mathbf v_j(t) -  \mathbf u(s_{j-1}) \| + \tfrac{\epsilon}{2}\\
 && = \sup_{j=1,\dots, m} \sup_{t \in [s_{j-1}, s_j]} \| e^{\tfrac{t-s_{j-1}}{s_j-s_{j-1}} 2\pi i a_j} - 1_n \| + \tfrac{\epsilon}{2} \\
 && < \epsilon.
\end{eqnarray*}
For the last inequality, notice that the function $z \mapsto e^{\tfrac{t-s_{j-1}}{s_j-s_{j-1}} 2\pi i z} - 1$ has absolute value bounded by $\tfrac{\epsilon}{2}$ on $[-\|a_j\|,\|a_j\|]$ by construction.

Let $D \subseteq \Ped(A)$ be a $\cs$-subalgebra which contains $(1_n - \mathbf u(t_0))_{k,\ell}$ and $(a_j)_{k,\ell}$ for each $k, \ell = 1, \dots, n$ and $j = 1,\dots, m$. Note that $\mathbf u(s_j) = \mathbf u(t_0) e^{2\pi i a_1} \cdots e^{2\pi i a_j}$, so that $\mathbf u(s_j) \in \mathcal U_n(D)$ for $j=0,\dots, m$. It follows that $\mathbf v ([t_0,t_1]) \subseteq \mathcal U_n(D)$.   

Finally, we have $\mathbf u(t_0)^\ast \mathbf u(t_1) = e^{2\pi i a_1}\cdots e^{2\pi i a_m}$ by construction and by \eqref{l:unitary-6} and \eqref{l:unitary-7}, we compute $\Det_A(\mathbf v) = \sum_{j=1}^m \Det_A(\mathbf v_j) = \sum_{j=1}^m \widehat a_j$.
\end{proof}

\begin{lemma}\label{l:Det_lemma}
    Let $A$ be a $\cs$-algebra, let $n \in \mathbb N$, let $t_0 < t_1$ be real numbers, and let $\mathbf u \colon [t_0,t_1] \to \mathcal U_n(\Ped(A))$ be a norm-continuous path. Then there is a $\cs$-subalgebra $D\subseteq \Ped(A)$ and a piecewise smooth path $\mathbf v \colon [t_0,t_1] \to \mathcal U_n(D)$ which is homotopic to $\mathbf u$ in an endpoint-preserving way. 
    
    Moreover, if $D_1,D_2\subseteq \Ped(A)$ are $\cs$-subalgebras and $\mathbf v_1, \mathbf v_2 \colon [t_0,t_1] \to \mathcal U_n(D_j)$ are piecewise $C^1$ paths which are norm-homotopic within $\mathcal U_n(\Ped(A))$ in an endpoint-preserving way, then $\Det_A(\mathbf v_1) = \Det_A(\mathbf v_2)$.
\end{lemma}
\begin{proof}
The first part follows immediately from Lemma \ref{l:unitary} \eqref{l:unitary-4} and \eqref{l:unitary-5}. 

For the second part, let $(\mathbf u_t)_{t \in [0,1]}$ be a norm-continuous endpoint-pres\-erv\-ing homotopy from $\mathbf v_1 = \mathbf u_{0}$ to $\mathbf v_2 = \mathbf u_{1}$. Pick $0 = s_0 < s_1 < \dots < s_m = 1$ such that $\| \mathbf u_{s_{j-1}} - \mathbf u_{s_{j}} \| < 2$ for $j = 1, \dots, m$. Using Lemma \ref{l:unitary} \eqref{l:unitary-5} find, for $j = 0, \dots, m$, $\cs$-subalgebras $D_j' \subseteq \Ped(A)$ and piecewise $C^1$ paths $\mathbf w_j \colon [t_0,t_1] \to \mathcal U_n(D_j')$ with $\mathbf w_1 = \mathbf v_1$, $\mathbf w_n = \mathbf v_2$, and $\mathbf w_j(t_0) = \mathbf v_1(t_0)$, $\mathbf w_j(t_1) = \mathbf v_1(t_1)$ and $\| \mathbf w_{j-1} - \mathbf w_{j} \| < 2$ for $j = 1, \dots, m$.  Let $D \subseteq \Ped(A)$ be a $\cs$-subalgebra containing $D_1$, $D_2$ and $D_j'$ for each $j$. For each $j$ the path $\tfrac{1}{2 \pi i} \log( \mathbf w_{j-1}^\ast \mathbf w_j)$ is $C^1$ and self-adjoint with endpoints 0 and so $\Det_A(\mathbf w_{j-1}^\ast \mathbf w_j) = 0$ by Lemma \ref{l:unitary} \eqref{l:unitary-7}. Thus by Lemma \ref{l:unitary} \eqref{l:unitary-6} we may conclude $\Det_A(\mathbf v_1) = \Det_A(\mathbf v_2)$.
\end{proof}

\begin{definition}[The de la Harpe--Skandalis determinant on norm-continuous unitary paths in the Pedersen ideal]
Let $A$ be a $\cs$-algebra, let $n \in \mathbb N$ and let $t_0 < t_1$ be real numbers, and let  $\mathbf u \colon [t_0,t_1] \to \mathcal U_n(\Ped(A))$ be a norm-continuous unitary path. The \emph{de la Harpe--Skandalis determinant} $\Det_A(\mathbf u) \in \T_{\mathbb R}(A)^\ast$ of $\mathbf u$ is given at $\tau \in \T_{\mathbb R}(A)$ by
\begin{equation}\label{eq:Det}
\Det_A(\mathbf u)(\tau) = \Det_A(\mathbf v)(\tau) =  \frac{1}{2\pi i} \int_{t_0}^{t_1} \tau^{(n)}(\mathbf v(t)^\ast \tfrac{\mathrm d}{\mathrm dt} \mathbf v(t)) \, \mathrm dt,
\end{equation}
where $\mathbf v \colon [t_0,t_1] \to \mathcal U_n(\Ped(A))$ is any piecewise $C^1$ unitary path which is norm-homotopic to $\mathbf u$ within $\mathcal U_n(\Ped(A))$ in an endpoint-preserving way and has values in $\mathcal U_n(D)$ for some $\cs$-subalgebra $D \subseteq \Ped(A)$.
\end{definition}
This is well-defined by Lemma \ref{l:Det_lemma}. The following basic properties of the de la Harpe--Skandalis determinant are immediate from Lemma \ref{l:unitary} and Lemma \ref{l:Det_lemma} after perhaps changing a given path into a path for which \eqref{eq:Det} is applicable. The details are left for the reader.

\begin{lemma}\label{l:basicdetproperties}
Let $A$ be a $\cs$-algebra, let $m,n \in \mathbb N$ and let $t_0 < t_1 < t_2$ be real numbers.
\begin{enumerate}
\item\label{l:basicdetproperties 1} If $\mathbf u \colon [t_0,t_1] \to \mathcal U_n(\Ped(A))$ and $\mathbf v \colon [t_0,t_1] \to \mathcal U_n(\Ped(A))$ are norm-continuous unitary paths which are endpoint-preserving norm-homo\-topic within $\mathcal U_n(\Ped(A))$, then \[\Det_A(\mathbf u) = \Det_A(\mathbf v).\]
\item\label{l:basicdetproperties 2} Let $\mathbf u \colon [t_0,t_1] \to \mathcal U_n(\Ped(A))$ and $\mathbf v \colon [t_1,t_2] \to \mathcal U_n(\Ped(A))$ be norm-continuous unitary paths with $\mathbf u(t_1)= \mathbf v(t_1)$. Then the concatenation $\mathbf u \hash \mathbf v \colon [t_0,t_2] \to \mathcal U_n(\Ped(A))$ has determinant \[\Det_A(\mathbf u \hash \mathbf v) = \Det_A(\mathbf u) + \Det_A(\mathbf v).\]
\item\label{l:basicdetproperties 3} For a constant path $\mathbf u \colon [t_0,t_1] \to \mathcal U_n(\Ped(A))$, \[\Det_A(\mathbf u) = 0.\]
\item\label{l:basicdetproperties 4} For any two norm-continuous paths $\mathbf u, \mathbf v \colon [t_0,t_1] \to \mathcal U_n(\Ped(A))$,  \[\Det_A(\mathbf u \mathbf v) = \Det_A(\mathbf u) + \Det_A(\mathbf v).\]
\item\label{l:basicdetproperties 5} For any two norm-continuous paths $\mathbf u\colon [t_0,t_1] \to \mathcal U_n(\Ped(A))$ and $\mathbf v \colon [t_0,t_1] \to \mathcal U_m(\Ped(A))$, the diagonal sum
\[
\mathbf u \oplus \mathbf v \colon [t_0,t_1] \to \mathcal U_{n+m}(\Ped(A))
\]
has determinant
\[
\Det_A(\mathbf u \oplus \mathbf v) = \Det_A(\mathbf u) + \Det_A(\mathbf v).
\]
\item\label{l:basicdetproperties 6} Let $\mathbf u \colon [t_0, t_1] \to \mathcal U_n(\Ped(A))$ be a norm-continuous path and let $\overline{\mathbf u}, \mathbf u^\ast \colon [t_0, t_1] \to \mathcal U_n(\Ped(A))$ be the paths given by 
\[
\overline{\mathbf u} (t) = \mathbf u(t_0 + t_1 -t), \qquad \mathbf u^\ast(t) = \mathbf u(t)^\ast, \qquad t\in [t_0,t_1].
\]
Then  
\[
\Det_A(\overline{\mathbf u}) = \Det_A(\mathbf u^\ast) = - \Det_A (\mathbf u)
\]
\item\label{l:basicdetproperties 7} Let $\mathbf u \colon [t_0,t_1] \to \mathcal U_n(\Ped(A))$ be a norm-continuous path, and let $\mathbf v \colon [0,1] \to \mathcal U_n(\Ped(A))$ be defined at $s \in [0,1]$ by 
\[\mathbf v(s) = \mathbf u(t_0 + s(t_1 - t_0)).\]
Then $\Det_A(\mathbf v) = \Det_A(\mathbf u)$.
\item\label{l:basicdetproperties 8} For any $a \in M_n(\Ped(A))_{\mathrm{sa}}$, the path $t \mapsto e^{2 \pi i t a} \colon [0,1] \to \mathcal U_n(\Ped(A))$ has determinant \[\Det_A(t \mapsto e^{2 \pi i t a}) = \widehat a.\]
\item\label{l:basicdetproperties 9} For any norm-continuous path $\mathbf a \colon [t_0,t_1] \to M_n(\Ped(A))_{\mathrm{sa}}$,  \[\Det_A(e^{2 \pi i \mathbf a}) = \widehat{\mathbf a(t_1)} - \widehat{\mathbf a(t_0)}.\]
\end{enumerate}
\end{lemma}

Recall from \cite[Remark 2.9]{GMtraces} that 
\begin{equation}\label{eq: K nought}
\K_0(A) = \left\{ [p]_0 - [q]_0 \suchthat \begin{array}{c}  n \in \mathbb N, \; p,q \in M_n(\Ped(A)^\dagger), \\ p - q \in M_n(\Ped(A))  \end{array} \right\}. 
\end{equation}
The pairing map $\rho_A \colon \K_0(A) \to \T_{\mathbb R}(A)^\ast$ \cite[Proposition 2.10]{GMtraces} is given for projections $p,q \in M_n(\Ped(A)^\dagger)$ with $p-q \in M_n(\Ped(A))$ by 
\begin{equation}\label{eq:Kthtraces}
\rho_A([p]_0-[q]_0) = \widehat{p - q}.
\end{equation}
An easy consequence of Lemma \ref{l:unitary} \eqref{l:unitary-1.5} and \eqref{l:unitary-3} is that
\begin{equation}\label{eq:K one}
\K_1(A) = \{ [u]_1 \mid n \in \mathbb N, \; u \in \mathcal U_n(\Ped(A)) \}, 
\end{equation}
with $[u]_1 = [v]_1$ for $u,v \in \mathcal U_n(\Ped(A))$ if and only if there are $k \in \mathbb N$ and self-adjoint elements $a_1,a_2,\dots,a_m \in M_{n+k}(\Ped(A))_{\mathrm{sa}}$ with $v \oplus 1_k = (u \oplus 1_k) e^{2 \pi i a_1} \cdots e^{2 \pi i a_m}$. 

By Bott periodicity, there is an isomorphism $\K_0(A) \cong \varinjlim_n \pi_1(\mathcal U_n(A),1_n)$ induced by mapping the class $[p]_0 - [q]_0 \in \K_0(A)$ for projections $p,q \in M_n(A^\dagger)$ with $p-q \in M_n(A)$ to the class of the loop $\mathbf u_q^\ast \mathbf u_p \colon [0,1] \to \mathcal U_n(A)$, where
\begin{equation}
\mathbf u_p(t) = e^{2\pi i t p}, \; \mathbf u_q(t) = e^{2\pi i t q}.
\end{equation}
This gives us an isomorphism $\K_0(A) \cong \varinjlim_{n}\pi_1(\mathcal U_n(\Ped(A)),1_n)$ by viewing $\K_0(A)$ an inductive limit over $\cs$-subalgebras $D \subseteq \Ped(A)$. Explicitly, for $n \in \mathbb N$ and projections $p,q \in M_n((\Ped(A)^\dagger)$ with $p-q \in M_n(\Ped(A))$, the loop $\mathbf u_q^\ast \mathbf u_p$ takes values in $\mathcal U_n(\Ped(A))$ and we have the isomorphism
\begin{equation}\label{eq:K nought unitary loops}
[p]_0 - [q]_0 \mapsto [\mathbf u_q^\ast \mathbf u_p] \colon \K_0(A) \isom \varinjlim_{n}\pi_1(\mathcal U_n(\Ped(A)),1_n). 
\end{equation}
Under this isomorphism, the pairing map $\rho_A \colon \K_0(A) \to \T_{\mathbb R}(A)^\ast$ becomes induced by the de la Harpe--Skandalis determinant $\Det_A$:

\begin{proposition}\label{p:pairing}
Let $A$ be a $\cs$-algebra, let $n \in \mathbb N$. Then for projections $p,q \in M_n(\Ped(A)^\dagger)$ with $p-q \in M_n(\Ped(A))$,
\[ \Det_A(\mathbf u_q^\ast \mathbf u_p) = \rho_A([p]_0 - [q]_0)  = \widehat{p-q}. \]
\end{proposition}
\begin{proof}
Let $\mathbf u = \mathbf u_q^\ast \mathbf u_p$, and let $D \subseteq \Ped(A)$ be a $\cs$-subalgebra such that $p,q \in M_n(D^\dagger)$. Then $\mathbf u$ (which is $C^1$) and $\tfrac{\mathrm d}{\mathrm d t} \mathbf u$ both take values in $M_n(D^\dagger)$.

For each $\tau \in \T_{\mathbb R}(A)$ pick $\widetilde \tau \in \T_{\mathbb R}(D^\dagger)$ such that $\tau|_D = \widetilde \tau|_D$. By definition we see that $\Det_A(\mathbf u)(\tau) = \Det_{D^\dagger}(\mathbf u)(\widetilde \tau)$. As both $\mathbf u_p$ and $\mathbf u_q$ take values in $\mathcal U_n(D^\dagger)$, we may now apply Lemma \ref{l:basicdetproperties} \eqref{l:basicdetproperties 4}, \eqref{l:basicdetproperties 6} and \eqref{l:basicdetproperties 8} to reduce the expression for $\Det_{D^\dagger}(\mathbf u)(\widetilde \tau)$ to $\widetilde \tau^{(n)} (p) - \widetilde \tau^{(n)}(q) = \tau^{(n)}(p-q)$, which is $\rho_A([p]_0-[q]_0)(\tau)$ (see \eqref{eq:Kthtraces}).
\end{proof}

We will use the following which is essentially by de la Harpe and Skandalis \cite[Lemme 3]{dlHS-det}. 

\begin{lemma}\label{l:dlHStrick}
Let $A$ be a $\cs$-algebra and let $\mathbf u\colon [0,1] \to \mathcal U_n(\Ped(A))$ be a norm-continuous path. Then there are elements $w\in \mathcal U_n(\Ped(A))$ and $a\in M_n(\Ped(A))_{\mathrm{sa}}$ such that $\mathbf u(1) = \mathbf u(0) we^{2\pi i a}$, $w$ is in the commutator subgroup of $\mathcal U_n(\Ped(A))$ and $\widehat a = \mathrm{Det}_A(\mathbf u)$.
\end{lemma}
\begin{proof}
Replacing $\mathbf u$ with $\mathbf u(0)^\ast  \mathbf u$, we may assume that $\mathbf u(0) = 1_n$. By Lemma \ref{l:unitary} \eqref{l:unitary-4} and \eqref{l:unitary-5} there is a $\cs$-subalgebra $D \subseteq \Ped(A)$ and elements $a_1,\dots,a_m \in M_n(D)_{\mathrm{sa}}$ such that $\mathbf u(1) = e^{2\pi i a_1} \cdots e^{2\pi i a_m}$ and $\Det_A(\mathbf u)  = \sum_{j=1}^m \widehat a_j$.

Now, pick $N\in \mathbb N$ sufficiently large so that $\sum_{j=1}^m\| e^{2\pi i a_j/N} - 1\| < 2$. By consideration of the unitary path 
\begin{equation}
\mathbf w \colon t \mapsto e^{2\pi i t a_1/N} \cdots e^{2\pi i t a_m/N} \colon [0,1] \to \mathcal U_n(\Ped(A)),
\end{equation}
the element
\begin{equation}
a :=  \tfrac{N}{2\pi i} \log(e^{2\pi i a_1/N} \cdots e^{2\pi i a_m/N}) \in M_n(\Ped(A))_{\mathrm{sa}}
\end{equation}
satisfies 
\begin{equation}
\widehat{a} = N \Det_A(\mathbf w) = N \sum_{j=1}^m \widehat{a_j/N}  = \sum_{j=1}^m \widehat a_j = \mathrm{Det}_A(\mathbf u).
\end{equation}
Moreover, as $\mathbf u(1) = e^{2\pi i a_1} \cdots e^{2\pi i a_m}$ and $(e^{2\pi i a_1/N} \cdots e^{2\pi i a_m/N})^N$ agree modulo the commutator subgroup of $\mathcal U_n(D)$, there is a $w\in \mathcal U_n(D)$ which is a product of commutators such that 
\begin{equation}
\mathbf u(1) = w(e^{2\pi i a_1/N} \cdots e^{2\pi i a_m/N})^N = w e^{2\pi i a}. \qedhere
\end{equation}
\end{proof}

We will also be using some results on when unitaries or unitary paths in $\mathcal U_n(\Ped(A))$ are conjugated by paths in $\mathcal U_n(\mathcal M(A))$. The subtlety of the following lemma is that we cannot simply use the multiplicative property of $\Det_A$ to write the determinant as the sum of three determinants, since the determinant is not well-defined for paths in $\mathcal U_n(\mathcal M(A))$.

\begin{lemma}\label{l:pathconjmulti}
Let $A$ be a $\cs$-algebra, let $n \in \mathbb N$ and consider norm-continu\-ous unitary paths $\mathbf u\colon [0,1] \to \mathcal U_n(\Ped(A))$ and $\mathbf v \colon [0,1] \to \mathcal U_n(\mathcal M(A))$. Then $\Det_A(\mathbf v \mathbf u \mathbf v^\ast) = \Det_A(\mathbf u)$.
\end{lemma}
\begin{proof}
    The path $\mathbf v\mathbf u \mathbf v^\ast$ is endpoint-preserving homotopic to the concatenation of the paths $v \mathbf u v^\ast$ and $\mathbf v u \mathbf v^\ast$, where $v = \mathbf{v}(0)$ and $u = \mathbf{u}(1)$. It follows from the fact that $\tau^{(n)}(v a v^\ast) = \tau^{(n)}(a)$ for each $\tau \in \T_{\mathbb R}(A)$ and $a \in M_n(\Ped(A))$ that $\Det_A( v \mathbf u v^\ast) = \Det_A(\mathbf u)$, so it suffices to show that $\Det_A(\mathbf v u \mathbf v^\ast)=0$. 
    
    By a standard rotation argument (see for instance \cite[Lemma 2.1.5]{Rordam-Kbook}), we may find a norm-continuous $\mathbf w \colon [0,1] \to \mathcal U_{2n}(\Ped(A))$ with $\mathbf w(0) = u\oplus u^\ast$ and $\mathbf w(1)= 1_{2n}$. 
    Let $\mathbf V = \mathbf v \oplus \mathbf 1_n$. Then 
    $(\mathbf v u \mathbf v^\ast)\oplus u^\ast = \mathbf V (u\oplus u^\ast) \mathbf V^\ast$
    is endpoint-preserving homotopic to the concatenation of the three paths 
    \[
    \mathbf V(0) \mathbf w \mathbf V(0)^\ast, \quad \mathbf V \mathbf w(1) \mathbf V^\ast = \mathbf 1_{2n}, \quad \mathbf V(1)  \overline{\mathbf w} \mathbf V(1)^\ast
    \]
    where $\overline{\mathbf w}(s) = \mathbf w(1-s)$. 
    The determinant of this concatenation is zero by Lemma \ref{l:basicdetproperties}. Hence
    \[
    0 = \Det_A((\mathbf v u \mathbf v^\ast) \oplus u^\ast) = \Det_A(\mathbf v u \mathbf v^\ast). \qedhere
    \]
\end{proof}

To handle paths defined only by asymptotic behaviour, we will use the following.

\begin{lemma}\label{l:Dvtrace}
Let $A$ be a $\cs$-algebra and let $D\subseteq \Ped(A)$ be a hereditary $\cs$-subalgebra. For each $n \in \mathbb N$ and $v\in \mathcal U_n(\mathcal M(A))$ define $D_v = \cs(M_n(D), vM_n(D)v^\ast)$. For every $\tau\in \T_{\mathbb R}(A)$ we have $\| \tau^{(n)}|_{D_v}\| \leq 2 n \| \tau|_D\|$. 
\end{lemma}
\begin{proof}
It suffices to prove the case $n=1$ since $M_n(D) \subseteq \Ped(M_n(A))$ is hereditary and $\| \tau^{(n)}|_{M_n(D)}\| = n \| \tau|_D\|$.

We first prove the lemma for positive traces $\tau\in \T_+(A)$. Let $(e_i)_i$ be an approximate identity in $D$ for which $\| e_i\| < 1$ for all $i$. Let $d_i := (1-e_i)^{-1} e_i$ and set
\[
f_i := (1+d_i + vd_i v^\ast)^{-1} (d_i + vd_iv^\ast) \in D_v.
\]
Then $(f_i)_i$ forms a net of positive contractions in $D_v$ for which $f_i \geq e_i$ and $f_i \geq ve_i v^\ast$ (see the proof of \cite[Theorem 1.4.2]{Pedersen-book-automorphism}). Hence for $y\in DA$ we have
\[
\| (1-f_i)y\| = \| y^\ast(1-f_i)^2 y\|^{1/2} \leq \| y^\ast (1-f_i) y \|^{1/2} \leq \| y^\ast(1-e_i)y \|^{1/2} \to 0,
\]
and similarly $\| (1-f_i)z \| \to 0$ for $z\in vDv^\ast A$. Since $D_v \subseteq \overline{DA + vDv^\ast A}$ it follows that $\|f_i x-  x\| \to 0$ for all $x\in D_v$. By \cite[Proposition 2.2]{GMtraces} we have $D_v \subseteq \Ped(A)$ and hence $\tau|_{D_v}$ is bounded. For every positive contraction $x\in D_v$ we therefore have
\[
\tau(x) = \lim_i \tau(f_i x) = \lim_i \tau(f_i^{1/2} x f_i^{1/2}) \leq \limsup_i \tau(f_i).
\]
Consequently $\|\tau|_{D_v}\|  \leq \limsup_i \tau(f_i)$. Note that
\begin{eqnarray}
f_i &=& d_i^{1/2} (1+d_i + vd_iv^\ast)^{-1} d_i^{1/2} + vd_i^{1/2} v^\ast (1+d_i + vd_i v^\ast)^{-1} vd_i^{1/2} v^\ast \nonumber\\
&\leq& \tfrac{d_i}{1+d_i} + \tfrac{vd_iv^\ast}{1+vd_i v^\ast}, \nonumber
\end{eqnarray}
since taking inverses is operator monotone decreasing on positive invertible elements, see \cite[Proposition 1.3.6]{Pedersen-book-automorphism}. The result for $\tau$ positive therefore follows since
\[
\tau(f_i) \leq \tau \left(\tfrac{d_i}{1+d_i}+\tfrac{vd_iv^\ast}{1+ vd_iv^\ast}\right) = 2 \tau\left(\tfrac{d_i}{1+d_i}\right) \leq 2 \| \tau|_D\|.
\]
For a general $\tau \in \T_{\mathbb R}(A)$, let $(\tau^+,\tau^-)$ be the Jordan decomposition of $\tau$ (see \cite[Definition 3.6]{GMtraces}). In particular, $\tau = \tau^+- \tau^-$, $\tau^\pm \in \T_+(A)$, and $\| \tau^+|_D\| + \| \tau^-|_D\| = \| \tau|_D\|$ (since $D$ is hereditary). By the above,
\[
\| \tau|_{D_v}\| \leq \| \tau^+|_{D_v} \| + \| \tau^- |_{D_v} \| \leq 2(\| \tau^+|_D\| + \| \tau^-|_D\|) = 2 \| \tau|_D\|. \qedhere
\]
\end{proof}

\begin{proposition}\label{p:DetOfShortPath}
Let $A$ be a $\cs$-algebra, let $D \subseteq \Ped(A)$ be a $\cs$-subalge\-bra, let $n \in \mathbb N$, and let $\mathbf u \colon [0,1] \to \mathcal U_n(\Ped(A))$ be a norm-continuous path such that $\mathbf u ([0,1]) \subseteq \bigcup_{v\in \mathcal U_n(\mathcal M(A))} v \mathcal U_n(D) v^\ast$. Then
\[ 
\Det_A(\mathbf u |_{[\delta, 1]}) \to \Det_A(\mathbf u) 
\]
as $\delta \to 0$ uniformly on compact subsets of $\T_{\mathbb R}(A)$.
\end{proposition}
\begin{proof}
We prove instead that $\Det_A(\mathbf u|_{[0,\delta]}) \to 0$ uniformly on compact subsets, which is equivalent by Lemma \ref{l:basicdetproperties}.

Pick $t_0 > 0$ so that $\lVert \mathbf u(0) - \mathbf u(t) \rVert < 2$ for $t \leq t_0$, and define $\mathbf a \colon [0, t_0] \to M_n(\Ped(A))_{\mathrm{sa}}$ by $\mathbf a(t) = \tfrac{1}{2\pi i } \log(\mathbf u(0)^\ast \mathbf u(t))$. Then $\mathbf a$ is norm-continuous, $\mathbf a(0)= 0$ and $\mathbf u|_{[0,t_0]} = \mathbf u(0)e^{2\pi i \mathbf a}$. Hence, by Lemma \ref{l:basicdetproperties},
\[
\Det_A(\mathbf u|_{[0,\delta]}) = \widehat{\mathbf a(\delta)} , \qquad \delta \in [0,t_0].
\]
By \cite[Proposition 2.2]{GMtraces} we may assume without loss of generality that $D$ is hereditary and $\mathbf u(0) \in \mathcal U_n(D)$. For every $t\in [0,t_0]$ we pick $v_t \in \mathcal U_n(\mathcal M(A))$ such that $\mathbf u(t) \in v_t \mathcal U_n(D) v_t^\ast$, and hence $\mathbf a(t) \in \cs(M_n(D), v_t M_n(D) v_t^\ast)$.

Let $S \subseteq \T_{\mathbb R}(A)$ be compact. By Lemma \ref{l:Dvtrace}, for $\tau\in S$ and $0 < \delta \leq t_0$ we have
\[
|\Det_A(\mathbf u|_{[0,\delta]})(\tau)| = |\tau^{(n)}(\mathbf a(\delta))| \leq 2 n \| \tau|_D\| \| \mathbf a(\delta)\|.
\]
Since $\sup_{\tau\in S} \| \tau|_D\| < \infty$ by \cite[Lemma 4.2]{GMtraces}, and since $\|\mathbf a(\delta)\| \to 0$ as $\delta\to 0$, the result follows.
\end{proof}

\begin{example}
 The following example illustrates that it is generally not true that $\Det_A(\mathbf u|_{[\delta,1]}) \to \Det_A(\mathbf u)$ as $\delta\to 0$.
 
    Consider the $\cs$-algebra of compact operators $\mathcal K$ with standard matrix units $e_{i,j}$ ($i,j\in \mathbb N$), for which $\T_{\mathbb R}(\mathcal K)$ is isomorphic to $\mathbb R$, spanned by the standard trace $\Tr$. Define a continuous map $\mathbf h \colon [1,\infty) \to \Ped(\mathcal K)_{\mathrm{sa}}$ by 
    \[
        \mathbf h(n+s) = (1-s)\frac{1}{n} \sum_{j=1}^n e_{j,j} + s \frac{1}{n+1} \sum_{j=1}^{n+1} e_{j,j}, \qquad n\in \mathbb N, \, s\in [0,1].
    \]
    Note that $\Tr(\mathbf h(t)) = 1$ for all $t\in [1,\infty)$ but that $\| \mathbf h(t)\| \to 0$ as $t\to \infty$. 
    Use this to define $\mathbf a \colon [0,1] \to \Ped(\mathcal K)_{\mathrm{sa}}$  by $\mathbf a(0) = 0$ and
    \[
    \mathbf a(t) = \mathbf h(1/t) , \quad t\in (0,1].
    \]
    Let $\mathbf u \colon [0,1] \to \mathcal U(\Ped(\mathcal K))$ be $e^{2\pi i \mathbf a}$.  By Lemma \ref{l:basicdetproperties} \eqref{l:basicdetproperties 9}, we have 
    \[
\Det_{\mathcal K}(\mathbf u|_{[\delta, 1]}) = \widehat{\mathbf a(1)} - \widehat{\mathbf a(\delta)} = 0
    \]
    for $\delta\in (0,1]$, while $\Det_{\mathcal K}(\mathbf u)(\Tr) =  \Tr(\mathbf a(1)) - \Tr(\mathbf a(0))  = 1$.
\end{example}

The above example can be modified to obtain other exotic behaviour by considering
\[
\mathbf h(n+s) = (1-s) \lambda_n \sum_{j=1}^n e_{j,j} + s\lambda_{n+1} \sum_{j=1}^{n+1} e_{j,j}
\]
for any sequence of scalars $\lambda_n \to 0$.

\begin{proposition}\label{p:Detasym}
    Let $v,w\in \mathcal U_n(\Ped(A))$ and suppose that $\mathbf u \colon [0,1) \to \mathcal U_n(\mathcal M(A))$ is norm-continuous such that $\| \mathbf u(t) v \mathbf u(t)^\ast - w\| \to 0$ as $t\to 1$. Define the continuous path
    \[
        \mathbf U \colon [0,1] \to \mathcal U_n(\Ped(A)), \quad \mathbf U(t) = \left\{ \begin{array}{ll} \mathbf u(t) v \mathbf u(t)^\ast , & \textrm{ for } t\in [0,1), \\
        w , & \textrm{ for } t=1. \end{array}\right.
    \]
    Then $\Det_A(\mathbf U) = 0$.
\end{proposition}
\begin{proof}
  Let $D\subseteq \Ped(A)$ be a $\cs$-subalgebra such that $v,w\in \mathcal U_n(D)$. Then $\mathbf U([0,1]) \subseteq \bigcup_{u\in \mathcal U_n(\mathcal M(A))} u \mathcal U_n(D) u^\ast$, so by Proposition \ref{p:DetOfShortPath} we get that $\Det_A(\mathbf U) = \lim_{\delta \to 0} \Det_A(\mathbf U|_{[0,1-\delta]})$, but $\Det_A(\mathbf U|_{[0,1-\delta]}) =0 $ for each $\delta > 0$ by Lemma \ref{l:pathconjmulti}.
\end{proof}

\section{The group $\Kd(A)$}\label{s:K1Delta}

In this section we construct the functor $\Kd$ which assigns to each $\cs$-algebra $A$ a topological abelian group $\Kd(A)$.
Let 
\[\mathcal U_\infty(\Ped(A)) = \varinjlim_{n \in \mathbb N} \mathcal U_n(\Ped(A))\]
denote the inductive limit under the inclusions $u \mapsto u \oplus 1_1$. We consider $\mathcal U_\infty(\Ped(A))$ as a group, but do not directly consider any topology on it.

\begin{definition}
    Let $A$ be a $\cs$-algebra. 
    Unitaries $u,v \in \mathcal U_\infty(\Ped(A))$ are \emph{$\Delta$-equivalent}, written $u \sim_\Delta v$, if there exist $n \in \mathbb N$ with $u,v \in \mathcal U_n(\Ped(A))$ and a norm-continuous homotopy $\mathbf u \colon [0,1] \to \mathcal U_n(\Ped(A))$ from $u$ to $v$ with $\Det_A(\mathbf u) = 0$.

    We define $\Kd(A) = \mathcal U_\infty(\Ped(A))/\!\sim_\Delta$.
\end{definition}

The set $\Kd(A)$ inherits a group structure from $\mathcal U_\infty(\Ped(A))$ because the determinant turns products into sums.
Note that if $a \in M_n(\Ped(A))_{\mathrm{sa}}$ has $\widehat a = 0$ then $[e^{2 \pi i a}]_\Delta = 0$, by Lemma \ref{l:basicdetproperties} which says that $\Det_A(t \mapsto e^{2 \pi i t a}) = \widehat a$.

\begin{lemma}\label{l:abelian}
Let $A$ be a $\cs$-algebra, and let $u\in \mathcal U_n(\Ped(A))$ and $v\in \mathcal U_n(\multialg{A})$. Then there exist $k \in \mathbb N$ and self-adjoint elements $b_1,\dots,b_k \in M_{4n}(\Ped(A))_{\mathrm{sa}}$ with $\widehat b_j = 0$ for each $j$ such that
\[ vuv^\ast u^\ast \oplus 1_{3n} = e^{2 \pi i b_1} \cdots e^{2 \pi i b_k}. \]
In particular, $[v u v^\ast]_\Delta = [u]_\Delta$ and $\Kd(A)$ is abelian.
\end{lemma}
\begin{proof}
Note that if we let $U = u \oplus u^\ast \in \mathcal U_{2n}(\Ped(A))$ and $V = v\oplus 1_n \in \mathcal U_{2n}(\multialg{A})$ then
\begin{equation}
vuv^\ast u^\ast \oplus 1_n = VUV^\ast U^\ast
\end{equation}
in $\mathcal U_{2n}(\Ped(A))$. Since $U$ is homotopic to $1_{2n}$ we may pick $m \in \mathbb N$ and $a_1,\dots, a_m \in M_{2n}(\Ped(A))_{\mathrm{sa}}$ such that
\begin{equation}
U = e^{2 \pi i a_1} \cdots e^{2 \pi i a_m}.
\end{equation}
Define $b_1,\dots,b_{2m} \in M_{4n}(\Ped(A))_{\mathrm{sa}}$ by
\begin{equation}
b_j := \begin{cases} Va_jV^\ast \oplus (-a_j) & 1 \leq j \leq m, \\ (-a_{2m-j+1}) \oplus a_{2m-j+1} & m+1 \leq j \leq 2m. \end{cases}
\end{equation}
In $\mathcal U_{4n}(\Ped(A))$ we compute
\begin{eqnarray}
VUV^\ast U^\ast \oplus 1_{2n} &=& vuv^\ast u^\ast \oplus 1_{3n}  \\
&=& (V\oplus 1_{2n})(e^{2 \pi i (a_1 \oplus (-a_1))} \cdots e^{2 \pi i (a_m \oplus (-a_m))}) \nonumber \\
& & \mathop{\cdot} (V^\ast \oplus 1_{2n}) ( e^{ 2 \pi i ((-a_m) \oplus a_m)} \cdots e^{ 2 \pi i ((-a_1)\oplus a_1)}) \nonumber\\
&=&  e^{2 \pi i b_1} \cdots e^{2 \pi i b_{2m}}. \nonumber
\end{eqnarray}
We see that $\widehat{b}_j = 0$ for each $j$ because 
\begin{equation}
(Va_jV^\ast \oplus (-a_j))^\wedge = (Va_jV^\ast)^\wedge - \widehat a_j = \widehat a_j - \widehat a_j = 0,
\end{equation}
and similarly (but easier) $(a_{2m-j+1} \oplus (-a_{2m-j+1}))^\wedge = 0$. It follows that  
\begin{equation}
[vuv^\ast u^\ast \oplus 1_{3n}]_\Delta = [e^{2 \pi i b_1} \cdots e^{2 \pi i b_{2m}}]_\Delta = 0,
\end{equation}
and hence $[vuv^\ast]_{\Delta} = [u]_{\Delta}$. If $v \in \mathcal U_n(\Ped(A))$ we get $[vu]_\Delta = [uv]_\Delta$.
\end{proof}

The following is analogous to Whitehead's Lemma; that a unitary is homotopic to $1$ if and only if it is a product of exponentials, although our version uses matrix amplifications.

\begin{proposition}\label{p:Delta equivalence}
Let $A$ be a $\cs$-algebra and let $u,v \in \mathcal U_n(\Ped(A))$. Then $u \sim_\Delta v$ if and only if there exist $m,k\in \mathbb N$ and self-adjoint elements $a_1$, $a_2,\dots,a_k \in M_{n+m}(\Ped(A))$ with $\widehat a_j = 0$ for each $j$ such that
\[ u^\ast v \oplus 1_m = e^{2 \pi i a_1} \cdots e^{2 \pi i a_k}. \]
\begin{proof}
For the `if' direction, simply observe that
\[
[0,1]\ni t \mapsto (u \oplus 1_m) e^{2\pi i t a_1} \cdots e^{2\pi i t a_k}
\]
is a homotopy from $u \oplus 1_m$ to $v \oplus 1_m$ with determinant $\sum_{j=1}^k \widehat a_j = 0$.

Conversely, if $u \sim_\Delta v$ then there is a norm-continuous path $\mathbf u \colon [0,1] \to \mathcal U_n(\Ped(A))$ from $u$ to $v$ with $\Det_A(\mathbf u) = 0$. By Lemma \ref{l:dlHStrick} there is an element $w \in \mathcal U_n(\Ped(A))$ in the commutator subgroup and a self-adjoint element $a \in M_n(\Ped(A))$ 
with $\widehat a = \Det_A(\mathbf u) = 0$ such that 
\[ u^\ast v = w e^{2 \pi i a}. \]
Since $w$ is in the commutator subgroup of $\mathcal U_n(\Ped(A))$, there are unitaries $r_1,\dots,r_k$ and $s_1,\dots,s_k$ with $w = [r_1,s_1]\cdots[r_k,s_k]$. Applying Lemma \ref{l:abelian} $k$ times, we are done.
\end{proof}
\end{proposition}

$\Kd$ is a functor from the category of $\cs$-algebras with (arbitrary) $\ast$-homomorphisms to the category of abelian groups as follows: let $\phi \colon A \to B$ be a $\ast$-homomorphism. We abuse notation and consider $\phi$ as a unital map from $M_n(A^\dagger)$ to $M_n(B^\dagger)$ for all $n$ in the obvious way. For every $u\in \mathcal U_n(\Ped(A))$ we have $\phi(u) -1_n = \phi(u-1_n) \in M_n(\Ped(B))$ (since $\phi(\Ped(A)) \subseteq \Ped(B)$).  Hence we define
\begin{equation}
\Kd(\phi)([u]_\Delta) = [\phi(u)]_\Delta.
\end{equation}
Clearly this turns $\Kd$ into a functor.

\begin{remark}
As opposed to some constructions of algebraic $\K_1$-like invariants, $\Kd(A)$ generally does not agree with the kernel of the canonical map $\Kd(A^\dagger) \to \Kd(\mathbb C)$ for non-unital $A$, where $A^\dagger \to \mathbb C$ is the canonical character. The main reason is that traces $\tau\in \T_{\mathbb R}(A)$ generally do not extend to $A^\dagger$. For instance, only the zero trace extends to $A^\dagger$ if $A$ is stable. 
\end{remark}

\begin{lemma}
Let $A$ be a $\cs$-algebra and let $u,v\in \mathcal U_n(\Ped(A))$. Then $[uv]_\Delta = [u\oplus v]_\Delta$. 
\end{lemma}
\begin{proof}
By Lemma \ref{l:abelian} applied to $v\oplus 1_n \in \mathcal U_{2n}(\Ped(A))$ and $\begin{psmallmatrix} 0_n & 1_n \\ 1_n & 0_n \end{psmallmatrix} \in \mathcal U_{2n}(\multialg{A})$ we get $[1_n\oplus v]_\Delta = [v\oplus 1_n]_\Delta$. Hence 
\begin{equation}
[uv]_\Delta = [u\oplus 1_n]_\Delta + [v\oplus 1_n]_\Delta = [u\oplus 1_n]_\Delta + [1_n\oplus v]_\Delta = [u\oplus v]_\Delta. \qedhere
\end{equation}
\end{proof}

\begin{lemma}\label{l:a+b}
Let $A$ be a $\cs$-algebra and let $a,b\in M_n(\Ped(A))_{\mathrm{sa}}$. Then 
\begin{equation}\label{eq:a+b}
[e^{2\pi i a}]_\Delta + [e^{2\pi i b}]_\Delta = [e^{2\pi i (a+b)}]_\Delta.
\end{equation} 
\end{lemma}
\begin{proof}
Let $a_0 = a\oplus (-a) \oplus 0$, $b_0 = b\oplus 0 \oplus (-b)$ and $c = (-a-b) \oplus a\oplus b$ in $M_{3n}(\Ped(A))_{\mathrm{sa}}$. Then $\widehat a_0 = \widehat b_0 = \widehat c = 0$ and 
\begin{equation}
e^{2\pi i a}e^{2\pi i b}e^{-2\pi i (a+b)}\oplus 1_n \oplus 1_n = e^{2\pi i a_0} e^{2\pi i b_0} e^{2\pi i c}
\end{equation}
has trivial $\Kd$-class.
\end{proof}

We equip $\Kd(A)$ with a topology turning it into a topological group. Intuitively, two elements $x,y\in \Kd(A)$ are close if $x=[u]_\Delta$ and $y= [v]_\Delta$ for some $u,v\in \mathcal U_n(\Ped(A))$ that are homotopic via a path $\mathbf u$ with small determinant $\Det_A(\mathbf u)$.

Explicitly, given $x\in \Kd(A)$, a compact subset $S\subseteq \T_{\mathbb R}(A)$ and $\epsilon >0$, we let $U_{x,S,\epsilon}$ denote the set
\[
U_{x,S,\epsilon} = \left\{ y\in \Kd(A) \suchthat \begin{array}{l}
\textrm{there is $n \in \mathbb N$ and a norm-continuous path } \\
\mathbf u \colon [0,1] \to \mathcal U_n(\Ped(A)) \textrm{ with } [ \mathbf u (0) ]_\Delta = x, \\
{} [ \mathbf u(1) ]_\Delta = y,  \textrm{ and } \| \Det_A(\mathbf u) \|_S < \epsilon
\end{array}
\right\} .
\]
For the reader's convenience we recall that $\Det_A(\mathbf u) \in \T_{\mathbb R}(A)^\ast$ and that for $f\in \T_{\mathbb R}(A)^\ast$, $\| f \|_S$ is given by $\sup_{\tau\in S} |f(\tau)|$. The seminorms $\| \cdot \|_S$ for compact $S \subseteq \T_{\mathbb R}(A)$ induce the compact-open topology on $\T_{\mathbb R}(A)^\ast$.

\begin{proposition}
Let $A$ be a $\cs$-algebra. The sets $U_{x,S,\epsilon}$ for $x\in \Kd(A)$, $S\subseteq \T_{\mathbb R}(A)$ compact and $\epsilon >0$ form the basis for a topology on $\Kd(A)$, with respect to which it forms a topological group. 
\end{proposition}
\begin{proof}

For $S \subseteq \T_{\mathbb R}(A)$ compact and $\epsilon > 0$ let 
\[O_{S,\epsilon} = \{ f \in \T_{\mathbb R}(A)^\ast \mid \| f \|_S < \epsilon \}.\]
These sets form a neighbourhood basis $\mathcal O = \{O_{S,\epsilon}\}_{S,\epsilon}$ of $0$ for the compact-open topology on $\T_{\mathbb R}(A)^\ast$.
For $i = 0,1$ let $x_i \in \Kd(A)$, let $S_i \subseteq \T_{\mathbb R}(A)$ be compact and let $\epsilon_i > 0$. We show that for any $x \in U_{x_0,S_0,\epsilon_0} \cap U_{x_1,S_1,\epsilon_1}$ there is a compact subset $S \subseteq \T_{\mathbb R}(A)$ and $\epsilon > 0$ such that $U_{x,S,\epsilon} \subseteq U_{x_0,S_0,\epsilon_0} \cap U_{x_1,S_1,\epsilon_1}$.

Pick $n \in \mathbb N$ and norm-continuous paths $\mathbf u_i \colon [0,1] \to \mathcal U_n(\Ped(A))$ for $i = 0,1$ with $[ \mathbf u_i (0) ]_\Delta = x_i$, $ [ \mathbf u_i (1) ]_\Delta = x$ and $\Det_A(\mathbf u_i) \in O_{S_i,\epsilon_i}$. By potentially enlarging $n$, we may assume that $\mathbf u_0(1) = \mathbf u_1(1)$.  
Pick $S \subseteq \T_{\mathbb R}(A)$ compact and $\epsilon > 0$ such that $\Det_A(\mathbf u_0) + O_{S,\epsilon} \subseteq O_{S_0,\epsilon_0}$ and $\Det_A(\mathbf u_1) + O_{S,\epsilon} \subseteq O_{S_1,\epsilon_1}$.
For each $y \in U_{x,S,\epsilon}$, we may (after perhaps enlarging $n$) pick a norm-continuous path $\mathbf u \colon [1,2] \to \mathcal U_n(\Ped(A))$ with $\| \Det_A(\mathbf u) \|_S < \epsilon$, $\mathbf u(1) = \mathbf u_i(1)$ (so that $[\mathbf u(1)]_\Delta = x$) and $[\mathbf u(2)]_\Delta = y$. The concatenation $\mathbf u_i \hash \mathbf u$ satisfies $\Det_A(\mathbf u_i \hash \mathbf u) \in \Det_A(\mathbf u_i) + O_{S,\epsilon} \subseteq O_{S_i,\epsilon_i}$, and so by doubling the speed of $\mathbf u_i \hash \mathbf u$ we find that $y \in U_{x_i,S_i,\epsilon_i}$ for $i=0,1$.

The resulting topology turns $\Kd(A)$ into a topological group because if $x_i \to x$ and $y_i \to y$ in $\Kd(A)$, with witnesses $(\mathbf u_i)_i$ and $(\mathbf v_i)_i$ respectively satisfying $\Det_A(\mathbf u_i) \to 0$ and $\Det_A(\mathbf v_i) \to 0$, then $(\mathbf u^\ast_i)_i$ witnesses that $x_i^{-1} \to x^{-1}$ and $(\mathbf u_i \mathbf v_i)_i$ witnesses that $x_i y_i \to x y$.
\end{proof}

Next we consider the \emph{Thomsen map} $\Thd_A \colon \T_{\mathbb R}(A)^\ast \to \Kd(A)$, inspired by the work of Thomsen \cite{Thomsen-traces}, see also \cite[Section 2.2]{CGSTW}.

\begin{proposition}\label{p:Thomsen map}
Let $A$ be a $\cs$-algebra. There is a continuous open homomorphism $\Thd_A \colon \T_{\mathbb R}(A)^\ast \to \Kd(A)$ such that 
\begin{equation}
\Thd_A(\widehat a) = [e^{2\pi i a}]_\Delta 
\end{equation}
for each $a \in \Ped(A)_{\mathrm{sa}}$. Moreover, for each $n \in \mathbb N$ and norm-continuous path $\mathbf u \colon [0,1] \to \mathcal U_n(\Ped(A))$ we have
    \begin{equation}\label{eq:FTC}
        \Thd_A(\Det_A(\mathbf u)) = [\mathbf u(1)]_{\Delta} - [\mathbf u(0)]_\Delta.
    \end{equation}
\end{proposition}
\begin{proof}
We first show that $\Thd_A$ is a well-defined homomorphism and continuous.
We view $\T_{\mathbb R}(A)^\ast$ as a quotient of $\Ped(A)_{\mathrm{sa}}$ by \cite[Theorem C]{GMtraces} and see that the map $a \mapsto [e^{2\pi i a}]_\Delta \colon \Ped(A)_{\mathrm{sa}} \to \Kd(A)$ vanishes for $a$ with $\widehat a = 0$ by Lemma \ref{l:basicdetproperties} \eqref{l:basicdetproperties 8}, which says that $\Det_A(t \mapsto e^{2 \pi i t a}) = \widehat a$. For continuity, given an element $a$ and a net $(a_i)_i$ in $\Ped(A)_{\mathrm{sa}}$ with $\widehat{a}_i \to \widehat{a}$ in $\T_{\mathbb R}(A)^\ast$, consider the net $(\mathbf u_i)_i$ of unitary paths $\mathbf u_i \colon t \mapsto e^{2 \pi i (ta + (1-t)a_i)}$ on $[0,1]$. This satisfies $\Det_A(\mathbf u_i) = \widehat a - \widehat a_i \to 0$ and thus $[e^{2\pi i a_i}]_\Delta \to [e^{2\pi i a}]_\Delta$.

Next we show the `moreover' statement.  Replacing $\mathbf u$ with $\mathbf u  \cdot \mathbf u(0)^\ast$, we may assume that $\mathbf u(0) = 1_n$. By Lemma \ref{l:dlHStrick} there are $w\in \mathcal U_n(\Ped(A))$ and $a\in M_n(\Ped(A))_{\mathrm{sa}}$ such that $\mathbf u(1) = w e^{2\pi i a}$, $\widehat a = \Det_A(\mathbf u)$, and $w$ is in the commutator subgroup of $\mathcal U_n(\Ped(A))$. Since $[w]_{\Delta} = 0$  we get
\begin{equation}
    [\mathbf u(1)]_{\Delta} = [w]_{\Delta} +  [e^{2\pi i a}]_{\Delta} = \Thd_A(\widehat a) = \Thd_A(\Det_A(\mathbf u)).
\end{equation}

For openness, we argue that given a compact subset $S \subseteq \T_{\mathbb R}(A)$ and $\epsilon > 0$, we have
\begin{equation}
    \Thd_A(\{f \in \T_{\mathbb R}(A)^\ast \mid \| f \|_S < \epsilon \}) = U_{0,S,\epsilon}. 
\end{equation}
If $ f \in \T_{\mathbb R}(A)^\ast$ satisfies $\| f \|_S < \epsilon$, then for any $a \in \Ped(A)_{\mathrm{sa}}$ with $\widehat a = f$, the exponential path $t \mapsto e^{2 \pi i t a}$ witnesses that $\Thd_A(f) \in U_{0,S,\epsilon}$. Conversely, if $x \in U_{0,S,\epsilon}$, then there is a norm-continuous path $\mathbf u \colon [0,1] \to \mathcal U_n(\Ped(A))$ with $[\mathbf u(0)]_\Delta = 0$ and $[\mathbf u(1)]_\Delta = x$ and $\| \Det_A(\mathbf u) \|_S < \epsilon$. Then by \eqref{eq:FTC} we have $\Thd_A(\Det_A(\mathbf u)) = x$.
\end{proof}

The Thomsen map $\Thd_A$ fits into a diagram
\begin{equation}\label{eq:exactseq}
\K_0(A) \xrightarrow{\rho_A} \T_{\mathbb R}(A)^\ast \xrightarrow{\Thd_A} \Kd(A) \to \K_1(A) \to 0,
\end{equation}
where the map $\Kd(A) \to \K_1(A)$ is the map $[u]_\Delta \mapsto [u]_1$, which is surjective by \eqref{eq:K one}. This map is moreover continuous when taking $\K_1(A)$ with the discrete topology, because each basic open set $U_{x,S,\epsilon} \subseteq \Kd(A)$ is sent to the value at $x \in \Kd(A)$.

\begin{proposition}\label{p:exactseq}
The sequence \eqref{eq:exactseq} is exact for every $\cs$-algebra $A$.
\end{proposition}
\begin{proof}
As exponential unitaries vanish in $\K_1$, it follows that the composition from $\T_{\mathbb R}(A)^\ast$ to $\K_1(A)$ vanishes. Suppose $u\in \mathcal U_n(\Ped(A))$ vanishes in $\K_1$.  We may therefore find $k,m \in \mathbb N$ and $a_1,\dots,a_m \in M_{n+k}(\Ped(A))_{\mathrm{sa}}$ such that $u \oplus 1_k = e^{2\pi i a_1} \cdots e^{2\pi i a_m}$ (see \eqref{eq:K one}). Then
\begin{equation}
[u]_\Delta = \sum_{j=1}^m [e^{2\pi i a_j}]_\Delta = \Thd_A(\sum_{j=1}^m \widehat a_j).
\end{equation} 
Hence the sequence is exact at $\Kd(A)$. 

It remains to show exactness at $\T_{\mathbb R}(A)^\ast$. 
To see that $\Thd_A \circ \rho_A = 0$, let $x \in \K_0(A)$ and pick a unitary path $\mathbf u \colon [0,1] \to \mathcal U_n(\Ped(A))$ with $\mathbf u(0) = \mathbf u(1) = 1_n$ under the isomorphism \eqref{eq:K nought unitary loops}, so that $\rho_A(x) = \Det_A(\mathbf u)$ by Proposition \ref{p:pairing}. Then $\Thd_A(\rho_A(x)) = \Thd_A(\Det_A(\mathbf u)) = [\mathbf u(1)]_\Delta - [\mathbf u(0)]_\Delta = 0$ by Proposition \ref{p:Thomsen map}.

\begin{comment}
Let $x\in \K_0(A)$. 
By Lemma \ref{l:K0Ped} we pick projections $p,q\in M_n(\Ped(A)^\sim)$ and a positive contraction $d\in M_n(\Ped(A))$ satisfying (1)-(4) in the lemma. Note that 
\begin{equation}
\Thd_A(\rho_A(x)) = [e^{2\pi i (p-q)}]_\Delta = [e^{2\pi i d(p-q)}]_\Delta  \stackrel{\eqref{eq:a+b}}{=} [e^{2\pi i dp}]_\Delta - [e^{2\pi i dq}]_\Delta.
\end{equation}
By the final part of Lemma \ref{l:K0Ped}, $e^{2\pi i dp} = e^{2\pi i dq}$ and so $\Thd_A(\rho_A(x)) = 0$.
\end{comment}

Now, let $a\in \Ped(A)_{\mathrm{sa}}$ such that $\Thd_A(\widehat a)  = 0$. By Proposition \ref{p:Delta equivalence} there exist $m,n\in \mathbb N$ and $b_1,\dots,b_m \in M_n(\Ped(A))_{\mathrm{sa}}$ with $\widehat b_j = 0$ for each $j$ and 
\begin{equation}
e^{2\pi i (a\oplus 0_{n-1})} = e^{2\pi i b_1} \cdots e^{2\pi i b_m}.
\end{equation}
Define the norm-continuous path $\mathbf u \colon [0,1] \to \mathcal U_n(\Ped(A))$ by 
\begin{equation}
\mathbf u(t) = \left\{ \begin{array}{ll} e^{(2t)2\pi i a\oplus 0_{n-1}} , & \textrm{ for } t\in [0,\tfrac{1}{2}], \\
e^{(1-2t)2\pi i b_1} \cdots e^{(1-2t)2\pi i b_m}, & \textrm{ for } t\in [\tfrac{1}{2},1]. \end{array}\right.
\end{equation}
Then $\mathbf u(0) = \mathbf u(1) = 1$ and thus $\mathbf u$ gives a $\K_0$-class $x\in \K_0(A)$ via the isomorphism \eqref{eq:K nought unitary loops}. By Lemma \ref{l:basicdetproperties} we compute
\begin{equation}
\Det_A(\mathbf u) = \widehat a - \sum_{j=1}^m \widehat b_j = \widehat a.
\end{equation}
Hence $\rho_A(x) = \widehat a$ by Proposition \ref{p:pairing}.
\end{proof}

\begin{corollary}\label{c:direct sum}
For any $\cs$-algebra $A$, $\Kd(A)$ is topologically isomorphic (in an unnatural way) to $\K_1(A) \oplus \big( \T_{\mathbb R}(A)^\ast/\rho_A(\K_0(A)) \big)$. 
\end{corollary}
\begin{proof}
The exact sequence  \eqref{eq:exactseq} induces a short exact sequence
\begin{equation}
0 \to \T_{\mathbb R}(A)^\ast/\rho_A(\K_0(A))  \to \Kd(A) \to \K_1(A) \to 0.
\end{equation}
Since $\T_{\mathbb R}(A)^\ast$ is divisible as an abelian group, so is $\T_{\mathbb R}(A)^\ast/\rho_A(\K_0(A))$ and hence the short exact sequence above splits. The splitting is (vacuously) continuous, and thus the direct sum decomposition is topological.
\end{proof}

For the following result, recall that we equip $\Ped(A)_{\mathrm{sa}}$ with the inductive limit topological vector space topology, regarding it as the directed union of $D_{\mathrm{sa}}$ for $\cs$-subalgebras $D \subseteq \Ped(A)$. 

\begin{corollary}\label{c:ContinuityOutOfK1D}
Let $A$ be a $\cs$-algebra, let $G$ be a semitopological\footnote{A semitopological group is a group with a topology such that inversion and multiplication by fixed elements are continuous, but where multiplication is not assumed jointly continuous.} group and let $\phi \colon \Kd(A) \to G$ be a homomorphism. Then $\phi$ is continuous if and only if the homomorphism $\Phi \colon \Ped(A)_{\mathrm{sa}} \to G$ given at $a \in \Ped(A)_{\mathrm{sa}}$ by
\begin{equation}
\Phi(a) = \phi([e^{2 \pi i a}]_\Delta)
\end{equation}
is continuous. 
\end{corollary}
\begin{proof}
The `only if' direction is immediate by continuity of 
\begin{equation}
\psi \colon a \mapsto [e^{2 \pi i a}]_\Delta \colon \Ped(A)_{\mathrm{sa}} \to \Kd(A).
\end{equation}
Now assume that $\Phi$ is continuous. A homomorphism between semitopological groups is continuous if and only if it is continuous at the identity. Since $\psi$ is open, then for any open neighbourhood $U \subseteq G$ of the identity, the restricted preimage $\phi^{-1}(U) \cap \psi(\Ped(A)_{\mathrm{sa}}) = \psi(\Phi^{-1}(U))$ is an open neighbourhood of $0 \in \Kd(A)$. It follows that $\phi$ is continuous at $0 \in \Kd(A)$.
\end{proof}

\begin{corollary}\label{c:K1Dfullher}
A full hereditary inclusion $A \hookrightarrow B$ of $\cs$-algebras induces a topological isomorphism $\Kd(A) \isom \Kd(B)$. 
\end{corollary}
\begin{proof}
By naturality, the inclusion $\iota \colon A\hookrightarrow B$ induces a commutative diagram
\begin{equation}
\begin{tikzcd}
	{\K_0(A)} & {\T_{\mathbb R}(A)^\ast} & {\Kd(A)} & {\K_1(A)} & {0} \\
	{\K_0(B)} & {\T_{\mathbb R}(B)^\ast} & {\Kd(B)} & {\K_1(B)} & {0}
	\arrow["{\rho_A}", from=1-1, to=1-2]
	\arrow["{\K_0(\iota)}"', "\cong", from=1-1, to=2-1]
	\arrow["{\Thd_A}", from=1-2, to=1-3]
	\arrow["{\iota_\#}"', "\cong", from=1-2, to=2-2]
	\arrow[from=1-3, to=1-4]
    \arrow[from=1-4, to=1-5]
	\arrow["{\Kd(\iota)}"', from=1-3, to=2-3]
	\arrow["{\K_1(\iota)}"', "\cong", from=1-4, to=2-4]
	\arrow["{\rho_B}", from=2-1, to=2-2]
	\arrow["{\Thd_B}", from=2-2, to=2-3]
	\arrow[from=2-3, to=2-4]
    \arrow[from=2-4, to=2-5]
\end{tikzcd}
\end{equation}
with exact rows. It is classical that $\K_0(\iota)$ and $\K_1(\iota)$ are isomorphisms, and $\iota_\#$ is a topological isomorphism by \cite[Corollary 5.4]{GMtraces}. Hence $\Kd(\iota)$ is an isomorphism by the five lemma. Its inverse is automatically continuous.
\end{proof}

\begin{lemma}\label{l:unitarycontinuity}
Let $A$ be a $\cs$-algebra, let $D \subseteq \Ped(A)$ be a $\cs$-subalgebra and let $n \in \mathbb N$. Then the canonical map $\mathcal U_n(D) \to \Kd(A)$ is continuous.
\end{lemma}
\begin{proof}
Let $u_i \to u$ in $\mathcal U_n(D)$, let $S \subseteq \T_{\mathbb R}(A)$ be compact and let $\epsilon > 0$. By \cite[Lemma 4.2]{GMtraces}, $M := \sup_{\tau \in S} \| \tau|_D\|$ is finite and, for sufficiently large $i$, $\tfrac{1}{2\pi} \arccos(1-\tfrac{1}{2}\| u_i-u\|^2) < \epsilon/((M+1)n)$. Setting $a_i = \tfrac{1}{2\pi i} \log(u^\ast u_i) \in M_n(D)$, then $\Thd_A(\widehat a_i) = [u^\ast u_i]_\Delta = [u_i]_\Delta - [u]_\Delta$, and $\| a_i \| \leq \epsilon/((M+1)n)$ by Lemma \ref{l:unitary} \eqref{l:unitary-3}, and therefore 
    \begin{equation}
        \sup_{\tau\in S} | \widehat a_i(\tau)| \leq \sup_{\tau\in S} | \tau^{(n)}(a_i)| \leq \sup_{\tau\in S} \frac{\epsilon}{(M+1)n} n \| \tau|_D\| \leq \epsilon.
    \end{equation}
Thus $\widehat a_i \to 0$ and by continuity of $\Thd_A$ (Proposition \ref{p:Thomsen map}), we conclude that $[u_i]_\Delta \to [u]_\Delta$.
\end{proof}

\begin{definition}
Two $*$-homomorphisms $\phi,\psi \colon A \to B$ are said to be \emph{approximately Murr\-ay--von Neumann equivalent}, written $\phi \sim_{\mathrm{aMvN}} \psi$, if for each $\epsilon > 0$ and each finite set $\mathcal F \subseteq A$ there exists a contraction $v \in \multialg{B}$ such that $\| v^* \phi(a) v  - \psi(a) \| < \epsilon$ and $\|v \psi(a) v^\ast - \phi(a) \| < \epsilon$ for each $a \in \mathcal F$. If $v$ may be chosen to be unitary, $\phi$ and $\psi$ are said to be \emph{approximately unitarily equivalent} (with multiplier unitaries), written $\phi \sim_{\mathrm{au}} \psi$.

Assume $A$ to be separable. The $*$-homomorphisms $\phi,\psi \colon A \to B$ are said to be \emph{asymptotically Murray--von Neumann equivalent}, written $\phi \sim_{\mathrm{asMvN}} \psi$, if there is a norm-continuous path $t \mapsto v_t \colon [0,\infty) \to \multialg{B}$ of contractions in the multiplier algebra of $B$ such that, for each $a \in A$, $v_t^* \phi(a) v_t \to \psi(a)$ and $v_t \psi(a) v_t^* \to \phi(a)$ as $t \to \infty$. If $(v_t)_{t \in [0,\infty)}$ may be chosen to be unitary for each $t$, $\phi$ and $\psi$ are said to be \emph{asymptotically unitarily equivalent} (with multiplier unitaries), written $\phi \sim_{\mathrm{asu}} \psi$.
\end{definition}

Many functors that are invariant under approximate or asymptotic unitary equivalence will also be invariant under approximate or asymptotic Murray von Neumann equivalence (in particular, if the functor is $M_2$-stable\footnote{This means that the inclusion $A\to M_2(A)$ in the $(1,1)$-corner induces an isomorphism. Note that this is an inclusion of a full hereditary $C^\ast$-subalgebra, so the functors $\Kd$ and $\T_{\mathbb R}(-)^\ast$ are $M_2$-stable.} by a unitary dilation trick, see \cite[Corollary 3.11]{Gabe-o2class}). This in particular applies to $\Kd$ (and $\HKd$ introduced in the following section).

The following is immediate from Proposition \ref{p:Detasym}.

\begin{proposition}\label{p:asu}
Let $\phi, \psi\colon A \to B$ be $\ast$-homomorphisms between $\cs$-alge\-bras with $A$ separable. If $\phi$ and $\psi$ are asymptotically unitarily equivalent (with multiplier unitaries), then $\Kd(\phi) = \Kd(\psi)$. 

More generally, if $\phi$ and $\psi$ are asymptotically Murray--von Neumann equivalent, then $\Kd(\phi) = \Kd(\psi)$.
\end{proposition}

\section{The group $\HKd(A)$}\label{s:HK1Delta}

The topological abelian group $\Kd(A)$ often fails to be Hausdorff, with the identity element $0$ not closed. As we shall shortly see this happens precisely when the image of $\rho_A \colon \K_0(A) \to \T_{\mathbb R}(A)^*$ is not closed.

\begin{definition}
Let $A$ be a $\cs$-algebra.
The group $\HKd(A)$ is the quotient of $\Kd(A)$ by the closure $\overline{\{0\}}$ of $0$. For $u\in \mathcal U_n(\Ped(A))$ we let $[u]_{\overline{\Delta}}$ denote the class of $u$ in $\HKd(A)$. We equip $\HKd(A)$ with the quotient topology, under which it is a Hausdorff topological abelian group.
\end{definition}

\begin{remark}
We refer to $\HKd$ as Hausdorffised $\Kd$ and it is the Hausdorffisation of $\Kd$ in the sense that any continuous homomorphism $\Kd(A) \to G$ to a Hausdorff topological abelian group factors through the quotient map $\Kd(A) \to \HKd(A)$.
\end{remark}

Functoriality of $\HKd$ with respect to $*$-homomorphisms is inherited from $\Kd$. The following is immediate from Corollary \ref{c:K1Dfullher}.

\begin{proposition}\label{p:HK1Dfullher}
    A full hereditary inclusion $A\hookrightarrow B$ of $\cs$-algebras induces a topological isomorphism $\HKd(A) \isom \HKd(B)$.
\end{proposition}

For the following result we highlight that we consider $\T_{\mathbb R}(A)^\ast$ with the compact-open topology, and that $\rho_A(\K_0(A)) \subseteq \T_{\mathbb R}(A)^\ast$ coincides with the kernel of the Thomsen map $\Thd_A \colon \T_{\mathbb R}(A)^\ast \to \Kd(A)$ by Proposition \ref{p:exactseq}.

\begin{proposition}\label{p:zeroclosure}
For any $\cs$-algebra $A$ we have $\overline{\{0\}} = \Thd_A\left(\oimrho{A}\right)$ in $\Kd(A)$. 
\end{proposition}
\begin{proof}
The inclusion $\Thd_A\left(\overline{\rho_A(\K_0(A))}\right) \subseteq \overline{\{0\}}$ follows from continuity of $\Thd_A$.
The image of $\Thd_A$ is closed because the exact sequence \eqref{eq:exactseq} has continuous maps with $\K_1(A)$ discrete. Openness of $\Thd_A$ thus implies that $\Thd_A\left(\overline{\rho_A(\K_0(A))}\right)$ is closed, and therefore contains $\overline{\{0\}}$.
\end{proof}

The canonical surjection $\Kd(A) \to \K_1(A)$ factors through $\HKd(A)$. Let 
\begin{equation}
\Thod_A \colon \T_{\mathbb R}(A)^\ast \to \HKd(A)
\end{equation}
denote the composition of $\Thd_A$ with the quotient map. Explicitly, this is given by $\widehat a \mapsto [e^{2\pi i a}]_{\overline \Delta}$ for $a\in \Ped(A)_{\mathrm{sa}}$.
The following is immediate from Propositions \ref{p:exactseq} and \ref{p:zeroclosure}.
\begin{corollary}\label{c:HK1Dexactseq}
For every $\cs$-algebra $A$, the sequence below is exact.
\begin{equation}\label{eq:ThomsenODelta}
0 \to \frac{\T_{\mathbb{R}}(A)^\ast}{\, \, \overline{\rho_A(\K_0(A))} \, \,} \xrightarrow{\Thod_A} \HKd (A) \to \K_1(A) \to 0
\end{equation}
\end{corollary}
\begin{remark}
The short exact sequence \eqref{eq:ThomsenODelta} splits unnaturally because $\T_{\mathbb R}(A)^\ast$ is divisible, so there is an unnatural topological isomorphism
\begin{equation}
\HKd (A) \cong  \frac{\T_{\mathbb{R}}(A)^\ast}{\, \, \overline{\rho_A(\K_0(A))} \, \,}\oplus \K_1(A).
\end{equation}
\end{remark}

One of the key advantages for working with $\HKd$ instead of $\Kd$, is that, on morphisms, it is an approximate unitary equivalence invariant.

We say that two unitaries $v,w \in \mathcal U(A)$ are \emph{approximately unitarily equivalent} (with unitaries in $\mathcal U(\mathcal M(A))$) if there is a sequence $(u_n)_{n\in \mathbb N}$ of unitaries in $\mathcal M(A)$ such that $u_n v u_n^\ast$ converges to $w$ in norm.

\begin{lemma}
Let $A$ be a $\cs$-algebra, let $n \in \mathbb N$, and let $v,w\in \mathcal U_n(\Ped(A))$ be approximately unitarily equivalent (with unitaries in $\mathcal U_n(\multialg{A})$). Then $[v]_{\overline{\Delta}} = [w]_{\overline{\Delta}}$. 
\end{lemma}
\begin{proof}
Let $(u_k)_k$ be a sequence in $\mathcal U_n(\mathcal M(A))$ such that $\| u_k v u_k^\ast - w \| \to 0$. We may assume $\| u_k v u_k^\ast - w\| < 2$ for all $k$. Let $a_k = \tfrac{1}{2\pi i} \log(w^\ast u_kv u_k^\ast)$. Then $[0,1] \ni s \mapsto w e^{2\pi i s a_k}$ is a unitary path from $w$ to $u_k v u_k^\ast$ with determinant
\[
\Det_A(s \mapsto e^{2\pi i s a_k} w) = \widehat{a}_k.
\]
Since $[u_k v u_k^\ast]_\Delta = [v]_\Delta$ by Lemma \ref{l:abelian}, it suffices to check that $[u_k v u_k^\ast]_\Delta \to [w]_\Delta$, which happens if $\| \widehat{a}_k \|_{S} \to 0$ for each compact $S \subseteq \T_{\mathbb R}(A)$.
Let $D\subseteq \Ped(A)$ be a hereditary $\cs$-subalgebra such that $v,w\in \mathcal U_n(D)$. 
Setting $D_k = \cs(M_n(D), u_k M_n(D) u_k^\ast)$ for each $k$, Lemma \ref{l:Dvtrace} gives us
\[
\sup_{\tau\in S}\| \tau^{(n)}|_{D_k} \| \leq \sup_{\tau\in S} 2n \| \tau|_D\| =: K < \infty. 
\]
Therefore, since $\|a_k\| \to 0$ and $a_k \in D_k$ for each $k$, we have
\[
\| \widehat{a}_k \|_{S} \leq K \| a_k\| \to 0. \qedhere
\]
\end{proof}

\begin{corollary}\label{c:approx}
Let $\phi, \psi\colon A \to B$ be $\ast$-homomorphisms between $\cs$-algebras. If $\phi$ and $\psi$ are approximately unitarily equivalent (with multiplier unitaries), then $\HKd(\phi) = \HKd(\psi)$. 

More generally, if $\phi$ and $\psi$ are approximately Murray--von Neumann equivalent, then $\HKd(\phi) = \HKd(\psi)$. 
\end{corollary}

\section{The $\zeta$-maps}\label{s:zeta}

In \cite{CGSTW}, natural maps $\zeta_n \colon \K_0(A; \mathbb Z/n\mathbb Z) \to \HKa(A)$ were introduced, where $\K_0(A; \mathbb Z/n\mathbb Z)$ is the $\K$-theory group with $\mathbb Z/n\mathbb Z$-coefficients. The maps $\zeta_n$ comprised part of the total invariant $\TKTu(A)$ of a unital $\cs$-algebra, and played a subtle but crucial role in the new approach to classifying $\cs$-algebras. In this section we show that there is a similar construction $\zeta_n \colon \K_0(A;\mathbb Z/n\mathbb Z)\to \Kd(A)$. 

In fact, we give a more general construction of $\zeta_\phi \colon \K_1(C_\phi) \to \Kd(A)$ for mapping cones $C_\phi$ of $\ast$-homomorphisms $\phi \colon A\to B$ which induce an isomorphism on traces, see Lemma \ref{l:mappingcone}.

For a $*$-homomorphism $\phi \colon A \to B$ between $\cs$-algebras, the \emph{mapping cone} $C_\phi$ is the $\cs$-algebra
\begin{equation}
C_\phi := \{ (a,\mathbf b) \in A \oplus C([0,1],B) \mid \phi(a) = \mathbf b(0), \, \mathbf b(1) = 0\}.
\end{equation}
This fits naturally into a short exact sequence 
\begin{equation}\label{eq:cone short exact}
0 \to SB \to C_\phi \to A \to 0, 
\end{equation}
where the suspension $SB:= C_0((0,1),B)$ sits as an ideal via the inclusion $\mathbf b \mapsto (0,\mathbf b)$, and $C_\phi \to A$ is given by $(a,\mathbf b) \mapsto a$. By identifying $\K_{i}(SB)$ with $\K_{1-i}(B)$, the six term exact sequence in $\K$-theory becomes
\begin{equation}\label{eq:cone six term}
\begin{tikzcd}
	{\K_1(B)} & {\K_0(C_\phi)} & {\K_0(A)} \\
	{\K_1(A)} & {\K_1(C_\phi)} & {\K_0(B).}
	\arrow[from=1-1, to=1-2]
	\arrow[from=1-2, to=1-3]
	\arrow["{\K_0(\phi)}", from=1-3, to=2-3]
	\arrow["{\K_1(\phi)}", from=2-1, to=1-1]
	\arrow[from=2-2, to=2-1]
	\arrow[from=2-3, to=2-2]
\end{tikzcd}
\end{equation}

If $\phi_i \colon A_i \to B_i$ are $*$-homomorphisms for $i=0,1$ and $\psi_A \colon A_0 \to A_1$ and $\psi_B \colon B_0 \to B_1$ are $*$-homomorphisms such that $\psi_B \circ \phi_0 = \phi_1 \circ \psi_A$, we write $\psi = (\psi_A,\psi_B) \colon C_{\phi_0} \to C_{\phi_1}$ for the induced $*$-homomorphism satisfying $\psi(a,\mathbf b) = (\psi_A(a),\psi_B(\mathbf b))$ for $(a,\mathbf b) \in C_{\phi_0}$.

In the following lemma, given $n \in \mathbb N$ and $(u,\mathbf v) \in \mathcal U_n(\Ped(C_\phi))$, we consider the element $\mathbf v$ as a norm-continuous path $[0,1] \to \mathcal U_n(\Ped(B))$ in order to take its determinant $\Det_B(\mathbf v)\in \T_{\mathbb R}(B)^\ast$.

\begin{lemma}\label{l:mappingcone}
   Let $A$ and $B$ be $\cs$-algebras and suppose that $\phi \colon A\to B$ is a $\ast$-homomorphism such that $\phi_\# \colon \T_{\mathbb R}(A)^\ast \to \T_{\mathbb R}(B)^\ast$ is an isomorphism. Let $C_\phi$ be the mapping cone of $\phi$. Then there is a (unique) homomorphism $\zeta_\phi \colon \K_1(C_\phi) \to \Kd(A)$ such that 
   \begin{equation}\label{eq:zetaphi_statement}
   \zeta_\phi([u,\mathbf v]_1) = [u]_{\Delta} + \mathrm{Th}_A^\Delta(\phi_\#^{-1}(\Det_B(\mathbf v)))
   \end{equation}
   for each $(u,\mathbf v) \in \mathcal U_\infty(\Ped(C_\phi))$.  
   This makes the following diagram commute:
    \begin{equation}\label{eq:zetacone}
    \begin{tikzcd}
	{\K_0(B)} & {\K_1(C_\phi)} & {\K_1(A)} \\
	{\T_{\mathbb R}(A)^\ast} & {\Kd(A)} & {\K_1(A)}
	\arrow[from=1-1, to=1-2]
	\arrow["{\phi_\#^{-1} \circ \rho_B}", from=1-1, to=2-1]
	\arrow[from=1-2, to=1-3]
	\arrow["{\zeta_\phi}", from=1-2, to=2-2]
	\arrow["\id", from=1-3, to=2-3]
	\arrow["{\Thd_A}", from=2-1, to=2-2]
	\arrow[two heads, from=2-2, to=2-3]
    \end{tikzcd}
    \end{equation}
    Moreover, $\zeta_\phi$ is natural in the following sense: suppose $\phi_j \colon A_j \to B_j$ are $\ast$-homomorphisms for $j=0,1$ for which $(\phi_j)_\#$ are isomorphisms, and $\psi_A \colon A_0 \to A_1$ and $\psi_B \colon B_0 \to B_1$ are $\ast$-homomorphisms such that $\psi_B \circ \phi_0 = \phi_1 \circ \psi_A$. Let $\psi \colon C_{\phi_0} \to C_{\phi_1}$ be the induced $\ast$-homomorphism between mapping cones. Then $\Kd(\psi_A) \circ \zeta_{\phi_0} = \zeta_{\phi_1} \circ \K_1(\psi)$.
\end{lemma}
\begin{proof}
     We define $\widetilde \zeta_\phi \colon \mathcal U_\infty(\Ped(C_\psi)) \to \Kd(A)$ by
    \begin{equation}\label{eq:zetaphi}
    \widetilde \zeta_\phi(u,\mathbf v) = [u]_{\Delta} + \mathrm{Th}_A^\Delta(\phi_\#^{-1}(\Det_B(\mathbf v)))
    \end{equation}
    which is easily seen to be a homomorphism. If $(a,\mathbf b) \in M_n(\Ped(C_\psi))$ is self-adjoint, then $[e^{2\pi i a}]_{\Delta} = \mathrm{Th}_A^\Delta(\widehat a)$, and since $\mathbf b(0) = \phi^{(n)}(a)$ and $\mathbf b(1)=0$ we get
    \[
    \Det_B(e^{2\pi i \mathbf b}) = \widehat{\mathbf b(1)}  - \widehat{\mathbf b(0)} = -\phi_\#(\widehat a)
    \]
    by Lemma \ref{l:basicdetproperties}. Hence
    \[
    \widetilde \zeta_\phi(e^{2\pi i (a,\mathbf b)}) = \mathrm{Th}_A^\Delta(\widehat a) + \mathrm{Th}_A^\Delta(\phi_\#^{-1} (-\phi_\#(\widehat a))) = 0.
    \]
    By \eqref{eq:K one}, $\widetilde \zeta_\phi$ descends to a
    homomorphism $\zeta_\phi \colon \K_1(C_\phi) \to \Kd(A)$.

    For commutativity of the right square in \eqref{eq:zetacone}, let $(u,\mathbf v) \in \mathcal U_\infty(\Ped(C_\phi))$. Then $[(u,\mathbf v)]_1$ is sent to $[u]_1$ under $\K_1(C_\phi) \to \K_1(A)$, and $\zeta_\phi([(u,\mathbf v)]_1)$ also maps to $[u]_1$ in $\K_1$ as the image of $\Thd_A$ vanishes in $\K_1(A)$.

    For the left square of \eqref{eq:zetacone}, let $x\in \K_0(B)$. Identifying $\K_0(B)$ with $\K_1(SB)$ we find $\mathbf v \in \mathcal U_n(\Ped(SB))$ inducing $x$. Then $[(1_n,\mathbf v)]_1 \in \K_1(C_\phi)$ is the value of $x$ under $\K_0(B) \to \K_1(C_\phi)$. Therefore,
    \[
    \zeta_\phi([(1_n,\mathbf v)]_1) \stackrel{\eqref{eq:zetaphi_statement}}{=} \Thd_A (\phi_\#^{-1}(\Det_A(\mathbf v))) \stackrel{\textrm{Prop.}~\ref{p:pairing}}{=} \Thd_A(\phi_\#^{-1}(\rho_B(x))),
    \]
    and so the left square in \eqref{eq:zetacone} commutes.

    Naturality of $\zeta_\phi$ is immediate from \eqref{eq:zetaphi_statement} since all the maps involved in the definition are natural.
\end{proof}

\begin{example}
    Let $A$ be a $\cs$-algebra and let $\mathcal R$ be the hyperfinite $\mathrm{II}_1$ factor. Let $\iota \colon A \to A\otimes_{\min{}} \mathcal R$ be the obvious inclusion and let $C_\iota$ be the mapping cone. Then the above lemma constructs a canonical homomorphism $\zeta_\iota \colon \K_1(C_\iota) \to \Kd(A)$.

    Assume moreover that $A$ is separable and satisfies the UCT. Then we have $\K_1(C_\iota) = \K_0(A; \mathbb R/\mathbb Z)$ in the sense of \cite{AAS}. Moreover, $\K_0(A\otimes_{\min} \mathcal R) \cong \K_0(A) \otimes \mathbb R$ canonically by the Künneth theorem, and the diagram \eqref{eq:zetacone} becomes
    \[
    \begin{tikzcd}
	{\K_0(A) \otimes \mathbb R} & {\K_0(A; \mathbb R/\mathbb Z)} & {\K_1(A)} \\
	{\T_{\mathbb R}(A)^\ast} & {\Kd(A)} & {\K_1(A)}
	\arrow[from=1-1, to=1-2]
	\arrow["{\rho_A^{\mathbb{R}}}", from=1-1, to=2-1]
	\arrow[from=1-2, to=1-3]
	\arrow["{\zeta_\iota}", from=1-2, to=2-2]
	\arrow["\id", from=1-3, to=2-3]
	\arrow["{\Thd_A}", from=2-1, to=2-2]
	\arrow[two heads, from=2-2, to=2-3]
    \end{tikzcd}
    \]
    where $\rho_A^{\mathbb R} \colon \K_0(A)\otimes \mathbb R \to \T_{\mathbb R}(A)^\ast$ is the obvious $\mathbb R$-linear map induced by $\rho_A$.
\end{example}

We use the following picture of K-theory with $\mathbb Z/n\mathbb Z$-coefficients in terms of the diagonal embedding $\iota_{n,A} := \id \otimes 1_n \colon A \to M_n(A)$, closely matching the picture in \cite{CGSTW}. 
 \begin{definition}
    Let $A$ be a $\cs$-algebra, let $n \in \mathbb N$ with $n \geq 2$ and let $i=0,1$. Consider the mapping cone 
    \[ \mathbb I_{n,A} := C_{\iota_{n,A}} = \{ f \in C([0,1],M_n(A)) \mid f(0) \in A \otimes 1_n, f(1) = 0 \}\]
    of the diagonal embedding $\iota_{n,A} \colon A \to M_n(A)$. The \emph{K-theory with $\mathbb Z/n\mathbb Z$-coefficients} $\K_i(A;\mathbb Z/n\mathbb Z)$ is given by $\K_i(A;\mathbb Z/n\mathbb Z) := \K_{1-i}(\mathbb I_{n,A})$. Given a $*$-homomorphism $\psi \colon A \to B$, the map  $\K_i(\psi;\mathbb Z/n \mathbb Z) \colon \K_i(A;\mathbb Z/n\mathbb Z) \to \K_i(B;\mathbb Z/n\mathbb Z)$ is induced by the induced $*$-homomorphism $(\psi,\psi^{(n)})\colon \mathbb I_{n,A} \to \mathbb I_{n,B}$ of mapping cones.
    
    The mapping cone six-term sequence \eqref{eq:cone six term} produces a natural six-term exact sequence of the following form, after identifying $\K_i(M_n(A))$ with $\K_i(A)$ via the top-corner inclusion $\epsilon_{n,A} \colon A \to M_n(A)$.
    \begin{equation}\label{eq:mod n}
    \begin{tikzcd}
	{\K_0(A)} & {\K_0(A)} & {\K_0(A;\mathbb Z/n\mathbb Z)} \\
	{\K_1(A;\mathbb Z/n\mathbb Z)} & {\K_1(A)} & {\K_1(A)}
	\arrow["{\times n}", from=1-1, to=1-2]
	\arrow["\mu^{(n)}_{0,A}", from=1-2, to=1-3]
	\arrow["\nu^{(n)}_{0,A}", from=1-3, to=2-3]
	\arrow["\nu^{(n)}_{1,A}", from=2-1, to=1-1]
	\arrow["\mu^{(n)}_{1,A}", from=2-2, to=2-1]
	\arrow["{\times n}", from=2-3, to=2-2]
    \end{tikzcd}
    \end{equation}
    The \emph{Bockstein operations/maps} are the natural homomorphisms $\mu^{(n)}_{i,A}$ and $\nu^{(n)}_{i,A}$ in \eqref{eq:mod n}
together with, given further $m \geq 2$, the homomorphisms 
\begin{eqnarray}
& \kappa^{(mn,n)}_{i,A} \colon & \K_i(A;\mathbb Z/n \mathbb Z) \to \K_i(A;\mathbb Z/mn\mathbb Z) \\
& \kappa^{(n,mn)}_{i,A} \colon & \K_i(A;\mathbb Z/mn \mathbb Z)  \to \K_i(A;\mathbb Z/n\mathbb Z)
\end{eqnarray} 
defined as follows. Set $\iota_1 \colon \mathbb I_{n,A} \hookrightarrow \mathbb I_{mn,A}$ to be the inclusion $f \mapsto f \otimes 1_m$ and $\iota_2 \colon \mathbb I_{mn,A} \hookrightarrow M_m(\mathbb I_{n,A})$ to be the inclusion obtained by viewing both algebras as subalgebras of $C([0,1],M_{mn}(A))$ in the canonical way, and consider the top-corner inclusion $\epsilon_{n,A} \colon A \to M_n(A)$. Then $\kappa^{(mn,n)}_{i,A} := \K_1(\iota_1)$ and $\kappa^{(n,mn)}_{i,A} := \K_1(\epsilon_{n,A})^{-1} \circ \K_1(\iota_2)$.  
 \end{definition}

The diagonal embedding $\iota_{n,A} \colon A \to M_n(A)$ induces the $\times n$ map on both $\T_{\mathbb R}(A)^\ast$ and $\Kd(A)$, after we identify $\T_{\mathbb R}(M_n(A))^\ast$ and $\Kd(M_n(A))$ with $\T_{\mathbb R}(A)^\ast$ and $\Kd(A)$ via the top-corner inclusion $\epsilon_{n,A} \colon A \to M_n(A)$. 
In particular, $(\iota_{n,A})_\#$ is an isomorphism.

\begin{definition}
Let $A$ be a $\cs$-algebra and let $n \geq 2$. The \emph{$\zeta$-map}
\[ \zeta_{n,A} \colon \K_0(A;\mathbb Z/n \mathbb Z) \to \Kd(A)\]
is the homomorphism $\zeta_{\iota_{n,A}} \colon \K_1(\mathbb I_{n,A}) \to \Kd(A)$ from Lemma \ref{l:mappingcone} associated to the diagonal map $\iota_{n,A} \colon A \to M_n(A)$. 
Explicitly, for $\mathbf u \in \mathcal U_k(\Ped(\mathbb I_{n,A}))$ with $\mathbf u(0) = u \otimes 1_n$, $u \in \mathcal U_k(\Ped(A))$, we have
\begin{equation}\label{eq:zeta map}
\zeta_n([\mathbf u]_1) := [u]_\Delta + \Th_A^\Delta (\tfrac{1}{n}\Det_{A}(\mathbf u)).
\end{equation}
\end{definition}

Naturality of $\zeta_{n,-}$ with respect to $*$-homomorphisms $\psi \colon A \to B$ follows from the naturality in Lemma \ref{l:mappingcone}.

\begin{proposition}\label{p:zeta maps}
Let $A$ be a $\cs$-algebra and let $m,n \geq 2$. Then the following diagrams commute.
\[\begin{tikzcd}
	{\K_0(A)} & {\K_0(A;\mathbb Z/n\mathbb Z)} & {\K_1(A)} \\
	{\T_{\mathbb R}(A)^\ast} & {\Kd(A)} & {\K_1(A)}
	\arrow["\mu^{(n)}_{0,A}", from=1-1, to=1-2]
	\arrow["{\frac{1}{n}\rho_A}", from=1-1, to=2-1]
	\arrow["\nu^{(n)}_{0,A}", from=1-2, to=1-3]
	\arrow["{\zeta_{n,A}}", from=1-2, to=2-2]
	\arrow[from=1-3, to=2-3, "\id"]
	\arrow[from=2-1, to=2-2, "\Th_A^\Delta"]
	\arrow[from=2-2, to=2-3]
\end{tikzcd}\]
\[\begin{tikzcd}
	{\K_0(A;\mathbb Z / n \mathbb Z)} & {\K_0(A;\mathbb Z / mn \mathbb Z)} & {\K_0(A;\mathbb Z / n \mathbb Z)}  \\
	 {\Kd(A)} & {\Kd(A)} &  {\Kd(A)}
	\arrow["{\kappa_{0,A}^{(n,mn)}}", from=1-2, to=1-3]
	\arrow["{\times m}"', from=2-2, to=2-3]
    \arrow["{\id}"', from=2-1, to=2-2]
	\arrow["{\kappa_{0,A}^{(mn,n)}}", from=1-1, to=1-2]
	\arrow["{\zeta_{n,A}}", from=1-1, to=2-1]
    \arrow["{\zeta_{n,A}}", from=1-3, to=2-3]
	\arrow["{\zeta_{mn,A}}", from=1-2, to=2-2]
\end{tikzcd}\]
\end{proposition}
\begin{proof}
Commutativity of the first diagram is immediate from Lemma \ref{l:mappingcone}. For the second diagram, consider the following commutative diagram
\begin{equation}\label{eq:bockstein compatibility}
\begin{tikzcd}[column sep = 4em]
	A & A & {M_m(A)} & A \\
	{M_n(A)} & {M_{mn}(A)} & {M_{mn}(A)} & {M_n(A),}
	\arrow["\id", from=1-1, to=1-2]
	\arrow["{\iota_{n,A}}"', from=1-1, to=2-1]
	\arrow["{\iota_{m,A}}", from=1-2, to=1-3]
	\arrow["{\iota_{mn,A}}"', from=1-2, to=2-2]
	\arrow["{\iota_{n,M_m(A)}}"', from=1-3, to=2-3]
	\arrow["{\iota_{m,M_n(A)}}", from=2-1, to=2-2]
	\arrow["\id", from=2-2, to=2-3]
    \arrow["\epsilon_{m,A}", from=1-4, to=1-3]	
    \arrow["M_n(\epsilon_{m,A})", from=2-4, to=2-3]
    \arrow["{\iota_{n,A}}"', from=1-4, to=2-4]
\end{tikzcd}
\end{equation}
where $\iota$ represents the diagonal inclusions and $\epsilon$ the top-corner inclusions.
The maps $(\id,\iota_{m,M_n(A)}) \colon \mathbb I_{n,A} \to \mathbb I_{mn,A}$ and $(\iota_{m,A},\id) \colon \mathbb I_{mn,A} \to \mathbb I_{n,M_m(A)} = M_m(\mathbb I_{n,A})$ induced at the level of mapping cones by the first and second squares of \eqref{eq:bockstein compatibility} are precisely the $*$-homomorphisms $\iota_1$ and $\iota_2$ which, by definition, induce $\kappa_{0,A}^{(mn,n)}$ and $\kappa_{0,A}^{(n,mn)}$ in $\K_1$. Explicitly, $\kappa_{0,A}^{(mn,n)} = \K_1(\iota_1)$ and $\kappa_{0,A}^{(n,mn)} = \K_1(\epsilon_{m,\mathbb I_{n,A}})^{-1} \circ \K_1(\iota_2)$. On the other hand, the induced map $(\epsilon_{m,A},(\epsilon_{m,A})^{(n)}) \colon \mathbb I_{n,A} \to \mathbb I_{n,M_m(A)} = M_m(\mathbb I_{n,A})$ is given by $\epsilon_{m,\mathbb I_{n,A}}$, so that $\K_0(\epsilon_{m,A};\mathbb Z / n \mathbb Z) = \K_1(\epsilon_{m,\mathbb I_{n,A}})$.

Naturality of $\zeta_\phi$ in Lemma \ref{l:mappingcone} applied to the first square of \eqref{eq:bockstein compatibility} yields $\zeta_{n,A} = \zeta_{mn,A} \circ \K_1(\iota_1) = \zeta_{mn,A} \circ \kappa_{0,A}^{(mn,n)}$. Now consider the following diagram.
\[\begin{tikzcd}[column sep = 3.5em]
	{\K_0(A;\mathbb Z/mn\mathbb Z)} && {\K_0(A;\mathbb Z/n\mathbb Z)} \\
	& {\K_0(M_m(A);\mathbb Z / n \mathbb Z)} \\
	& {\Kd(M_m(A))} \\
	{\Kd(A)} && {\Kd(A)}
	\arrow["{\kappa_{0,A}^{(n,mn)}}", from=1-1, to=1-3]
	\arrow["{\K_1(\iota_2)}"', from=1-1, to=2-2]
	\arrow["{\zeta_{mn,A}}"', from=1-1, to=4-1]
	\arrow["{\K_0(\epsilon_{m,A};\mathbb Z/n\mathbb Z)}", "\cong"', from=1-3, to=2-2]
	\arrow["{\zeta_{n,A}}"', from=1-3, to=4-3]
	\arrow["{\zeta_{n,M_m(A)}}"', from=2-2, to=3-2]
	\arrow["{\Kd(\iota_{m,A})}", from=4-1, to=3-2]
	\arrow["{\times m}", from=4-1, to=4-3]
	\arrow["{\Kd(\epsilon_{m,A})}"', "\cong", from=4-3, to=3-2]
\end{tikzcd}\]
The left square commutes by naturality of $\zeta$ applied to the second square of \eqref{eq:bockstein compatibility}. The top triangle commutes because $\K_0(\epsilon_{m,A};\mathbb Z / n \mathbb Z) = \K_1(\epsilon_{m,\mathbb I_{n,A}})$, and the bottom triangle commutes by construction. The right square commutes by naturality of $\zeta_{n,-}$. We conclude that $\zeta_{n,A} \circ \kappa_{0,A}^{(n,mn)} = m \zeta_{mn,A}$.
\end{proof}

\section{Comparison of $\Kd$ to $\Ka$}

Let $A$ be a unital $\cs$-algebra. The unitary algebraic $\K_1$ group
$\Ka(A)$ is the abelianisation of $\mathcal U_{\infty}(A)$. We will use the (na\"ive) topology on $\Ka(A)$ induced by the inductive limit topology on $\mathcal U_{\infty}(A) = \varinjlim_{n} \mathcal U_n(A)$ as a topological space. Under the resulting quotient topology $\Ka(A)$ is a semitopological abelian group.\footnote{We warn readers that multiplication need not (a priori) be jointly continuous.} 
For nonunital $A$ with unitisation $A^\dagger$ and scalar map $s \colon A^\dagger \to \mathbb C$, we set $\Ka(A) = \ker(\Ka(s)) \subseteq \Ka(A^\dagger)$.
The Hausdorffisation $\HKa(A) = \Ka(A)/{\overline{\{0\}}}$ is a Hausdorff topological abelian group (see e.g. \cite[Proposition 2.12]{CGSTW}). 

We note that $\Ka$ and its Hausdorffisation are very far from being stable functors (see Proposition \ref{p:K1a not stable}), but for unital $\cs$-algebras $A$ it is reasonable to ask whether the canonical homomorphism $\Ka(A) \to \Kd(A)$ is an isomorphism.

\begin{proposition}\label{p:pure}
Let $A$ be a unital separable simple finite pure $\cs$-algebra whose $2$-quasitraces are traces. Then there is a topological isomorphism
$\Ka(A) \isom \Kd(A)$ sending $[u]_{\mathrm{alg}}$ to $[u]_\Delta$ for $u \in \mathcal U_\infty(A)$.
\end{proposition}
\begin{proof}
The canonical map $\mathcal U_{\infty}(A) \to \Kd(A)$ is continuous by Lemma \ref{l:unitarycontinuity} and surjective, and hence induces a continuous surjective homomorphism $\Ka(A) \to \Kd(A)$. 

To see that this map is injective, suppose for some $u \in \mathcal U_\infty(A)$ there is $n \in \mathbb N$ and a norm-continuous path $\mathbf u \colon [0,1] \to \mathcal U_n(A)$ from $1_n$ to $u$ with $\Det_A(\mathbf u) = 0$. By \cite[Theorem 1.4]{Lin} $A$ has stable rank one, so we may apply
\cite[Theorem 1.1]{NgRobert} to $M_n(A)$, and find that $u$ is in the commutator subgroup of $\mathcal U_n(A)$. Therefore, $u$ also vanishes in $\Ka(A)$.

Corollary \ref{c:ContinuityOutOfK1D} says that continuity of the inverse $\Kd(A) \to \Ka(A)$ follows from continuity of the map $a \mapsto [e^{2 \pi i a}]_{\mathrm{alg}} \colon A_{\mathrm{sa}} \to \Ka(A)$, which is immediate.
\end{proof}

The conditions on $A$ in the above statement are not optimised, but rather give us access to the (stronger than we need) theorem of \cite{NgRobert}. For example, the conclusion also holds for $A = M_n(\mathbb C)$ even though it is not pure (see \cite{Winter-inv} for pureness, as well as \cite{APTV} for recent breakthroughs). Note also that by a theorem of Haagerup \cite{Haagerup}, all 2-quasitraces on unital exact $\cs$-algebras are traces. 

\begin{definition}
We say that a $\cs$-algebra $A$ is \emph{tracially unital} if the set $\T_1(A)$ of tracial states on $A$ is weak$^\ast$-compact and if every positive trace on $\Ped(A)$ is bounded.
\end{definition}

\begin{example}
    Unital $\cs$-algebras are clearly tracially unital (even if they have no tracial states). 
\end{example}

\begin{example} Every simple $\sigma$-unital $\cs$-algebra with continuous scale is tracially unital (see \cite{Lin-contscale1, Lin-contscale2}). This includes the Razak--Jacelon algebra $\mathcal W$ \cite{Razak, Jacelon}. 
\end{example}

\begin{example}
By \cite[Lemma 2.5 and Proposition 2.6]{CastillejosEvington}, every separable simple $\mathcal Z$-stable $\cs$-algebra contains a non-zero hereditary $\cs$-subalgebra which is tracially unital.
\end{example}

For the proof of the following, we let $\T_+(A)$ denote the convex cone of all positive tracial functionals on $\Ped(A)$ equipped with the weak$^\ast$ topology coming from $\Ped(A)$, and $\T_+^{\mathrm b}(A)$ the convex cone of all bounded positive tracial functionals on $A$ equipped with the weak$^\ast$ topology coming from $A$. We let $\Aff_0 \T_+(A)$ and $\Aff_0\T_+^{\mathrm{b}}(A)$ denote the spaces of real continuous affine functions on these cones which vanish at $0$. 
Note that a space of affine functions on an empty trace space is a zero-dimensional vector space.

\begin{proposition}\label{p:tracially unital trace space identification}
Let $A$ be a tracially unital $\cs$-algebra. Then the natural map $\T_{\mathbb R}(A)^\ast \to \Aff \T_1(A)$ is an isomorphism of ordered topological vector spaces. 
\end{proposition}
\begin{proof}
We may factor $\T_{\mathbb R}(A)^\ast \to \Aff \T_1(A)$ as 
\[\T_{\mathbb R}(A)^\ast \to \Aff_0 \T_+(A) \to \Aff_0 \T_+^{\mathrm b}(A) \to \Aff \T_1(A),\] 
where the weak$^\ast$ topology on $\T_+^{\mathrm b}(A)$ and $\T_1(A)$ is taken with respect to $A_{\mathrm{sa}}$ (as opposed to $\Ped(A)_{\mathrm{sa}}$). 
By \cite[Propositions 6.1 and 6.3]{GMtraces}, the first and last maps are isomorphisms. 
It suffices to show that the continuous bijection $\T_+^{\mathrm b}(A) \to \T_+(A)$ is a homeomorphism.
Let $\tau_i \to \tau$ in $\T_+(A)$. 
Note that $\T_1(A)$ is also compact in the weak$^\ast$ topology coming from $\Ped(A)^\ast$. Hence we may apply Hahn--Banach separation to the compact convex sets $\T_1(A), \{0\} \subseteq \T_{\mathbb R}(A)$ to obtain an element $f \in \T_{\mathbb R}(A)^\ast$ and $\epsilon \in \mathbb R$ with 
\[
f(0) = 0 < \epsilon < f(\sigma) , \qquad \textrm{for all }\sigma \in \T_1(A).
\]
Let $e\in \Ped(A)_{\mathrm{sa}}$ such that $\widehat e = f$. Then $\|\sigma\| \leq \tfrac{1}{\epsilon}\sigma(e)$ for all $\sigma \in \T_+(A)$. Since $\tau_i(e) \to \tau(e)$, the net $(\tau_i(e))_i$ is eventually bounded and hence so is $(\tau_i)_i$. The topology on $\T_+(A)$ agrees with the topology on $\T_+^{\mathrm b}(A)$ on bounded sets, so $\tau_i \to \tau$ in $\T_+^{\mathrm b}(A)$ too.
\end{proof}

The following lemma was first proved in the context of another (still ongoing) project \cite{CGSTW2}. 

\begin{lemma}\label{tracially unital exact sequence}
Let $A$ be a tracially unital $\cs$-algebra. Then there is a continuous open homomorphism $\overline \Th_A \colon \Aff \T_1(A)/\overline{\imrho{A}} \to \HKa(A)$ determined by $\overline \Th_A([\widehat a]) = [e^{2 \pi i a}]_{\overline{\mathrm{alg}}}$ for $a \in A_{\mathrm{sa}}$ which makes the following sequence exact.
\[\begin{tikzcd}
	0 & {\Aff \T_1(A)/\overline{\imrho{A}}} & {\HKa(A)} & {\K_1(A)} & 0
	\arrow[from=1-2, to=1-3, "\overline \Th_A"]
	\arrow[from=1-3, to=1-4]
	\arrow[from=1-4, to=1-5]
	\arrow[from=1-1, to=1-2]
\end{tikzcd}\]
\begin{proof}
For a unital $\cs$-algebra $B$, we consider the pairing map $\rho_B \colon \K_0(B) \to \Aff \T_1(B)$ and the Thomsen map $\overline \Th_B \colon \Aff \T_1(B) \to \HKa(B)$ determined by $\overline \Th_B ([\widehat b]) = [e^{2 \pi i b}]_{\overline{\mathrm{alg}}}$ for $b \in B_{\mathrm{sa}}$. The Thomsen map $\overline \Th_B$ is continuous and open, vanishes on $\overline{ \imrho{B}}$ and its image is the kernel of the natural map $\HKa(B) \to \K_1(B)$ (see \cite[\S 2.2]{CGSTW}, cf. \cite{Thomsen-traces}).

Consider the map $\widehat a \mapsto \widehat a \colon \Aff \T_1(A) \to \Aff \T_1(A^\dagger)$, $a \in A_{\mathrm{sa}}$ and consider the commuting diagram below. The second and third columns are exact by definition, and we have established that the second and third rows are exact.
\begin{equation}\label{eq:Aff nine}
\begin{tikzcd}
	& 0 & 0 & 0 \\
	0 & {\Aff \T_1(A)/\overline{\imrho{A}}} & {\HKa(A)} & {\K_1(A)} \\
	0 & {\Aff \T_1(A^\dagger)/ \overline{\imrho{A^\dagger}}} & {\HKa(A^\dagger)} & {\K_1(A^\dagger)} \\
	0 & {\Aff \T_1(\mathbb C)/\overline{\imrho{\mathbb C}}} & {\HKa(\mathbb C)} & {\K_1(\mathbb C)}
	\arrow[from=1-4, to=2-4]
	\arrow[from=2-4, to=3-4]
	\arrow[from=3-4, to=4-4]
	\arrow[from=1-3, to=2-3]
	\arrow[from=2-3, to=3-3]
	\arrow[from=3-3, to=4-3]
	\arrow[from=2-3, to=2-4]
	\arrow[from=3-3, to=3-4]
	\arrow[from=4-3, to=4-4]
	\arrow[from=1-2, to=2-2]
	\arrow[from=2-2, to=3-2]
	\arrow[from=3-2, to=4-2]
	\arrow[from=4-2, to=4-3, "\overline \Th_{\mathbb C}"]
	\arrow[from=3-2, to=3-3, "\overline \Th_{A^\dagger}"]
	\arrow[from=2-2, to=2-3, "\overline \Th_A", dashed]
	\arrow[from=2-1, to=2-2]
	\arrow[from=3-1, to=3-2]
	\arrow[from=4-1, to=4-2]
\end{tikzcd}
\end{equation}
A simple diagram chase shows that $\overline \Th_A$ as described in the statement is well-defined. We claim that the sequence of continuous maps
\begin{equation}\label{eq:Aff SES}
0 \to \Aff \T_1(A) \to \Aff \T_1(A^\dagger) \to \Aff \T_1(\mathbb C) \to 0 \
\end{equation}
is exact. Since every tracial state on $A$ extends to a tracial state on $A^\dagger$, the left map is a topological embedding. For exactness in the middle note that if $a + \lambda 1_A \in A^\dagger_{\mathrm{sa}}$ satisfies $\widehat \lambda = 0 \in \Aff \T_1(\mathbb C)$, then $\lambda = 0$ and hence $(a + \lambda 1_A)^\wedge = \widehat a$. Moreover, \eqref{eq:Aff SES} admits a continuous splitting induced by the inclusion of the unit $\mathbb C \to A^\dagger$.  This induces a topological direct sum $\Aff \T_1(A^\dagger) \cong \Aff \T_1(A) \oplus \Aff \T_1(\mathbb C)$ which identifies $\imrho{A^\dagger}$ with $\imrho{A} \oplus \imrho{\mathbb C}$. It follows that the left column of \eqref{eq:Aff nine} is exact, with its inclusion a topological embedding. The top row is therefore exact by the nine lemma. As $\overline \Th_{A^\dagger}$ is a topological embedding and $\ker (\HKa(A) \to \K_1(A)) \subseteq \HKa(A)$ is open, we may conclude that $\overline \Th_A$ is a topological embedding with open image, hence open and continuous. 
\end{proof}
\end{lemma}

\begin{proposition}\label{p:HK1Dalg}
Let $A$ be a tracially unital $\cs$-algebra. Then there is a topological isomorphism $\HKd(A) \isom \HKa(A)$ given by $[u]_{\overline \Delta} \mapsto [u]_{\overline{\mathrm{alg}}}$ for all $u\in \mathcal U_\infty(\Ped(A))$.
\end{proposition}
\begin{proof}
Consider the following diagram, which has exact rows by Corollary \ref{c:HK1Dexactseq} and Lemma \ref{tracially unital exact sequence}. Moreover, the horizontal maps are continuous and the left vertical map is a topological isomorphism by Proposition \ref{p:tracially unital trace space identification}.
\begin{equation}\label{eq:alg comparison}
\begin{tikzcd}[column sep = 2.2em]
	0 & {\T_{\mathbb R}(A)^*/\overline{\rho_A(\K_0(A))}} & {\HKd(A)} & {\K_1(A)} & 0 \\
	0 & {\Aff \T_1(A)/\overline{\rho_A(\K_0(A))}} & {\HKa(A)} & {\K_1(A)} & 0
	\arrow[from=1-2, to=1-3]
	\arrow[from=1-3, to=1-4]
	\arrow[from=2-3, to=2-4]
	\arrow[from=2-2, to=2-3]
	\arrow["\cong"', from=1-2, to=2-2]
	\arrow["\id"', from=1-4, to=2-4]
	\arrow[from=2-4, to=2-5]
	\arrow[from=1-4, to=1-5]
	\arrow[from=2-1, to=2-2]
	\arrow[from=1-1, to=1-2]
	\arrow["\alpha"', shift right, dashed, from=1-3, to=2-3]
	\arrow["\beta"', shift right, dashed, from=2-3, to=1-3]
\end{tikzcd}
\end{equation}
Any homomorphism $\alpha \colon \HKd(A) \to \HKa(A)$ or $\beta \colon \HKa(A) \to \HKd(A)$ which makes \eqref{eq:alg comparison} commute must be an isomorphism by the five lemma, and is automatically continuous, as $\K_1(A)$ is discrete and the inclusion maps are open by Proposition \ref{p:Thomsen map} and Lemma \ref{tracially unital exact sequence}. 

First assume $A$ is unital. 
Since $\HKd(A)$ is a Hausdorff topological abelian group, the continuous homomorphism $\mathcal U_\infty(A) \to \HKd(A)$ descends to a homomorphism $\beta_A \colon \HKa(A) \to \HKd(A)$, which must then be an isomorphism. We note that $\beta_A$ is natural in $A$ with respect to unital $*$-homomorphisms.

For general tracially unital $A$, consider the following diagram.
\[\begin{tikzcd}
	& {\HKd(A)} & {\HKd(A^\dagger)} & {\HKd(\mathbb C)} \\
	0 & {\HKa(A)} & {\HKa(A^\dagger)} & {\HKa(\mathbb C)} & 0
	\arrow[from=1-2, to=1-3]
	\arrow["\alpha"', dashed, from=1-2, to=2-2]
	\arrow[from=1-3, to=1-4]
	\arrow["\beta_{A^\dagger}^{-1}", from=1-3, to=2-3]
	\arrow["\beta_{\mathbb C}^{-1}", from=1-4, to=2-4]
	\arrow[from=2-1, to=2-2]
	\arrow[from=2-2, to=2-3]
	\arrow[from=2-3, to=2-4]
	\arrow[from=2-4, to=2-5]
\end{tikzcd}\]
By naturality of $\beta$ and exactness of the lower row, there is a homomorphism $\alpha \colon \HKa(A) \to \HKa(A)$ which satisfies $\alpha([u]_{\overline{\Delta}}) = [u]_{\overline{\mathrm{alg}}}$ for each $u \in \mathcal U_{\infty}(\Ped(A))$. This forces the diagram \eqref{eq:alg comparison} to commute, and thus $\alpha$ is a topological isomorphism.
\end{proof}

However, some form of unitality is indeed essential for relating $\HKd(A)$ to $\HKa(A)$, as the latter invariant is not stable:
\begin{proposition}\label{p:K1a not stable}
For any $\cs$-algebra $A$, the natural map $\HKa(A \otimes \mathcal K) \to \K_1(A \otimes \mathcal K)$ is an isomorphism. 
\end{proposition}
\begin{proof}
Set $B := (A \otimes \mathcal K)^\dagger$. The scalar map $s \colon B \to \mathbb C$ is the unique tracial state on $B$, so $\rho_B(x) = \K_0(s)(x) \in \mathbb Z \subseteq \mathbb R = \Aff \T_1(B)$ for each $x \in \K_0(B)$. Thus $\Aff \T_1(B)/\imrho{B} = \mathbb R/\mathbb Z$, and the following diagram commutes. 
\[\begin{tikzcd}[row sep = small]
	&& 0 & 0 & \\
	& 0 & {\HKa(A \otimes \mathcal K)} & {\K_1(A \otimes \mathcal K)} & 0 \\
	0 & {\mathbb R/\mathbb Z} & {\HKa(B)} & {\K_1(B)} & 0 \\
	0 & {\mathbb R/\mathbb Z} & {\HKa(\mathbb C)} & 0 \\
	& 0 & 0
	\arrow[from=1-3, to=2-3]
	\arrow[from=1-4, to=2-4]
	\arrow[from=2-2, to=2-3]
	\arrow[from=2-2, to=3-2]
	\arrow[from=2-3, to=2-4]
	\arrow[from=2-3, to=3-3]
	\arrow[from=2-4, to=2-5]
	\arrow[from=2-4, to=3-4]
	\arrow[from=3-1, to=3-2]
	\arrow["{\overline \Th_B}", from=3-2, to=3-3]
	\arrow["\id", from=3-2, to=4-2]
	\arrow[from=3-3, to=3-4]
	\arrow[from=3-3, to=4-3]
	\arrow[from=3-4, to=3-5]
	\arrow[from=3-4, to=4-4]
	\arrow[from=4-1, to=4-2]
	\arrow["{\overline \Th_{\mathbb C}}", from=4-2, to=4-3]
	\arrow[from=4-2, to=5-2]
	\arrow[from=4-3, to=4-4]
	\arrow[from=4-3, to=5-3]
\end{tikzcd}\]
All columns are exact, and the lower two rows are exact. An easy diagram chase (or the nine lemma) shows that the desired map is an isomorphism.
\end{proof}

\section{Classification of stable $\cs$-algebras containing a projection}

The functor $\HKd$ is one part of a larger functor $\TKT$ incorporating K-theory and traces, which can be used to classify $*$-homomorphisms between certain stable $\cs$-algebras.

\begin{definition}\label{d:total}
Let $A$ be a $\cs$-algebra. The \emph{total stable invariant} $\TKT(A)$ consists of the following data:
\begin{enumerate}
\item The K-theory groups $\K_i(A)$ for $i=0,1$, as abelian groups.
\item The dual $\T_{\mathbb R}(A)^\ast$ of the space of self-adjoint traces, taken with the compact-open topology, as an ordered real topological vector space.
\item $\HKd(A)$ taken as a topological abelian group.
\item The K-theory with coefficients groups $\K_i(A;\mathbb Z/n\mathbb Z)$ for $i=0,1$ and $n \geq 2$, as abelian groups.
\item The following homomorphisms: 
\begin{equation}\label{eq:total homomorphisms}
\begin{tikzcd}[row sep = 0em, column sep = 1.3em]
	{\text{the pairing map}} & {\K_0(A)} & {\T_{\mathbb R}(A)^*,} \\
	{\text{the Thomsen map}} & {\T_{\mathbb R}(A)^\ast} & {\HKd(A),} \\
	{\text{the obvious map}} & {\HKd(A)} & {\K_1(A),} \\
	{\text{the $\zeta$-maps (for $n \geq 2$)}} & {\K_0(A;\mathbb Z/n \mathbb Z)} & {\HKd(A),} \\
	{\text{and the Bockstein maps}} & {\K_i(A)} & {\K_i(A;\mathbb Z/n\mathbb Z),} \\
	{\text{(for $i=0,1$ and $m,n \geq 2$)}} & {\K_i(A;\mathbb Z/n\mathbb Z)} & {\K_{1-i}(A),} \\
	& {\K_i(A;\mathbb Z/n\mathbb Z)} & {\K_i(A;\mathbb Z/mn\mathbb Z),} \\
	& {\K_i(A;\mathbb Z/mn\mathbb Z)} & {\K_i(A;\mathbb Z/n\mathbb Z).}
	\arrow["{\rho_A}", from=1-2, to=1-3]
	\arrow["{\overline\Th_A^\Delta}", from=2-2, to=2-3]
	\arrow["{\mathrm{obv}_A}", from=3-2, to=3-3]
	\arrow["{\zeta_n}", from=4-2, to=4-3]
	\arrow["{\mu^{(n)}_{i,A}}", from=5-2, to=5-3]
	\arrow["{\nu^{(n)}_{i,A}}", from=6-2, to=6-3]
	\arrow["{\kappa_{i,A}^{(mn,n)}}", from=7-2, to=7-3]
	\arrow["{\kappa_{i,A}^{(n,mn)}}", from=8-2, to=8-3]
\end{tikzcd}
\end{equation}
\end{enumerate}
A morphism $(\underline \alpha,\beta,\gamma) \colon \TKT(A) \to \TKT(B)$ consists of:
\begin{enumerate}
\item\label{eq:hom K theory} group homomorphisms $\alpha_i \colon \K_i(A) \to \K_i(B)$ for $i=0,1$,
\item a continuous positive linear map $\gamma \colon \T_{\mathbb R}(A)^\ast \to \T_{\mathbb R}(B)^\ast$,
\item a continuous\footnote{Note that continuity of $\beta$ in this definition is automatic from continuity of $\gamma$ because $\K_1$ is discrete and the inclusion map in the exact sequence \eqref{eq:ThomsenODelta} is open by Proposition \ref{p:Thomsen map}.} group homomorphism $\beta \colon \HKd(A) \to \HKd(B)$, and
\item\label{eq:hom K theory with coefficients} group homomorphisms $\alpha_i^{(n)} \colon \K_i(A;\mathbb Z/n \mathbb Z) \to \K_i(B;\mathbb Z/n \mathbb Z)$ for $i=0,1$ and $n \geq 2$,
\end{enumerate}
intertwining the homomorphisms \eqref{eq:total homomorphisms}. 
\end{definition}

The total stable invariant is natural with respect to $*$-homomorphisms. It is stable in the following sense:
\begin{proposition}
A full hereditary inclusion $A \subseteq B$ of $\cs$-algebras induces an isomorphism $\TKT(A) \isom \TKT(B)$.
\end{proposition}
\begin{proof}
It is well-known that full hereditary inclusions induce isomorphisms in K-theory and therefore in K-theory with coefficients. For $\mathrm{T}_{\mathbb R}(-)^\ast$ and $\HKd$ this is \cite[Corollary 5.4]{GMtraces} and Proposition \ref{p:HK1Dfullher} respectively.
\end{proof}

\begin{remark} The collection of K-theory groups and K-theory with coeffic\-ients groups is known as the \emph{total K-theory} $\underline \K(A)$. A morphism $\underline \alpha \colon \underline \K(A) \to \underline \K(B)$ is a collection of homomorphisms \eqref{eq:hom K theory} and \eqref{eq:hom K theory with coefficients} intertwining the Bockstein maps.
\end{remark}

For a unital $\cs$-algebra $A$, the \emph{total unital invariant} $\TKTu(A)$ consists of $\TKT(A)$ and additionally the class $[1]_0 \in \K_0(A)$ of the unit. Morphisms in the target category of the total unital invariant are required to preserve $[1]_0$; the total unital invariant is thus natural with respect to unital $*$-homomorphisms. 

\begin{remark}
For a unital $\cs$-algebra $A$, the total unital invariant here contains the same data as the total unital invariant in \cite{CGSTW} via the natural isomorphisms $\T_{\mathbb R}(A)^\ast \isom \Aff \T_1(A)$ and $\HKd(A) \isom \HKa(A)$ from Propositions \ref{p:tracially unital trace space identification} and \ref{p:HK1Dalg} respectively, which intertwine all the relevant homomorphisms. 
\end{remark}

The total stable invariant is invariant under approximate unitary equivalence.

A $*$-homomorphism $\phi \colon A \to B$ is said to be \emph{full} if $\phi(a)$ generates $B$ as a two-sided closed ideal for each non-zero $a \in A$ (for simple $B$ this just means that $\phi$ is injective). If an embedding $\phi \colon A \to B$ is full, then the morphism $\TKT(\phi) \colon \TKT(A) \to \TKT(B)$ it induces is faithful in the sense that it maps positive traces on $B$ to faithful traces on $A$. The following classification theorem for unital embeddings is Theorem 9.3 of \cite{CGSTW}.

\begin{theorem}[Classification of unital embeddings]\label{t:unital classification}
Let $A$ and $B$ be unital separable nuclear $\cs$-algebras such that $A$ satisfies the UCT, and $B$ is simple and $\mathcal Z$-stable. Then every morphism $\Phi \colon \TKTu(A) \to \TKTu(B)$ of the total unital invariant mapping positive traces on $B$ to faithful traces on $A$ is induced by a unital embedding $\phi \colon A \to B$, and moreover $\phi$ is unique up to approximate unitary equivalence.
\end{theorem}

It is well known that a $\sigma$-unital stable $\cs$-algebra $A$ with a full projection $p \in A$ is isomorphic to $pAp \otimes \mathcal K$ by Brown's theorem \cite[Theorem 2.8]{Brown77}. We will use a slightly sharper version of this result to convert the above classification of unital $*$-homomorphisms to a classification of $*$-homomorphisms between stable $\cs$-algebras that contain projections.

\begin{lemma}\label{l:stabilisation trick}
Let $A$ be a $\sigma$-unital stable $\cs$-algebra with a full projection $p \in A$. Then there is an isomorphism $\theta \colon A \isom pAp \otimes \mathcal K$ whose restriction to $pAp$ is the top-corner inclusion $pAp \to pAp \otimes \mathcal K$.
\end{lemma}
\begin{proof}
As $A$ is stable, the trivial Hilbert $A$-module $A$ is isomorphic to its infinite repeat $A^{\oplus \infty}$ \cite[Lemma 1.3.2]{JensenThomsenKK}. 
Set $B = (1-p)A(1-p)$. This generates $A$ as an ideal because by stability we may pick isometries $v_n \in \multialg{A}$ with $v_n v_n^* \to 0$ strictly, so that $v^*_n p v_n \to 0$ by the $\cs$-identity, hence for any $a \in A$ the sequence $v_n(1-p)v_n^* a v_n (1-p) v_n^* \in \overline{ABA}$ converges to $a$. Moreover, $B$ is stable by \cite[Corollary 4.3]{HjelmborgRordam}.\footnote{In our setting we do not need $A$ to be separable as $B$ inherits $\sigma$-unitality from $A$.} Thus \cite[Theorem 4.23]{Brown88} produces a Murray--von Neumann equivalence $1-p \sim 1 \in \multialg{A}$ from which we obtain an isomorphism of Hilbert $A$-modules $(1-p)A \cong A$. By \cite[Theorem 1.9]{MingoPhillips84}, the Hilbert module $(pA)^{\oplus \infty}$ is isomorphic to $A^{\oplus \infty}$, so we arrive at an isomorphism $u \colon (1-p)A \to (pA)^{\oplus \infty}$ by chaining together the isomorphisms $(1-p)A \cong A \cong A^{\oplus \infty} \cong (pA)^{\oplus \infty}$. We obtain an isomorphism 
\begin{equation}
\id \oplus u \colon pA \oplus (1-p)A \to pA \oplus (pA)^{\oplus \infty}
\end{equation}
of Hilbert $A$-modules. After identifying $\mathcal K(pA \oplus (1-p)A)$ with $A$ and $\mathcal K(pA \oplus (pA)^{\oplus \infty})$ with $pAp \otimes \mathcal K$, the induced isomorphism $\theta = \Ad (\id \oplus u) \colon A \to pAp \otimes \mathcal K$ satisfies $\theta|_{pAp} = \id \otimes e$ by construction, where $e \in \mathcal K$ is the top-corner minimal projection. 
\end{proof}

\begin{theorem}\label{t:stable classification}
Let $A$ and $B$ be stable separable nuclear $\cs$-algebras such that $A$ satisfies the UCT and contains a full projection, and $B$ is simple and $\mathcal Z$-stable. Then every morphism $\Phi \colon \TKT(A) \to \TKT(B)$ of the total stable invariant which maps traces on $B$ to faithful traces on $A$ is induced by an embedding $\phi \colon A \to B$, and moreover $\phi$ is unique up to approximate unitary equivalence.
\end{theorem}
\begin{proof}
Fix a full projection $p$ in $A$. 

Let $\Phi = (\underline \alpha, \beta, \gamma) \colon \TKT(A) \to \TKT(B)$ be a morphism of total stable invariants which sends traces on $B$ to faithful traces on $A$. 
Since 
\[
\rho_B(\alpha_0([p]_0))(\tau_B) = \rho_A([p]_0)(\gamma^\ast(\tau_B)) > 0
\]
for all $\tau_B\in \T_+(B)$ by faithfulness of $\gamma^\ast(\tau_B) \in \T_+(A)$, it follows from \cite[Corollary 4.10]{ROrdamZstable} using $\mathcal Z$-stability, nuclearity, stability and simplicity of $B$ that there is a (necessarily full) projection $q \in B$ such that $\alpha_0([p]_0) = [q]_0$. Then $q B q$ is a full hereditary subalgebra, so $\Phi$ descends to a morphism $\Phi_p \colon \TKT(pAp) \to \TKT(qBq)$ by stability of $\TKT$. This is compatible with the unit classes $[p]_0$ and $[q]_0$ by construction, and by Theorem \ref{t:unital classification} there is a unital embedding $\phi_p \colon pAp \to qBq$ with $\Phi_p = \TKT(\phi_p)$. By Lemma \ref{l:stabilisation trick}, we pick isomorphisms $\theta_A \colon A \to p A p  \otimes \mathcal K$ and $\theta_B \colon B \to q B q \otimes \mathcal K$ that restrict to the top-corner inclusions $p A p \to p A p \otimes \mathcal K$ and $qBq \to qBq \otimes \mathcal K$, making the following diagram commute.
\[\begin{tikzcd}
	& {pAp}  & qBq \\
	{A} & {pAp \otimes \mathcal K} & {qBq \otimes \mathcal K} & B
	\arrow["{\phi_p}", from=1-2, to=1-3]
	\arrow[from=1-2, to=2-2]
	\arrow[from=1-3, to=2-3]
	\arrow[from=1-3, to=2-4]
    \arrow[from=1-2, to=2-1]
    \arrow["{\theta_A}"', from=2-1, to=2-2]
	\arrow["{\phi_p \otimes \id}"', from=2-2, to=2-3]
	\arrow["\theta^{-1}_B"', from=2-3, to=2-4]
\end{tikzcd}\]
By stability of $\TKT$ we get that $\TKT(\theta^{-1}_B \circ (\phi_p \otimes \id) \circ \theta_A) = \Phi$, so the embedding $\phi = \theta^{-1}_B \circ (\phi_p \otimes \id) \circ \theta_A$ induces $\Phi$.

We turn to uniqueness. For $n \in \mathbb N$ we set $A_n = \theta_A^{-1}(pAp \otimes M_n)$ so that $A = \overline{\bigcup_{n \geq 1} A_n}$. 
We write $1_n \in A_n$ for the unit in $A_n$.
Consider two embeddings $\phi, \psi \colon A \hookrightarrow B$ which agree on the total stable invariant $\TKT$. For each $n \in \mathbb N$ set $\phi_n,\psi_n \colon A_n \to B$ to be the restrictions of $\phi$ and $\psi$ to $A_n$. By a straightforward argument, it suffices to prove that $\phi_n \sim_{\mathrm{au}} \psi_n$ for each $n$ to obtain an approximate unitary equivalence $\phi \sim_{\mathrm{au}} \psi$. Set $e_n = \phi_n(1_n)$ and $f_n = \psi_n(1_n)$. 

Since $A_n \hookrightarrow A$ is a full hereditary $\cs$-subalgebra, we have $\TKT(\phi_n) = \TKT(\psi_n)$ by stability of $\TKT$. The full projections $e_n,f_n \in B$ have the same $\K_0$-class and $B$ is $\mathcal Z$-stable, so we apply \cite[Theorem 3.5]{Jiang97} to find $v_n \in B$ satisfying $v_n v_n^\ast = e_n$ and $v_n^\ast v_n = f_n$. Then by construction and stability of $\TKT$, the unital $*$-homomorphisms $\phi_n, \Ad v_n \circ \psi_n \colon A_n \to e_n B e_n$ agree on $\TKT$ and therefore on $\TKTu$. By Theorem \ref{t:unital classification} they are approximately unitarily equivalent, and thus, by extending unitaries $u\in \mathcal U(e_n B e_n)$ to $u + 1-e_n \in \mathcal U(\mathcal M(B))$, we obtain approximate unitary equivalences $\phi_n \sim_{\mathrm{au}} \Ad v_n \circ \psi_n \colon A_n \to B$. Stability of $B$ implies that $\Ad v_n \circ \psi_n \sim_{\mathrm{au}} \psi_n \colon A_n \to B$ for each $n$ by \cite[Proposition 3.13]{Gabe-o2class}, and we conclude that $\phi$ and $\psi$ are approximately unitarily equivalent.
\end{proof}

As in \cite{CGSTW}, classification of $*$-homomorphisms implies classification of $\cs$-algebras via Elliott intertwining. The \emph{stable Elliott invariant} $\KT(A)$ of a $\cs$-algebra $A$ is given by the tuple
\[ \KT(A) = (\K_0(A),\K_1(A),\T_{\mathbb R}(A)^\ast, \rho_A \colon \K_0(A) \to \T_{\mathbb R}(A)^\ast). \]
We consider $\K_0(A)$ as an abelian group with no order structure encoded.

A \emph{morphism} $(\alpha_0,\alpha_1,\gamma) \colon \KT(A) \to \KT(B)$ of stable Elliott invariants consists of homomorphisms $\alpha_i \colon \K_i(A) \to \K_i(B)$ for $i = 0,1$ and a continuous positive linear map $\gamma \colon \T_{\mathbb R}(A)^\ast \to \T_{\mathbb R}(B)^\ast$ such that $\rho_B \circ \alpha_0 = \gamma \circ \rho_A$. The following lemma has the exact same proof as the analogous result \cite[Theorem 3.9]{CGSTW} for the unital invariants. We omit the details.

\begin{lemma}
Let $A$ and $B$ be $\cs$-algebras. Then any morphism \[(\alpha_0,\alpha_1,\gamma) \colon \KT(A) \to \KT(B)\] of stable Elliott invariants 
extends to a morphism $(\underline \alpha, \beta, \gamma) \colon \TKT(A) \to \TKT(B)$. 
Moreover, if $(\alpha_0,\alpha_1, \gamma)$ is an isomorphism, then any such extension $(\underline \alpha, \beta, \gamma)$ is an isomorphism.
\end{lemma}

Arguing almost exactly as in \cite[Section 9.3]{CGSTW}, we obtain the following result. The idea is to use Theorem \ref{t:stable classification} (in place of Theorem \ref{t:unital classification}) together with a two-sided intertwining argument to prove classification by $\TKT$ (part \ref{enum.b} below), and then use the above lemma to prove classification by $\KT$ (part \ref{enum.a} below). 

The only subtlety is that when using two-sided intertwining in the non-unital case, there is a gap in the standard reference \cite[Corollary 2.3.4]{RordamClassBook}, see \cite[Remark 6.2]{Gabe-Oinfty}.\footnote{One can modify the proof slightly (\cite[Proposition 6.3]{Gabe-Oinfty}) and still obtain an isomorphism, but it is a priori not possible to arrange for the constructed isomorphism to be approximately unitarily equivalent to the given embedding $A \to B$, they will only be approximately Murray--von Neumann equivalent.} However, combining \cite[Proposition 6.3]{Gabe-Oinfty} and \cite[Proposition 3.13]{Gabe-o2class} yields:

\begin{proposition}[Two-sided intertwining]
Let A and B be stable separable $\cs$-algebras, and let $\phi \colon A \to B$ and $\psi \colon A \to B$ be $*$-homomorphisms. Suppose that $\psi \circ \phi$ is approximately unitarily equivalent to $\id_A$ and $\phi \circ \psi$ is approximately unitarily equivalent to $\id_B$. Then $\phi$ is approximately unitarily equivalent to an isomorphism between A and B.
\end{proposition}

With that in place the following is proved exactly as in \cite[Section 9.3]{CGSTW}. We omit the details.

\begin{theorem}
Let $A$ and $B$ be stable simple separable nuclear $\mathcal Z$-stable $\cs$-algebras which satisfy the UCT and contain non-zero projections. Then 
\begin{enumerate}[label=(\alph*)]
\item\label{enum.a} every isomorphism $\KT(A) \xrightarrow \cong \KT(B)$ lifts to a $\ast$-isomorphism $A \xrightarrow \cong B$;
\item\label{enum.b} every isomorphism $\TKT(A) \xrightarrow \cong \TKT(B)$ lifts to a $\ast$-isomorphism $A \xrightarrow \cong B$, which is unique up to approximate unitary equivalence. 
\end{enumerate}
\end{theorem}

In the same way that Theorem \ref{t:unital classification} can be used to compute groups of automorphisms of a unital $\cs$-algebra modulo approximate unitary equivalence, Theorem \ref{t:stable classification} can be used to compute groups of stable automorphisms of a unital $\cs$-algebra modulo approximate unitary equivalence. Let $A$ be a unital $\cs$-algebra, and let $\alpha \in \Aut(A)$ be an automorphism. This induces a stable isomorphism $\alpha \otimes \id \in \Aut(A \otimes \mathcal K)$. If $\alpha$ is approximately inner, then so is $\alpha \otimes \id$, thus we obtain a homomorphism
\begin{equation}\label{eq:outer automorphisms}
[\alpha] \mapsto [\alpha \otimes \id] \colon \Aut(A) / \overline{\Inn(A)} \to \Aut(A \otimes \mathcal K)/ \overline{\Inn(A \otimes \mathcal K)}. 
\end{equation}
The relationship between the total stable invariant and the total unital invariant induces an embedding $\Aut(\TKTu(A)) \hookrightarrow \Aut(\TKT(A))$ for each unital $\cs$-algebra $A$. Stability of $\TKT$ induces a canonical isomorphism $\Aut(\TKT(A)) \isom \Aut(\TKT(A \otimes \mathcal K))$ for each $\cs$-algebra $A$.

If $A$ is unital, separable, simple, nuclear, $\mathcal Z$-stable and satisfies the UCT, Theorem \ref{t:unital classification} implies that the group homomorphism
\[ [\alpha] \mapsto \TKTu(\alpha) \colon \Aut(A) / \overline{\Inn(A)} \to \Aut(\TKTu(A)) \]
is an isomorphism (see also \cite[Corollary 9.10]{CGSTW}). Similarly, Theorem \ref{t:stable classification} implies that
\[ [\alpha] \mapsto \TKT(\alpha) \colon \Aut(A \otimes \mathcal K)/ \overline{\Inn(A \otimes \mathcal K)} \to \Aut(\TKT(A \otimes \mathcal K)) \]
is an isomorphism, noting that moreover $\Aut(\TKT(A \otimes \mathcal K)) \cong \Aut(\TKT(A))$.

\begin{example}\label{e:Aut1}
Consider the unital $\cs$-algebra $A = \left( \mathcal Z^{\otimes \mathbb Z} \right) \rtimes \mathbb Z$ examined in \cite[Example 9.11]{CGSTW}. This has invariants
\[\begin{tikzcd}[column sep = 1.85em, row sep = 1em]
	0 & {\ker \rho_A} & {\K_0(A)} & {\T_{\mathbb R}(A)^\ast} & {\Kd(A)} & {\K_1(A)} & 0 \\
	0 & 0 & {\mathbb Z} & {\mathbb R} & {\mathbb T \oplus \mathbb Z} & {\mathbb Z} & 0
	\arrow[from=1-1, to=1-2]
	\arrow[from=1-2, to=1-3]
	\arrow["{\rho_A}", from=1-3, to=1-4]
	\arrow["{\Th_A^\Delta}", from=1-4, to=1-5]
	\arrow[from=1-5, to=1-6]
	\arrow[from=1-6, to=1-7]
	\arrow[from=2-1, to=2-2]
	\arrow[from=2-2, to=2-3]
	\arrow[from=2-3, to=2-4]
	\arrow[from=2-4, to=2-5]
	\arrow[from=2-5, to=2-6]
	\arrow[from=2-6, to=2-7]
\end{tikzcd}\]
with $[1_A]_0 = 1$ in $\K_0(A) = \mathbb Z$ and positive part $[0,\infty)$ in $\T_{\mathbb R}(A)^\ast = \mathbb R$. Note that $\Kd(A)$ is already Hausdorff. Let $\alpha \in \Aut(A \otimes \mathcal K)$ be a stable automorphism of $A$. Then $\alpha$ must preserve the order on $\K_0(A) = \mathbb Z$ induced through $\rho_A$ by the ordered real vector space $\T_{\mathbb R}(A)^\ast = \mathbb R$. It must therefore fix $\K_0(A)$, and in particular the class $[1_A]_0 = 1 \in \mathbb Z$ of the unit. The embedding $\Aut(\TKTu(A)) \hookrightarrow \Aut(\TKT(A))$ is therefore surjective, and so $\Aut(A \otimes \mathcal K)/ \overline{\Inn(A \otimes \mathcal K)} \cong \Aut(\TKTu(A))$. We refer to \cite[Example 9.11]{CGSTW} for a computation of $\Aut(\TKTu(A))$ as $\mathbb T \rtimes \mathbb Z/2\mathbb Z$. 
\end{example}

Let us consider an example which has stable automorphisms that scale the traces, and therefore do not just come from automorphisms.

\begin{example}\label{e:Aut2}
Set $A$ to be the unital separable simple nuclear $\mathcal Z$-stable $\cs$-algebra satisfying the UCT with invariants 
\[\begin{tikzcd}[column sep = 1.4em, row sep = 1em]
	0 & {\ker \rho_A} & {\K_0(A)} & {\T_{\mathbb R}(A)^\ast} & {\HKd(A)} & {\K_1(A)} & 0 \\
	0 & 0 & {\mathbb Z[\tfrac{1}{2}]} & {\mathbb R^2} & {\mathbb R \oplus \mathbb Z} & {\mathbb Z} & 0,
	\arrow[from=1-1, to=1-2]
	\arrow[from=1-2, to=1-3]
	\arrow["{\rho_A}", from=1-3, to=1-4]
	\arrow["{\overline \Th_A^\Delta}", from=1-4, to=1-5]
	\arrow[from=1-5, to=1-6]
	\arrow[from=1-6, to=1-7]
	\arrow[from=2-1, to=2-2]
	\arrow[from=2-2, to=2-3]
	\arrow["{\mathrm{diag}}", from=2-3, to=2-4]
	\arrow[from=2-4, to=2-5]
	\arrow[from=2-5, to=2-6]
	\arrow[from=2-6, to=2-7]
\end{tikzcd}\]
where $[1_A]_0 = 1 \in \mathbb Z[\tfrac{1}{2}] = \K_0(A)$ and $\T_{\mathbb R}(A)^\ast = \mathbb R^2$ has positive part $[0,\infty)^2$ (this algebra exists by a range-of-invariant result \cite[Remark 2.5]{CGSTW}, cf. \cite{RordamRange, ElliottRange}). Here we are identifying the quotient $\mathbb R^2/\mathrm{diag}(\mathbb R)$ appearing as a direct summand of $\HKd(A)$ with $\mathbb R$ via the map $\mathbb R^2 \to \mathbb R \colon (t,s) \mapsto t-s$.

Let $\alpha \in \Aut(A \otimes \mathcal K)$ be a stable automorphism of $A$. We first argue that the action of $\alpha$ on $\TKT(A)$ is determined by its action on $\HKd(A)$. The order-preserving isomorphism of $\K_0(A) = \mathbb Z[\tfrac{1}{2}]$ must be $\K_0(\alpha) = 2^k$ for some $k \in \mathbb Z$. The induced automorphism $\alpha_\#$ of the ordered real vector space $\T_{\mathbb R}(A)^\ast = \mathbb R^2$ must be $\mathrm{flip}^{\epsilon} \circ 2^k$ for some $\epsilon \in \mathbb Z/2\mathbb Z$, i.e. either $(s,t) \mapsto (2^k s,2^k t)$ or $(s,t) \mapsto (2^k t,2^k s)$. The action on $\K_1(A) = \mathbb Z$ must be $(-1)^\eta$ for some $\eta \in \mathbb Z/2\mathbb Z$. Since the K-theory is torsion-free, the boundary maps of \eqref{eq:mod n} must vanish, which implies that the natural maps $\K_i(A) \to \K_i(A;\mathbb Z/n\mathbb Z)$ induce natural isomorphisms $\K_i(A)/n\K_i(A) \isom \K_i(A;\mathbb Z/n\mathbb Z)$, and so the action on $\K_i(A;\mathbb Z/n\mathbb Z)$ is determined by the action on $\K_i(A)$. The action on $\HKd(A) = \mathbb R \oplus \mathbb Z$ is of the form
\begin{equation}\label{example K1DeltaBar automorphism}
\HKd(\alpha) = \begin{bmatrix}  (-1)^\epsilon 2^k & \theta \\ 0 & (-1)^\eta \end{bmatrix} \in \Aut(\mathbb R \oplus \mathbb Z)
\end{equation}
for some $\theta \in \mathbb R$. Note that $\HKd(\alpha)$ determines both $\epsilon, \eta \in \mathbb Z/2\mathbb Z$ and $k \in \mathbb Z$, and thus the entirety of $\TKT(\alpha)$. Note also that $\HKd(\alpha)$ is automatically compatible with $\overline \zeta_n \colon \K_0(A;\mathbb Z/n\mathbb Z) \to \HKd(A)$ by surjectivity of $\K_0(A) \to \K_0(A;\mathbb Z/n\mathbb Z)$ and Proposition \ref{p:zeta maps}.

For any $k \in \mathbb Z$, $\epsilon,\eta \in \mathbb Z/2\mathbb Z$ and $\theta \in \mathbb R$, the automorphisms 
\begin{itemize}
\item $2^k$ on $\K_0(A) = \mathbb Z[\tfrac{1}{2}]$ and on $\K_0(A;\mathbb Z/n\mathbb Z)$, 
\item $\mathrm{flip}^{\epsilon} \circ 2^k$ on $\T_{\mathbb R}(A)^\ast = \mathbb R^2$, 
\item $(-1)^\eta$ on $\K_1(A) = \mathbb Z$ and $\K_1(A;\mathbb Z/n\mathbb Z)$, 
\item and $\begin{bmatrix}  (-1)^\epsilon 2^k & \theta \\ 0 & (-1)^\eta \end{bmatrix}$ on $\HKd(A) = \mathbb R \oplus \mathbb Z$
\end{itemize}
are compatible with all the structure maps of $\TKT(A)$ and thus form an automorphism of $\TKT(A)$. 
Thus $\Aut(A \otimes \mathcal K)/\overline{\Inn(A \otimes \mathcal K)}$ is in bijection with $\mathbb Z \times \mathbb R \times (\mathbb Z/2\mathbb Z)^2$. Since 
\begin{eqnarray}
& \begin{bmatrix}  (-1)^{\epsilon_1} 2^{k_1} & \theta_1 \\ 0 & (-1)^{\eta_1} \end{bmatrix} & \begin{bmatrix}  (-1)^{\epsilon_2} 2^{k_2} & \theta_2 \\ 0 & (-1)^{\eta_2} \end{bmatrix} \nonumber\\ 
&& = \begin{bmatrix}  (-1)^{\epsilon_1 +\epsilon_2} 2^{k_1 + k_2} & (-1)^{\epsilon_1} 2^{k_1} \theta_1 + \theta_2 (-1)^{\eta_2} \\ 0 & (-1)^{\eta_1  + \eta_2} \end{bmatrix}, \nonumber
\end{eqnarray}
the group structure is given by  
\[ \Aut(A \otimes \mathcal K)/\overline{\Inn(A \otimes \mathcal K)} \cong (\mathbb Z/2\mathbb Z \times \mathbb Z) \ltimes \mathbb R \rtimes \mathbb Z/2\mathbb Z,\] 
where $\mathbb Z/2\mathbb Z \times \mathbb Z$ acts on $\mathbb R$ via the multiplicative subgroup $\{ (-1)^\epsilon 2^k \mid \epsilon \in \mathbb Z/2\mathbb Z, k \in \mathbb Z \}$ of $\mathbb R \setminus \{0\}$ and $\mathbb Z/2\mathbb Z$ by the multiplicative subgroup $\{1,-1\} \subseteq \mathbb R \setminus \{0\}$.
\end{example}

\section{Obstructions to asymptotic unitary equivalence}

If two automorphisms $\alpha, \beta \in \Aut(A)$ are asymptotically unitarily equivalent, then $\Kd(\alpha) = \Kd(\beta)$ by Proposition \ref{p:asu}. This can be used to find obstructions to asymptotic unitary equivalence of automorphisms in natural examples.

\begin{example}[The Bunce--Deddens algebra]\label{e:Bunce Deddens}
Consider the odometer act\-ion of $\mathbb Z$ on the Cantor set $X = \{0,1\}^{\mathbb N}$ implemented by addition with carry-over (see \cite{BunceDeddens} and \cite[VIII.4]{DavidsonBook}). The crossed product $A = C(X) \rtimes \mathbb Z$ is the Bunce--Deddens algebra and has the following invariants (e.g. \cite[Example 3.2.11]{RordamClassBook}): 
\[\begin{tikzcd}[row sep = 1em]
	{\K_0(A)} & {\T_{\mathbb R}(A)^\ast} & {\Kd(A)} & {\K_1(A)} \\
	{\mathbb Z[\frac{1}{2}]} & {\mathbb R} & {\mathbb R/\mathbb Z[\frac{1}{2}] \oplus \mathbb Z} & {\mathbb Z}
	\arrow["\rho_A", from=1-1, to=1-2]
	\arrow[hook, from=2-1, to=2-2]
	\arrow["{\Th_A^\Delta}", from=1-2, to=1-3]
	\arrow[from=1-3, to=1-4]
	\arrow[from=2-2, to=2-3]
	\arrow[from=2-3, to=2-4]
\end{tikzcd}\]
The Bernoulli measure on $X$ is the unique invariant probability measure, and induces the unique tracial state on $A$. The class of the unitary $u \in A$ representing the generator $1 \in \mathbb Z$ generates $\K_1(A)$. We identify $\Kd(A)$ with $\mathbb R/\mathbb Z[\frac{1}{2}] \oplus \mathbb Z$ via the splitting $n[u]_1 \mapsto n[u]_\Delta \colon \K_1(A) \to \Kd(A)$, so that $[e^{2 \pi i \theta} u^n]_\Delta$ is identified with $([\theta],n)$ for $\theta \in \mathbb R$ and $n \in \mathbb Z$.

For each $z \in \mathbb T$ there is a `gauge' automorphism $\gamma_z \colon A \to A$ which fixes $C(X)$ and satisfies $\gamma_z(u) = z u$.  The gauge automorphism $\gamma_z$ is trivial on the K-theory groups and on $\T_{\mathbb R}(A)^\ast$, but on $\Kd(A)$ we get
\begin{equation}\label{eq:K1DBD}
\Kd(\gamma_z) = \begin{bmatrix} \id & \phi_z \\ 0 & \id \end{bmatrix} \in \Aut(\mathbb R/\mathbb Z[\tfrac{1}{2}] \oplus \mathbb Z), 
\end{equation}
where $\phi_z \colon \mathbb Z \to \mathbb R / \mathbb Z [\frac{1}{2}]$ is the homomorphism given by $\phi_z(1) = \left[\tfrac{\log(z)}{2 \pi i}\right]$. 

We deduce that if $z, w \in \mathbb T$ do not differ by a dyadic rational viewed as elements of $\mathbb R / \mathbb Z$, then $\gamma_z$ is not asymptotically unitarily equivalent to $\gamma_w$. In fact, we deduce further that the stabilisations $\gamma_z \otimes \id, \gamma_w \otimes \id \in \Aut(A \otimes \mathcal K)$ are not asymptotically unitarily equivalent. Note that all the $\gamma_z$ are approximately unitarily equivalent (e.g.~by Theorem \ref{t:unital classification}).
\end{example}

\begin{example}[The odometer with a point removed]\label{e:odometer2}
Consider the odometer action of $\mathbb Z$ on $X = \{0,1\}^{\mathbb N}$ as above, but remove the point $x_0 = 000\cdots$ from $X$ to form a locally compact space $Y = X \setminus \{x_0\}$. The $\cs$-subalgebra $B = \overline{C_0(Y) A C_0(Y)}$ is a full hereditary subalgebra of $A$. Each gauge automorphism $\gamma_z \in \Aut(A)$ restricts to an automorphism $\gamma'_z \in \Aut(B)$. Since the full hereditary inclusion $B \hookrightarrow A$ induces an isomorphism $\Kd(B) \isom \Kd(A)$ by Proposition \ref{c:K1Dfullher}, we deduce further that $\gamma'_z \not \sim_{\mathrm{asu}} \gamma'_w$ when $z,w \in \mathbb T$ are not related by a dyadic rational. 
\end{example}

\begin{example}\label{e:irr}
Consider the irrational rotation algebra $A_\theta$ for $\theta \in \mathbb R \setminus \mathbb Q$, which is the universal $\cs$-algebra generated by unitaries $U,V$ satisfying $UV = e^{2 \pi i \theta} VU$ \cite{RieffelIRA}. This admits an action $\alpha$ of the torus $\mathbb T^2$ given at $(z,w) \in \mathbb T^2$ by 
\begin{eqnarray}
\alpha_{z,w}(U) = zU, & \alpha_{z,w}(V) = w V.
\end{eqnarray}
The irrational rotation algebra $A_\theta$ has invariants \cite{PV80}
\[\begin{tikzcd}[row sep = 1em, column sep = 3em]
	{\K_0(A_\theta)} & {\T_{\mathbb R}(A_\theta)^\ast} & {\Kd(A_\theta)} & {\K_1(A_\theta)} \\
	{\mathbb Z + \theta \mathbb Z} & {\mathbb R} & {\mathbb R/(\mathbb Z + \theta \mathbb Z) \oplus \mathbb Z^2} & {\mathbb Z^2},
	\arrow["\rho_{A_\theta}", from=1-1, to=1-2]
	\arrow[hook, from=2-1, to=2-2]
	\arrow["{\Th_{A_\theta}^\Delta}", from=1-2, to=1-3]
	\arrow[from=1-3, to=1-4]
	\arrow[from=2-2, to=2-3]
	\arrow[from=2-3, to=2-4]
\end{tikzcd}\]
with $\K_1(A_\theta) = \mathbb Z^2$ generated by $[U]_1$ and $[V]_1$. 
A similar analysis to the one in Example \ref{e:Bunce Deddens} reveals that $\alpha_{z,w}$ fixes $\K_0$, $\K_1$ and $\T_{\mathbb R}$, but on $\Kd$ gives
\begin{equation}\label{eq:K1Dirrational}
\Kd(\alpha_{z,w}) = \begin{bmatrix} 1 & \left[\tfrac{\log(z)}{2 \pi i}\right] & \left[\tfrac{\log(w)}{2 \pi i}\right] \\ 0 & 1 & 0 \\ 0 & 0 & 1 \end{bmatrix} \in \Aut(\mathbb R/(\mathbb Z + \theta \mathbb Z) \oplus \mathbb Z \oplus \mathbb Z).
\end{equation}
The values $\left[\tfrac{\log(z)}{2 \pi i}\right], \left[\tfrac{\log(w)}{2 \pi i}\right] \in \mathbb R/(\mathbb Z + \theta \mathbb Z)$ therefore distinguish the asymptotic unitary equivalence classes of the automorphisms $\{\alpha_{z,w} \mid z,w \in \mathbb T\}$ (and their stabilisations), which are all approximately unitarily equivalent (e.g.~by Theorem \ref{t:unital classification}).
\end{example}

\subsection{An Ext-obstruction}

We give another obstruction for asymptotic unitary equivalence related to $\Kd$ which to our knowledge first appeared in special cases in \cite{KishimotoKumjian01}, modifications of which have been applied to characterise asymptotic unitary equivalence in for instance \cite{LinNiu-Adv, Lin-Inv, GLN1}.

We note that this is a \emph{secondary} obstruction rather than a primary obstruction, that is, one can only use it to compare two $\ast$-homomorphisms, not as an invariant of a single $\ast$-homomorphism. Of course, if applied to automorphisms $\alpha$ of a $\cs$-algebra $A$, one can consider it as a primary invariant by comparing $\alpha$ and $\mathrm{id}_A$, which is how it appears in the work of Kishimoto and Kumjian \cite{KishimotoKumjian01}.

Let $\phi, \psi \colon A \to B$ be $\ast$-homomorphisms, and let 
\[
M_{\phi,\psi} := \{ (a,\mathbf b) \in A \oplus C([0,1], B) \mid \mathbf b(0) = \phi(a) , \, \mathbf b(1) = \psi(a)\}
\]
be the mapping torus of $\phi$ and $\psi$. This fits into a short exact sequence $0 \to SB \to M_{\phi, \psi} \to A \to 0$ in the obvious way, and standard techniques identify the index and exponential maps in the resulting six term exact sequence as follows (this is a variant of \cite[Exercise 12.5]{Rordam-Kbook}).
\begin{equation}\label{eq:mappingtorus}
\begin{tikzcd}
	{\K_0(B)} & {\K_1(M_{\phi,\psi})} & {\K_1(A)} \\
	{\K_1(A)} & {\K_0(M_{\phi,\psi})} & {\K_1(B)}
	\arrow[from=1-1, to=1-2]
	\arrow[from=1-2, to=1-3]
	\arrow["{\K_1(\phi) - \K_1(\psi)}", from=1-3, to=2-3]
	\arrow["{\K_0(\phi) - \K_0(\psi)}", from=2-1, to=1-1]
	\arrow[from=2-2, to=2-1]
	\arrow[from=2-3, to=2-2]
\end{tikzcd}
\end{equation}

\begin{lemma}\label{l:K one torus}
For $*$-homomorphisms $\phi,\psi \colon A \to B$ of $\cs$-algebras, consider the set 
\[
M_{\phi,\psi}^{\Ped} := \{ (a,\mathbf b) \in M_{\phi, \psi} \mid a\in \Ped(A), \, \mathbf b(s) \in \Ped(B) \textrm{ for all }s\in [0,1]\}. 
\]
Then $M_{\phi,\psi}^\Ped$ is a dense ideal in $M_{\phi,\psi}$ and hence $\Ped(M_{\phi, \psi}) \subseteq M_{\phi, \psi}^\Ped$ and
\[ \K_1(M_{\phi,\psi}) = \{ [(u,\mathbf v)]_1 \mid (u,\mathbf v) \in \mathcal U_n(M_{\phi,\psi}^\Ped), n \in \mathbb N \}. \]
Moreover, $(u_0,\mathbf v_0),(u_1,\mathbf v_1) \in \mathcal U_n(M_{\phi,\psi}^\Ped)$\footnote{Here we write $\mathcal U_n(M_{\phi,\psi}^{\Ped})$ for the unitaries $(u,\mathbf v) \in \mathcal U_n(M_{\phi,\psi})$ with $(1_n,\mathbf 1_n) - (u,\mathbf v) \in M_n(M_{\phi,\psi}^{\Ped})$.} define the same class in $\K_1$ if and only if for some $k \in \mathbb N$, the element $(u_0^\ast u_1, \mathbf v_0^\ast \mathbf v_1) \oplus (1_k, \mathbf 1_k)$ is a product of exponential unitaries of self-adjoints from $M_{n+k}(M_{\phi, \psi}^\Ped)$.
\end{lemma}
\begin{proof}
That $M_{\phi,\psi}^{\Ped}$ is an ideal is immediate. For density, consider $(a,\mathbf b) \in M_{\phi,\psi}$ and let $\epsilon > 0$. Find $a' \in \Ped(A)$ and $\mathbf b' \in \Ped(C([0,1],B))$ $\epsilon$-close to $a$ and $\mathbf b$ respectively. Note that $\mathbf b'(s) \in \Ped(B)$ for all $s\in [0,1]$. Let $\mathbf h \in C([0,1],B)$ be the straight line from $\phi(a') - \mathbf b'(0)$ to $\psi(a') - \mathbf b'(1)$. Then $(a',\mathbf b' + \mathbf h) \in M_{\phi,\psi}^\Ped$ is $3 \epsilon$-close to $(a,\mathbf b)$.

If $[(u_0,\mathbf v_0)]_1 = [(u_1,\mathbf v_1)]_1$, then there is some $k \in \mathbb N$ and a finite sequence $((u^{(j)}, \mathbf v^{(j)}))_{j=0}^N \subseteq  \mathcal U_{n+k}(M_{\phi,\psi})$ with $(u^{(0)}, \mathbf v^{(0)}) = (u_0,\mathbf v_0) \oplus (1_k, \mathbf 1_k)$ and $(u^{(N)}, \mathbf v^{(N)}) = (u_1,\mathbf v_1) \oplus (1_k, \mathbf 1_k)$ such that 
\[\| (u^{(j-1)}, \mathbf v^{(j-1)}) - (u^{(j)}, \mathbf v^{(j)}) \| < 2\] 
for each $1 \leq j \leq N$. By Lemma \ref{l:unitary} \eqref{l:unitary-1.5}, we may assume for each $1 \leq j \leq N-1$ that $(u^{(j)}, \mathbf v^{(j)}) \in \mathcal U_{n+k}(\Ped(M_{\phi,\psi})) \subseteq \mathcal U_{n+k}(M_{\phi,\psi}^\Ped)$. For each $1 \leq j \leq N$ the element $(a^{(j)},\mathbf b^{(j)}) = \tfrac{1}{2 \pi i} \log( (u^{(j-1)}, \mathbf v^{(j-1)})^\ast (u^{(j)}, \mathbf v^{(j)}))$ produced in Lemma \ref{l:unitary} \eqref{l:unitary-3} is in $M_{n+k}(M_{\phi,\psi}^\Ped)_{\mathrm{sa}}$ and satisfies 
\[(u^{(j)}, \mathbf v^{(j)}) = (u^{(j-1)}, \mathbf v^{(j-1)}) e^{2 \pi i (a^{(j)},\mathbf b^{(j)})}. \qedhere\]
\end{proof}

\begin{proposition}
Let $\phi,\psi \colon A \to B$ be $*$-homomorphisms between $\cs$-algebras. Then there is a (unique) homomorphism 
\[\Delta_{\phi,\psi} \colon \K_1(M_{\phi,\psi}) \to \frac{\T_{\mathbb R}(B)^\ast}{(\phi_\# - \psi_\#)(\T_{\mathbb R}(A)^\ast)} \]
such that $\Delta_{\phi,\psi}([(u,\mathbf v)]_1) = [\Det_B(\mathbf v)]$ for each $(u,\mathbf v) \in \mathcal U_n(M_{\phi,\psi}^\Ped)$. Moreover, $\Delta_{\phi,\psi}$ makes the following diagram commute.
\begin{equation}\label{eq:MTDelta}
\begin{tikzcd}
	{\K_0(B)} & {\K_1(M_{\phi,\psi})} \\
	{\T_{\mathbb R}(B)^\ast} & {\frac{\T_{\mathbb R}(B)^\ast}{(\phi_\# - \psi_\#)(\T_{\mathbb R}(A)^\ast)}}
	\arrow[from=1-1, to=1-2]
	\arrow["{\rho_B}", from=1-1, to=2-1]
	\arrow["{\Delta_{\phi,\psi}}", from=1-2, to=2-2]
	\arrow[two heads, from=2-1, to=2-2]
\end{tikzcd}
\end{equation}
\end{proposition}
\begin{proof}
For any self-adjoint $(a,\mathbf b) \in M_n(M_{\phi, \psi}^\Ped)$ we have by Lemma \ref{l:unitary} \eqref{l:unitary-7} that
\[
\Det_{\phi, \psi}(e^{(a,\mathbf b)}) = \widehat{\mathbf b(1)} - \widehat{\mathbf b(0)} = (\psi_\#-  \phi_\#)(\widehat a).
\]
Hence $\Delta_{\phi,\psi}$ is well-defined by Lemma \ref{l:K one torus}.

The homomorphism $\K_0(B) \to \K_1(M_{\phi,\psi})$ in \eqref{eq:mappingtorus} is defined via the Bott map $\K_0(B) \to \K_1(SB)$ and thus sends $[p]_0 - [q]_0$ for projections $p,q \in M_n(\Ped(B)^\dagger)$ with $p-q \in M_n(\Ped(B))$ to $[(0,\mathbf u_q^\ast \mathbf u_p)]_1$. Commutativity of \eqref{eq:MTDelta} therefore follows from Proposition \ref{p:pairing}.
\end{proof}

Note that \eqref{eq:MTDelta} induces a map $\ker \rho_B \to \ker \Delta_{\phi, \psi}$ whose image is contained in the kernel of the map $\K_1(M_{\phi,\psi}) \to \K_1(A)$ in \eqref{eq:mappingtorus}.

\begin{proposition}[The $\Ext$-obstruction]
Let $\phi, \psi \colon A \to B$ be $\ast$-homo\-morphisms between $\cs$-algebras. If $\K_0(\phi) = \K_0(\psi)$, $\T_{\mathbb R}(\phi) = \T_{\mathbb R}(\psi)$ and $\Kd(\phi) = \Kd(\psi)$ then the sequence
\begin{equation}\label{eq:Linobs}
0 \to \ker \rho_B \to \ker \Delta_{\phi, \psi} \to \K_1(A) \to 0
\end{equation}
induced by \eqref{eq:mappingtorus} and \eqref{eq:MTDelta} is exact, and if $\phi$ and $\psi$ are asymptotically unitarily equivalent, the induced class in $\Ext^1_{\mathbb Z}(\K_1(A), \ker \rho_B)$ vanishes.
\end{proposition}
\begin{proof}
Since $\K_0(\phi) = \K_0(\psi)$ the map $\ker \rho_B \to \ker \Delta_{\phi, \psi}$ is injective by exactness of \eqref{eq:mappingtorus}. That the composition $\ker \rho_B \to \K_1(A)$ is zero follows likewise from \eqref{eq:mappingtorus}. Suppose $x\in \ker \Delta_{\phi, \psi}$ vanishes in $\K_1(A)$. From exactness of \eqref{eq:mappingtorus} we may find $y\in \K_0(B)$ whose image in $\K_1(M_{\phi, \psi})$ is $x$. By \eqref{eq:MTDelta}, $\rho_B(y)$ is in the image of $\phi_\# - \psi_\#$, but this map is zero by duality since $\T_{\mathbb R}(\phi) = \T_{\mathbb R}(\psi)$. Hence $y\in \ker \rho_B$. 

Hence for exactness of \eqref{eq:Linobs} it remains to prove surjectivity of the map $\ker \Delta_{\phi, \psi} \to \K_1(A)$. Let $x\in \K_1(A)$ and pick $n \in \mathbb N$ and $u\in \mathcal U_n(\Ped(A))$ with $[u]_1 = x$. Since $\Kd(\phi) = \Kd(\psi)$ it follows that $\phi(u) \sim_{\Delta} \psi(u)$ and hence we may find $k \in \mathbb N$ and a continuous unitary path $\mathbf v \colon [0,1] \to \mathcal U_{n+k}(\Ped(B))$ such that $\mathbf v(0)= \phi(u) \oplus 1_k$, $\mathbf v(1) = \psi(u) \oplus 1_k$ and $\Det_B(\mathbf v) =0$. The induced class of $(u \oplus 1_k, \mathbf v) \in \mathcal U_{n+k}(M_{\phi, \psi}^\Ped)$ in $\K_1(M_{\phi, \psi})$ lifts $x$ and is in $\ker\Delta_{\phi,\psi}$. Hence, the sequence \eqref{eq:Linobs} is exact.

Suppose now that $\phi$ and $\psi$ are asymptotically unitarily equivalent and let $(u_t)_{t\in [0,1)}$ be a norm-continuous path in $\mathcal U(\mathcal M(B))$ such that $\Ad u_t \circ \phi \to \psi$ in the point-norm topology as $t\to 1$. As usual, we abuse notation and also write $\phi,\psi$ for the matrix amplified maps, and similarly $\Ad u_t$ for the matrix amplified automorphism on $M_n(B)$. Define a homomorphism
\[
\mathcal U_\infty(\Ped(A)) \to \mathcal U_\infty(M_{\phi, \psi}^\Ped), \quad v \mapsto (v, \mathbf w_v),
\]
where for $a \in M_n(\Ped(A)^\dagger)$ the element $\mathbf w_a$ is given by
\[
\mathbf w_a(t) = \Ad u_t (\phi(a)), \textrm{ for }t\in [0,1) , \qquad \mathbf w_a(1) = \psi(a).
\]
This defines a splitting $\K_1(A) \to \K_1(M_{\phi,\psi})$ because if $v = e^{2 \pi i a}$ for $a \in M_n(\Ped(A))_{\mathrm{sa}}$, then  $(v, \mathbf w_v) = e^{2\pi i (a,\mathbf w_a)}$.
Moreover, $\Det_{B}(\mathbf w_v) = 0$ for all $v\in \mathcal U_\infty(\Ped(A))$ by Proposition \ref{p:Detasym}, and so this splitting takes values in $\ker \Delta_{\phi, \psi}$.
\end{proof}

\section{$\Kd$ as an invariant for group actions}

Let $\alpha \colon G \to \Aut(A)$ be an action of a discrete group $G$ on a $\cs$-algebra $A$. Corollary \ref{c:direct sum} says that $\Kd(A)$ is recovered by K-theory and traces as 
\[\Kd(A) \cong \T_{\mathbb R}(A)^\ast / \imrho{A} \oplus \K_1(A).\]
This explains why $\Kd(A)$, and similarly $\HKd(A)$, are not needed as part of the invariant for classifying simple $\mathrm{C}^\ast$-algebras. However, this direct sum decomposition comes from a splitting which is not natural, thus the group $\Kd(A)$ need not decompose $G$-equi\-variantly as a direct sum $\T_{\mathbb R}(A)^\ast / \imrho{A} \oplus \K_1(A)$. Example \ref{e:Bunce Deddens} demonstrates this.

\begin{proposition}\label{prop:no eq split}
Pick $\theta \in \mathbb R \setminus \mathbb Z[\tfrac{1}{2}]$ and consider the action of $\mathbb Z$ on the Bunce--Deddens algebra $A$ via the automorphism $\gamma = \gamma_{e^{2 \pi i \theta}}$ as in Example \ref{e:Bunce Deddens}. Then there is no $\mathbb Z$-equivariant isomorphism between
\[ \Kd(A) \quad \textrm{and} \quad  \T_{\mathbb R}(A)^\ast/{\imrho{A}} \oplus \K_1(A). \]
\end{proposition}
\begin{proof}
The automorphism $\Kd(\gamma)$ is nontrivial by \eqref{eq:K1DBD}, while $\K_1(\gamma)$ and $\T_{\mathbb R}(\gamma)$ are both trivial.
\end{proof}
A similar analysis also applies to Examples \ref{e:odometer2} and \ref{e:irr}. We can go further than this; $\Kd$ can be used to find obstructions to cocycle conjugacy of group actions on $\cs$-algebras.

\begin{definition}
Let $G$ be a discrete group and let $A$ and $B$ be $\cs$-algebras. Actions $\alpha \colon G \to \Aut(A)$ and $\beta \colon G \to \Aut(B)$ are said to be \emph{cocycle conjugate} if there exists a $*$-isomorphism $\phi \colon A \to B$ and a map $\mathbf u \colon G \to \mathcal U(\multialg{B})$ such that $\mathbf u(1) = 1$ and $\mathbf u({gh}) = \mathbf u(g) \beta(g)(\mathbf u(h))$ for each $g,h \in G$ (i.e. $\mathbf u$ is a $\beta$-cocycle), and $\phi \circ \alpha(g) = \Ad (\mathbf u(g)) \circ \beta(g) \circ \phi$ for each $g \in G$. The actions $\alpha$ and $\beta$ are \emph{stably cocycle conjugate} if $\alpha \otimes \id \colon G \to \Aut(A \otimes \mathcal K)$ and $\beta \otimes \id \colon G \to \Aut(B \otimes \mathcal K)$ are cocycle conjugate.
\end{definition}

For an action $\alpha \colon G \to \Aut(A)$, we write $\Kd(A,\alpha)$ for the structure of $\Kd(A)$ equipped additionally with its $G$-module structure from $\alpha$,  and similarly $\T_{\mathbb R}(A,\alpha)^\ast$, $\K_0(A,\alpha)$ and $\K_1(A,\alpha)$.
It is immediate from unitary equivalence invariance that
a cocycle conjugacy $(\phi,\mathbf u) \colon (A,\alpha) \to (B,\beta)$ induces a $G$-equivariant isomorphism $\Kd(\phi) \colon \Kd(A,\alpha) \isom \Kd(B,\beta)$. This applies equally to $\phi_\#$ and $\K_i(\phi)$. In particular, if $\alpha$ and $\beta$ are stably cocycle conjugate, there are $G$-equivariant isomorphisms
\[\begin{tikzcd}
	{\K_0(A,\alpha)} & {\mathrm{T}_{\mathbb R}(A,\alpha)^\ast} & {\Kd(A,\alpha)} & {\K_1(A,\alpha)} \\
	{\K_0(B,\beta)} & {\mathrm{T}_{\mathbb R}(B,\beta)^\ast} & {\Kd(B,\beta)} & {\K_1(B,\beta)}
	\arrow["{\rho_A}", from=1-1, to=1-2]
	\arrow["\cong", from=1-1, to=2-1]
	\arrow["{\Thd_A}", from=1-2, to=1-3]
	\arrow["\cong", from=1-2, to=2-2]
	\arrow[from=1-3, to=1-4]
	\arrow["\cong", from=1-3, to=2-3]
	\arrow["\cong", from=1-4, to=2-4]
	\arrow["{\rho_B}", from=2-1, to=2-2]
	\arrow["{\Thd_B}", from=2-2, to=2-3]
	\arrow[from=2-3, to=2-4]
\end{tikzcd}\]
making the diagram commute.

Recall that an action $\alpha \colon G \to \Aut(A)$ is \emph{strongly outer} if, for each $g \in G \setminus \{1\}$ and each positive trace $\tau \in \T_{+}(A)$, the automorphism induced by $\alpha(g)$ on $\pi_\tau(A)''$ is outer. This is a standard hypothesis (see e.g. \cite{MatuiSatoStronglyOuter, MatuiSatoStronglyOuterII}) when considering the classification of group actions up to cocycle conjugacy in the tracial setting.

\begin{theorem}
Let $A = A_\theta$ be the irrational rotation algebra for $\theta \in \mathbb R \setminus \mathbb Q$. 
Let $(s_i,t_i) \in \mathbb R^2 \setminus (\mathbb Q + \theta \mathbb Q)^2$ for $i=0,1$. Set $\alpha_i \colon \mathbb Z \to \Aut(A)$ to be the action generated by the automorphism $\alpha_{z_i,w_i}$ with $z_i = e^{2 \pi i s_i}$ and $w_i = e^{2 \pi i t_i}$ as in Example \ref{e:irr}. 

Then $\alpha_0$ and $\alpha_1$ are strongly outer, and there exists a $\mathrm{KK}^{\mathbb Z}$-equivalence 
\[x\in \mathrm{KK}^{\mathbb Z}((A, \alpha_0), (A, \alpha_1))\] 
such that $\K_0(x)([1_A]_0) = [1_A]_0$ and $\rho_A \circ \K_0(x) = \rho_A$. However, if $\theta$ is not an algebraic number of degree $2$, and
\[\begin{bmatrix} [s_0] \\ [t_0] \end{bmatrix}, \begin{bmatrix} [s_1] \\ [t_1] \end{bmatrix} \in \mathbb R/(\mathbb Z + \theta \mathbb Z) \oplus \mathbb R/(\mathbb Z + \theta \mathbb Z)\]
are in distinct orbits under the action of $\mathrm{GL}_2(\mathbb Z)$ by matrix multiplication, then $(\T_{\mathbb R}(A,\alpha_0)^\ast \xrightarrow{\Thd_A} \Kd(A,\alpha_0)) \not \cong (\T_{\mathbb R}(A,\alpha_1)^\ast \xrightarrow{\Thd_A} \Kd(A,\alpha_1))$. In particular, $\alpha_0$ and $\alpha_1$ are not (stably) cocycle conjugate. 
\end{theorem}
\begin{proof}
Strong outerness of $\alpha_i$ for $i = 0,1$ follows from \cite[Lemma 3.2]{OP06} applied to $\alpha_i^n$, $n \in \mathbb Z \setminus \{0\}$, using that one of $ns_i$ and $nt_i$ is not in $\mathbb Z + \theta \mathbb Z$.  

For $i=0,1$ the KK-class $[\alpha_i] \in \mathrm{KK}(A,A)$ must be $[\id_A]$ because each $\alpha_i$ is homotopic to $\id_A$. Direct application of \cite[Theorem 2.12]{MeyerEqKKClass} yields a $\mathrm{KK}^{\mathbb Z}$-equivalence $x \in \mathrm{KK}^{\mathbb Z}((A, \alpha_0), (A, \alpha_1))$ which maps to $\id_A \in \mathrm{KK}(A,A)$. Thus $\K_\ast(x) =  \id_{\K_\ast (A)}$ trivially satisfies $\K_0(x)([1_A]_0) = [1_A]_0$ and $\rho_A \circ \K_0(x) = \rho_A$.

Suppose now that $\theta$ is not algebraic of degree $2$ and that there are automorphisms $\beta \in \Aut(\Kd(A))$ and $\gamma \in \Aut(\T_{\mathbb R}(A)^\ast)$ such that $\beta \circ \Kd(\alpha_{z_0,w_0}) = \Kd(\alpha_{z_1,w_1}) \circ \beta$ and $\beta \circ \Thd_A = \Thd_A \circ \gamma$ (note that $(\alpha_{z_0,w_0})_\# = (\alpha_{z_1,w_1})_\# = \mathrm{id}$ so $\gamma$ intertwines these). Then $\gamma \in \Aut(\mathbb R)$ is order-preserving and hence given by multiplication by some fixed $\lambda > 0$. From $\beta \circ \Thd_A = \Thd_A \circ \gamma$ it follows that $\lambda(\mathbb Z + \theta \mathbb Z) = (\mathbb Z + \theta \mathbb Z)$. Because $\theta$ is not algebraic of degree $2$ we must have $\lambda = 1$. As $\mathbb R$ is divisible there are no non-zero homomorphisms $\mathbb R/(\mathbb Z + \theta \mathbb Z) \to \mathbb Z$, so $\beta$ must be of the form 
\[ \beta = \begin{bmatrix}  1 & k & \ell \\ 0 & a & b \\ 0 & c & d \end{bmatrix} \in \Aut(\mathbb R/(\mathbb Z + \theta \mathbb Z) \oplus \mathbb Z \oplus \mathbb Z)  \]
for some $\begin{bmatrix} a & b \\  c & d \end{bmatrix} \in \mathrm{GL}_2(\mathbb Z)$ and $k, \ell \in \mathbb R/(\mathbb Z + \theta \mathbb Z)$. 
Plugging the description \eqref{eq:K1Dirrational} of $\Kd(\alpha_{z_0,w_0})$ and $\Kd(\alpha_{z_1,w_1})$ into the equation $\beta  \circ (\Kd(\alpha_{z_0,w_0})- \id) = (\Kd(\alpha_{z_1,w_1}) - \id) \circ \beta$ yields
\[ \begin{bmatrix} [s_0] \\ [t_0] \end{bmatrix} = \begin{bmatrix} a & c \\  b & d \end{bmatrix} \begin{bmatrix} [s_1] \\ [t_1] \end{bmatrix} \in \mathbb R/(\mathbb Z + \theta \mathbb Z) \oplus \mathbb R/(\mathbb Z + \theta \mathbb Z). \qedhere \]
\end{proof}

\newcommand{\etalchar}[1]{$^{#1}$}

\end{document}